\documentclass [11pt]{amsart}
\usepackage[utf8]{inputenc}
\pdfoutput=1

\usepackage{amsmath,amssymb,amsthm,amsfonts}
\usepackage{mathtools}%

\usepackage[main=english,french]{babel}%
\usepackage[bookmarksdepth=4]{hyperref}%
\usepackage{bbm}%
\usepackage{mathrsfs}%

\usepackage{tikz,graphicx,color}
\usepackage{tikz-cd}%
\usepackage{tikz-3dplot}%

\usepackage[subtle,tracking=normal,mathdisplays=tight]{savetrees} %

\usetikzlibrary{calc}%
\usetikzlibrary{arrows}%
\usetikzlibrary{shapes}%
\usetikzlibrary{patterns}%
\usetikzlibrary{positioning}%
\usetikzlibrary{arrows.meta}
\usetikzlibrary{decorations.markings}
\usepackage{epstopdf}%
\pdfpageattr{/Group <</S /Transparency /I true /CS /DeviceRGB>>}

\usepackage[arrow]{xy}%
\usepackage{diagbox}%
\usepackage{subfig}%
\usepackage{arcs}%
\usepackage{xcolor}%
\usepackage{tabu}
\usepackage{booktabs}%

\usepackage[margin=1in]{geometry}

\usepackage{soul}
\usepackage{accents}
\usepackage{enumitem}

\usepackage{letltxmacro}
\usepackage{thmtools,etoolbox}

\usepackage{multirow}

\usepackage{nameref,hyperref}
\usepackage[capitalize]{cleveref}

\usepackage{tikz-cd}
\usepackage{stmaryrd}

\usepackage{xspace}

\usepackage{stmaryrd}

\usepackage{bm}
\usepackage{crossreftools}

\usepackage{pdflscape}

\newtheorem{theorem}{Theorem}[section]

\newtheorem{lemma}[theorem]{Lemma}
\newtheorem{proposition}[theorem]{Proposition}
\newtheorem{prop}[theorem]{Proposition}
\newtheorem{corollary}[theorem]{Corollary}

\theoremstyle{definition}
\newtheorem{propdef}[theorem]{Proposition-Definition}
\theoremstyle{definition}
\newtheorem{remark}[theorem]{Remark}
\theoremstyle{definition}
\newtheorem{definition}[theorem]{Definition}

\theoremstyle{definition}

\theoremstyle{definition}

\theoremstyle{definition}

\theoremstyle{definition}
\newtheorem{assumption}[theorem]{Assumption}
\theoremstyle{definition}

\crefname{figure}{Figure}{Figures}

\addtotheorempostheadhook[theorem]{\crefalias{theoremlisti}{theorem}}
\addtotheorempostheadhook[lemma]{\crefalias{theoremlisti}{lemma}}
\addtotheorempostheadhook[proposition]{\crefalias{theoremlisti}{proposition}}
\addtotheorempostheadhook[corollary]{\crefalias{theoremlisti}{corollary}}

\def\Acal{\mathcal{A}}\def\Ocal{\mathcal{O}}\def\Ucal{\mathcal{U}}\def\Vcal{\mathcal{V}}\def\Xcal{\mathcal{X}}

\def\one{{\mathbbm{1}}}
\def\C{\mathbb{C}}
\def\R{\mathbb{R}}

\def\Z{\mathbb{Z}}
\def\Q{\mathbb{Q}}

\def\<{{\langle}}
\def\>{{\rangle}}

\def\la{{\lambda}}

\def\diag{{ \operatorname{diag}}}

\def\CC{{\mathbb C}}

\def\SL{\operatorname{SL}}

\def\Z{{\mathbb Z}}
\def\R{{\mathbb R}}

\def\diag{{\rm diag}}

\def\id{{\operatorname{id}}}

\newcounter{todobackgr}[section]

\newcounter{todofigure}[section]

\def\ub{{\bar u}}

\def\d#1{\dot{#1}}
\def\ds{\dot{s}}
\def\dw{\dot{w}}
\def\du{\dot{u}}
\def\dv{\dot{v}}
\def\da{\dot{a}}

\def\alphacheck{\alpha^\vee}

\makeatletter
\DeclareRobustCommand{\cev}[1]{%
  \mathpalette\do@cev{#1}%
}
\newcommand{\do@cev}[2]{%
  \fix@cev{#1}{+}%
  \reflectbox{$\m@th#1\vec{\reflectbox{$\fix@cev{#1}{-}\m@th#1#2\fix@cev{#1}{+}$}}$}%
  \fix@cev{#1}{-}%
}
\newcommand{\fix@cev}[2]{%
  \ifx#1\displaystyle
    \mkern#23mu
  \else
    \ifx#1\textstyle
      \mkern#23mu
    \else
      \ifx#1\scriptstyle
        \mkern#22mu
      \else
        \mkern#22mu
      \fi
    \fi
  \fi
}

\makeatother

\def\Rich_#1^#2{\accentset{\circ}{R}_{#1,#2}}
\def\Rtp_#1^#2{R_{#1,#2}^{>0}}

\def\Rtnn_#1^#2{R_{#1,#2}^{\geq0}}

\def\PR_#1^#2{\accentset{\circ}{\Pi}_{#1,#2}}%
\def\PRtp_#1^#2{\Pi_{#1,#2}^{>0}}%
\def\PRtnn_#1^#2{\Pi_{#1,#2}^{\geq0}}%
\def\PRcl_#1^#2{\Pi_{#1,#2}}%
\def\PRR_#1^#2{\accentset{\circ}{\Pi}_{#1,#2}^\R}%
\def\PRRcl_#1^#2{\Pi_{#1,#2}^\R}%

\def\bt{{\mathbf{t}}}
\def\bw{{\mathbf{w}}}
\def\bv{{\mathbf{v}}}
\def\bu{{\mathbf{u}}}

\def\bx{{\mathbf{x}}}
\def\by{{\mathbf{y}}}

\def\Hom{\operatorname{Hom}}

\def\BR{B(\R)}

\def\hjmap{\kappa}
\def\hjmp_#1{\hjmap_{#1}}

\def\Uom_#1{U^{\diamond,-}_{#1}}

\def\xrasim{\xrightarrow{\sim}}

\def\Richaff_#1^#2{\accentset{\circ}{\mathcal{R}}_{#1}^{#2}}

\def\Cast{\C^\times}

\def\Povar_#1{\accentset{\circ}{\Pi}_{#1}}
\def\Povarcl_#1{\Pi_{#1}}
\def\RPovar_#1{\accentset{\circ}{\Pi}^\R_{#1}}
\def\RPovarcl_#1{\Pi^\R_{#1}}
\def\Povtp_#1^#2{\Pi_{#1,#2}^{>0}}
\def\Povtnn_#1{\Pi_{#1}^{\geq0}}

\def\Star_#1{\operatorname{Star}_{#1}}
\def\Startnn_#1{\operatorname{Star}^{\geq0}_{#1}}

\def\Link{\operatorname{Lk}}
\def\Lkx_#1{\Link_{#1}}
\def\Lkxx_#1^#2{\accentset{\circ}{\Link}_{#1}^{#2}}

\def\Lktxx_#1^#2{\Link^{>0}_{#1,#2}}
\def\Starxx_#1^#2{\operatorname{Star}_{#1,#2}}
\def\Startxx_#1^#2{\operatorname{Star}^{\geq0}_{#1,#2}}

\def\sctnn_#1{\sc^{\geq0}_{#1}}

\def\sctp_#1^#2{\sc^{>0}_{#1,#2}}

\def\Seps_#1{S_{#1}}

\def\Lktpe_#1^#2{\Link^{>0}_{#1,#2}}
\def\Lktnne_#1{\Link^{\geq0}_{#1}}

\def\Lktp_#1^#2{\Link^{>0}_{#1,#2}}
\def\Lktnn_#1{\Link^{\geq0}_{#1}}

\def\sc{Z}

\def\sco_#1^#2{\accentset{\circ}{\sc}_{#1,#2}}

\def\sccl_#1^#2{\sc_{#1}^{#2}}

\def\Y{\mathcal{Y}}
\def\Yo_#1{\accentset{\circ}{\Y}_{#1}}
\def\Ycl_#1{\Y_{#1}}

\def\Ytp_#1{\Y_{#1}^{>0}}

\def\strg(#1){\normg{#1}}

\def\normg#1{\|#1\|}

\def\Spec{\operatorname{Spec}}
\def\int{{\operatorname{init}}}

\def\domk{\Delta_{\omega_k,\omega_k}}

\def\DOM^#1_#2{\Delta^{#1 \omega_i}_{#2 \omega_i}}
\def\DOMr^#1_#2{\Delta^{#1 \omega_r}_{#2 \omega_r}}
\def\DOMir^#1_#2{\Delta^{#1 \omega_{i_r}}_{#2 \omega_{i_r}}}
\def\om{\omega}

\def\line#1{\overline{#1}}
\def\lline#1{\overline{\overline{#1}}}

\def\sv{s^{\mathbf{v}}}

\newcommand*{\smallcap}{{\mathbin{\scalebox{0.5}{\ensuremath{\cap}}}}}%

\def\RLtp_#1^#2{\cev R_{#1,#2}^{>0}}
\def\Rsf_#1^#2{\Rich_{#1}^{#2}(K)}
\def\LRsf_#1^#2{\LRich_{#1}^{#2}(K)}

\def\pre{{\,\operatorname{pre}}}
\def\tpre_#1^#2{\vec\twistop^\pre_{#1,#2}}
\def\Ltpre_#1^#2{\cev\twistop^\pre_{#1,#2}}
\def\tpreL_#1^#2{\Ltpre_{#1}^{#2}}
\def\twistop{\tau}
\def\twist_#1^#2{\vec\twistop_{#1,#2}}
\def\twistL_#1^#2{\cev\twistop_{#1,#2}}

\def\om{\omega}
\def\sv_#1{s^{\bv}_{#1}}
\def\capBWB(#1){(#1)^{\smallcap w}}
\def\capBWoB(#1){(#1)^{\smallcap w_0}}
\def\capBWBnopar(#1){#1^{\smallcap w}}

\def\Pio_#1^#2{\Pi^\circ_{#1,#2}}
\def\UW{\wedge^\arrR}

\def\chmnrR_#1{\vec\Delta_{#1}}
\def\chmnrL_#1{\cev\Delta_{#1}}

\def\Yo{y_0}

\def\H{H}

\def\Dpm_#1{\Delta^{\pm}_{#1}}
\def\Dmp_#1{\Delta^{\mp}_{#1}}

\def\ach{\alphacheck}

\def\MS{\operatorname{MS}}
\def\TMSR_#1^#2{\vec \tau^{\,\MS}_{#1,#2}}
\def\TMSL_#1^#2{\cev \tau^{\,\MS}_{#1,#2}}

\def\Vlax_#1{V(\la)_{#1}}

\def\rev{\operatorname{rev}}

\def\A{{\mathcal A}}
\def\C{{\mathbb C}}
\def\Z{{\mathbb Z}}

\def\x{{\mathbf{x}}}
\def\a{{\mathbf{a}}}

\def\Y{{\overline{Y}}}
\def\tB{{\tilde B}}
\def\ord{{\mathrm{ord}}}

\def\dlog{\operatorname{dlog}}
\def\relspc{\ }
\newcommand{\Rrel}[1]{\stackrel{#1\relspc}{\longrightarrow}} 
\newcommand{\Rwrel}[1]{\stackrel{#1\relspc}{\Longrightarrow}}
\newcommand{\Lrel}[1]{ \stackrel{\relspc#1}{\longleftarrow}}
\newcommand{\Lwrel}[1]{\stackrel{\relspc#1}{\Longleftarrow}}

\def\CV(#1){x_{#1}}

\def\BR{\accentset{\circ}{R}}

\def\pmn{\pm I}
\def\pmnm{(\pmn)^m} %
\def\DRW{(\pmn)^m} %
\def\Demprod{\delta}

\def\br{\beta}  %
\def\bigBR{\accentset{\circ}{\mathcal{Y}}} %
\def\BigBR{\mathcal{Y}} %
\def\torus{T}
\def\torus(#1){T_{#1}}
\def\Cx{{\C^\times}}
\newcommand{\apd}[1]{^{\< #1\>}}
\newcommand{\crossing}[1]{\Delta_{#1}}

\def\grid_#1^#2{\Delta_{#1}^{#2}}
\def\grid_#1{\Delta_{#1}}
\def\z{\bar{z}}

\def\xbr{\x_{\br}}

\def\Abr{\A_{\br}}

\makeatletter
\newcommand*\bigcdot{\mathpalette\bigcdot@{.4}}
\newcommand*\bigcdot@[2]{\mathbin{\vcenter{\hbox{\scalebox{#2}{$\m@th#1\bullet$}}}}}
\makeatother

\def\fro{\operatorname{fro}}
\def\mut{\operatorname{mut}}
\def\Jfro{J_{\bu}^{\fro}}
\def\Jmut{J_{\bu}^{\mut}}

\def\Hollow{[m]\setminus\Jo}

\def\fro{\operatorname{fro}}
\def\mut{\operatorname{mut}}

\def\Jfro{J_{\bu}^{\fro}}
\def\Jmut{J_{\bu}^{\mut}}

\numberwithin{equation}{section}

\def\todoflr#1(#2){\textcolor{green!70!black}{\bf \{#1\}(#2)}\xspace}

\def\brp{\br'}

\def\wo{w_\circ}

\def\sinkrec{sink-recurrent\xspace}

\def\top{{\operatorname{top}}}
\def\Ptop_#1(#2){P^\top(#1;#2)}

\def\igrw[#1]#2{\includegraphics[width=#1\textwidth]{#2}}

\def\figref#1(#2){Figure~\hyperref[#1]{\ref*{#1}(#2)}}
\def\tabref#1(#2){Table~\hyperref[#1]{\ref*{#1}(#2)}}

\def\xbr{\x_\br}

\def\UU{W}

\def\Jo{J_\br}

\def\Poincare{Poincar\'e\xspace}

\def\Jfro{\Jo^{\fro}}
\def\Jmut{\Jo^{\mut}}

\def\lemmaref#1(#2){Lemma~\hyperref[#1]{\ref*{#1}(#2)}}

\def\Chmin(#1){\crossing{#1}}

\def\ds{\dot{s}}

\def\Xbul{X_{\bullet}}
\def\Ybul{Y_{\bullet}}
\def\Bruh{\Xcal}
\def\Bruho{\accentset{\circ}{\Xcal}}
\def\cham{chamber\xspace}

\def\sign{{\rm sign}}
\def\mut{{\rm mut}}
\def\fro{{\rm fro}}
\def\d{\mathbf{d}}
\def\tpar{t}

\def\UF#1{\dot{#1}} %
\def\UG{\UF{G}}
\def\UB{\UF{B}}
\def\UU{\UF{U}}
\def\UH{\UF{\H}}
\def\ux{\UF{x}}
\def\uy{\UF{y}}
\def\UI{\UF{I}}
\def\UW{\UF{W}}
\def\ui{{i'}}
\def\uj{{j'}}
\def\sig{\sigma}

\def\orb[#1]{\bm{#1}}
\def\us{{\tilde s}}

\def\uwo{{\tilde w}_0}
\def\bru{{\tilde\br}}
\def\um{\tilde{m}}

\def\relab{\la_\br}

\def\chL#1{X^\ast(#1)}
\def\cochL#1{X_\ast(#1)}

\def\uom{\UF{\om}}
\def\uach{\UF{\alpha}^\vee}
\def\ualpha{\UF{\alpha}}
\def\uup#1{\tilde{u}_{#1}}
\def\wwp#1{\tilde{w}_{#1}}
\def\ugrid_#1{\Delta'_{#1}}
\def\UXbul{\UF{X}_{\bullet}}
\def\UYbul{\UF{Y}_{\bullet}}
\def\UX{\UF{X}}
\def\UY{\UF{Y}}
\def\UJ{\UF{J}}
\def\UJo{\UF{J}_{\bru}}
\def\UJmut{\UF{J}^{\mut}}
\def\UJfro{\UF{J}^{\fro}}
\def\ucrossing#1{\tilde{\Delta}_{#1}}

\def\Lci#1#2{L_{#1,#2}}
\def\Lpci#1#2{L'_{#1,#2}}
\def\omp{\om'}

\def\Tbr{T_{\br}}

\def\eps{\epsilon}
\def\ombrc{\om_{\br,c}}
\def\ombrcp{\om_{\br,\cp}}
\def\ombrpc{\om_{\br',c}}
\def\ombrx#1{\om_{\br,#1}}
\def\ombrpx#1{\om_{\br',#1}}
\def\ombr{\om_{\br}}
\def\ombrp{\om_{\br'}}

\def\omcoef{\tB}
\def\Sbr{\Sigma_{\br}}
\def\Sbrp{\Sigma_{\br'}}
\def\ifreeze#1[#2]{#1^{\backslash #2}} %
\def\icon#1[#2]{#1^{\slash #2}} %

\def\bmref#1{\ref{#1}\xspace}
\def\Acalbr{\Acal(\Sbr)}

\def\brl{\tilde{\br}}
\def\ml{\tilde{\mis}}
\def\tSig{\dot{\Sigma}}
\def\tSigma{\tSig}
\DeclareMathOperator{\braid}{braid}
\DeclareMathOperator{\lift}{lift}
\DeclareMathOperator{\fold}{fold}
\DeclareMathOperator{\res}{res}
\def\frQ{\dot{Q}^{\operatorname{fr}}}
\def\line#1{\overline{#1}}
\def\lline#1{\overline{\overline{#1}}}
\def\hr{h^{+}}
\def\hb{h^{-}}
\def\hrb{h^{\pm}}

\def\UGU{U_+ \backslash G / U_+}
\def\tV{\tilde V}
\def\tTbr{\tilde{T}_\br}

\def\tparp{t'}
\def\btparp{\bt'}
\def\ordM{M_\br}

\def\ombrp{\om_{\br'}}
\def\ombrpc{\om_{\br',c}}

\def\tzero{t^*}

\def\bvap{\bu}
\def\short{short\xspace}
\def\llong{long\xspace}
\def\npm{{n+m}}

\def\pb{^\ast} %
\def\mis{\phi} %
\def\cp{e}

\hypersetup{bookmarksdepth = 2}
\def\tbeta{{\tilde \beta}}
\def\nrmfont{\textup}
\def\Digr{\Gamma}
\def\Digrt{\tilde\Gamma}
\def\sinkrec{sink-recurrent\xspace}

\def\ND{N^{\operatorname{in}}_s(\Digr)}
\def\NDSH{\widehat{N}^{\operatorname{in}}_s(\Digr(\Sigma))}
\def\NDH{\widehat{N}^{\operatorname{in}}_s(\Digr)}
\def\froi{f} %
\def\si{s} %

\def\orbit#1{#1}
\def\orbSet{J}

\def \fseed#1{#1} %
\def \fmtx#1{#1}
\def \midbr{\delta}
\def \midbrl{\tilde{\midbr}}
\def \quasi{\sim}
\def\dwo{\dot{w}_\circ}
\def\ris{r}
\def\risW{r}
\def\risV{r}

\def\USigma{\UF{\Sigma}}
\def\UT{\UF{T}}
\def\ud{\UF{\d}}
\def\Sheta{\sig}
\def\sheta{\sig}

\def\ua{a'}
\def\ub{b'}
\def\ucartan{\UF{a}}
\def\utB{\dot\tB}
\def\JJmut{J^{\mut}}
\def\JJfro{J^{\fro}}

\def\uk{k'}
\def\UQ{\UF{Q}}
\def\bz{\mathbf{z}}

\def\negspc{\\[-11pt]}
\def\uni{\underline{i}}
\def\unj{\underline{j}}
\def\uii{i''}
\def\orbC{{\orb[C]}}
\def\orbCp{{\orb[C']}}
\def\uc{{c'}}

\def\one{1\,}
\def\uone{\underline{1}\,}
\def\two{2\,}
\def\utwo{\underline{2}\,}
\def\twox{2}
\def\utwox{\underline{2}}

\def\fseedx{\iota^\ast}
\def\ptil{\tilde p}

\def\mtil{\tilde m}
\def\ubx{\UF{\bx}}
\def\uby{\UF{\by}}
\def\Qres{\UQ_{\res}}
\def\frQres{\frQ_{\res}}
\def\ucp{{\cp'}}
\def\cpp{{\cp'}}

\def\uwo{{\tilde w}_\circ}

\def\uisub_#1{{i'_{#1}}}

\def\tVo{\tilde{V}^\circ}
\def\tVgeq_#1{\tilde{V}_{\geq #1}}
\def\Vo{V^\circ}
\def\btgeq_#1{\btparp_{> #1}}
\def\ombrest{\omega_{\operatorname{rest}}}
\def\ombrestp{\omega'_{\operatorname{rest}}}
\def\Iset{F}
\def\Vproj{p}
\def\perm{\tau}
\def\ordbr{\ord^{(\br)}}
\def\ordbrp{\ord^{(\brp)}}
\def\dbr{\d_{\br}}
\begin{document}

\def\ups#1{u_{#1}}
\def\wp#1{w_{#1}}
\def\dwp#1{\dw_{#1}}
\def\vp#1{u_{#1}} 
\def\pd#1{_{#1}}
\def\pu#1{^{#1}}
\def\up#1{u_{#1}}

\title{Braid variety cluster structures, II: general type}

\author{Pavel Galashin}
\address{Department of Mathematics, University of California, Los Angeles, 520 Portola Plaza,
Los Angeles, CA 90025, USA}
\email{\href{mailto:galashin@math.ucla.edu}{galashin@math.ucla.edu}}

\author{Thomas Lam}
\address{Department of Mathematics, University of Michigan, 2074 East Hall, 530 Church Street, Ann Arbor, MI 48109-1043, USA}
\email{\href{mailto:tfylam@umich.edu}{tfylam@umich.edu}}

\author{Melissa Sherman-Bennett}
\address{Department of Mathematics, University of California Davis, One Shields Ave, Davis, CA 95616}
\email{\href{mailto:mshermanbennett@ucdavis.edu}{mshermanbennett@ucdavis.edu}}

\thanks{P.G.\ was supported by an Alfred P. Sloan Research Fellowship and by the National Science Foundation under Grants No.~DMS-1954121 and No.~DMS-2046915. T.L.\ was supported by the National Science Foundation under Grants No.~DMS-1953852 and No.~DMS-2348799. M.S.B. was supported by the National Science Foundation under Award No.~DMS-2103282 and No.DMS-2444020. 
 Any opinions, findings, and conclusions or recommendations expressed in this material are
	those of the authors and do not necessarily reflect the views of the National Science
	Foundation.}

\date\today

\makeatletter
\@namedef{subjclassname@2020}{%
  \textup{2020} Mathematics Subject Classification}
\makeatother



\begin{abstract}
We show that braid varieties for any complex simple algebraic group $G$ are cluster varieties. This includes open Richardson varieties inside the flag variety $G/B$.
\end{abstract}
\maketitle

\section{Introduction}\label{sec:intro}

This is one of two papers concerned with the construction of cluster structures on braid varieties.  In the present paper, we work in the setting of a general simple algebraic group $G$ and construct cluster seeds using algebraic  geometry.  In the companion paper~\cite{GLSBS1}, joint with David Speyer, we give an alternative proof in the special case $G = \SL_n$, using the combinatorics of plabic graphs and surfaces.  The current work is logically independent of~\cite{GLSBS1}, which, however, ultimately produces the same cluster structure in the case $G = \SL_n$. 

Let $G$ be a complex, simple, simply-connected algebraic group, $B_{\pm}$ opposing Borel subgroups, $U_{\pm}$ their unipotent radicals, $\H:= B_+ \cap B_-$ the torus, $I$ the vertex set of the Dynkin diagram, $W$ the Weyl group with simple generators $s_i, i \in I$, and denote by $\dot w$ the lift of $w\in W$ to $G$ as in~\eqref{eq:dw_dfn}.   Let $\wo \in W$ denote the longest element and $i \mapsto i^*$ the action of $\wo$ on $I$.  Let $\alpha_i, \alpha_i^\vee, \omega_i$ for $i \in I$ denote the simple roots, simple coroots, fundamental weights, respectively, and let $A = (a_{ij})_{i,j\in I}$ be the Cartan matrix given by $a_{ij}:= \< \alpha_i, \ach_j \>$.  Denote $d_i := 2/(\alpha_i,\alpha_i)$ so that $d_i a_{ij} = d_j a_{ji}$.

\subsection{Double braid varieties}\label{sec:intro:double}
A \emph{double braid word} $\br = i_1 i_2 \cdots i_m$ is a word in the alphabet $\pm I$.  For $i \in I$, we set $(-i)^*:=-i^*$. For $i\in\pmn$, define
\begin{equation}\label{eq:s_i-notation}%
 s_i^+:=
  \begin{cases}
    s_i, &\text{if }i>0,\\
    \id, &\text{if }i<0,
  \end{cases}\quad   s_i^-:=
  \begin{cases}
    \id, &\text{if }i>0, \\
    s_{-i}, &\text{if }i<0.\\
  \end{cases}
\end{equation}

A \emph{weighted} (or \emph{framed}) \emph{flag} is an element $F = gU_+ \in G/U_+$.  %
Two weighted flags $(F,F')$ are \emph{weakly $w$-related} (resp., \emph{strictly $w$-related}) if there exist $g \in G$ and $h\in H$ (resp., $g \in G$) such that $(gF,gF') = (U_+, h\dot{w}U_+)$ (resp., $(gF,gF') = (U_+, \dot{w}U_+)$).  We write this as $F \Rwrel{w} F'$ (resp., $F \Rrel{w} F'$).

Suppose that the Demazure product of $\br$ is $\wo$;  see \eqref{eq:dem}. We consider the set $\bigBR_\br$ of tuples $(X_\bullet,Y_\bullet)$ of weighted flags satisfying the relative position conditions
\begin{equation}\label{eq:braidDiagram}
	\begin{tikzcd}
		X_0& \arrow[l,"{s_{i_1}^+}"'] X_1& \arrow[l,"{s_{i_2}^+}"'] \cdots& \arrow[l,"{s_{i_m}^+}"'] 
		X_m \\
		Y_0 \arrow[r,"{s_{i_1^\ast}^-}"'] \arrow[u, Rightarrow, "\wo"']& Y_1 \arrow[r,"{s_{i_2^\ast}^-}"'] & \cdots \arrow[r,"{s_{i_m^\ast}^-}"'] & Y_m. \ar[equal]{u} %
	\end{tikzcd}
\end{equation}
The group $G$ acts on $\bigBR_\br$ by acting on each individual weighted flag, and this action is free. The \emph{double braid variety} $\BR_\br:=\bigBR_\br/G$ is defined as the quotient of $\bigBR_\br$ modulo this $G$-action. It is a smooth, affine, irreducible complex algebraic variety (\cref{prop:smooth_affine_etc}).

Double braid varieties include open Richardson varieties~\cite{Rietsch,KLS2,Lec}, open positroid varieties~\cite{KLS}, double Bott-Samelson cells~\cite{ShWe}, the strata in \cite{WY}, and the braid varieties of~\cite{Mellit_cell,CGGS}; see~\cite{GLSBS1} for further discussion. Braid varieties have deep connections to knot theory, as their cohomology recovers part of the Khovanov--Rozansky homology~\cite{KR1,KhoSoe} of the associated link \cite{GL_qtcat,Trinh,CGGS2}.

For each $\beta$, we construct a cluster seed $\Sigma_\br$.  Our main result settles conjectures of~\cite{Lec,CGGS2} and generalizes work of~\cite{BFZ,GL_cluster,Ing,ShWe}. 
\begin{theorem}\label{thm:main}
The coordinate ring of $\BR_\br$ is isomorphic to the cluster algebra $\Abr = \A(\Sigma_\br)$.
\end{theorem}

It would be interesting to compare our construction and the cluster-categorical approach of \cite{GLS_rigid,BIRS,Lec, Menard, CaoKel}, as was done in type A in \cite{SSB}.

\begin{remark}
At the final stages of completing our construction, we learned that a cluster structure for braid varieties was independently announced in a recent preprint~\cite{CGGLSS}. 
We thank the authors of~\cite{CGGLSS} for updating us on their progress. It would be interesting to understand the relation between our approach and their Legendrian-geometric viewpoint.
\end{remark}

  One application of \cref{thm:main} is that a \emph{curious Lefschetz theorem} (see~\cite{HR,LS}, \cite[Theorem~10.1]{GLSBS1}, and~\cite[Theorem~1.5]{GL_qtcat}) holds for double braid varieties; see \cref{thm:cur_lef}. In the case of open Richardson varieties, this implies that the doubly-graded extension group ${\rm Ext}_{\mathcal{O}}(M_w,M_v)$ of two Verma modules in Category $\Ocal$ satisfies curious Lefschetz; cf.~\cite[Section~1.11]{GL_qtcat}.

\subsection{Cluster variables and Deodhar geometry}\label{sec:intro:constr}
To construct a cluster structure on $\BR_\br$, we need to identify certain regular functions on $\BR_\br$ as (initial) cluster variables, and then construct the exchange matrix $\tB$, or quiver, of the initial seed.   In previous works \cite{Scott,GL_cluster,BFZ,ShWe} establishing cluster structures on the Grassmannian, positroid varieties, double Bruhat cells and double Bott Samelson cells, the cluster variables are \emph{(generalized) minors} of some kind.  Determinantal identities satisfied by the minors become exchange relations in the cluster algebra.  A fundamental obstacle, already pointed out by Leclerc \cite{Lec}, to extending these constructions to open Richardson varieties or braid varieties, is that the (generalized) minors are no longer irreducible elements of the coordinate ring $\C[\BR_\br]$.

In the present work, we construct cluster variables using \emph{Deodhar geometry}; see \cref{sec:deodhar-geometry} for details and \cite{GLSBS1,GalTut} for examples.\footnote{To compare our quivers to the quivers in~\cite{GLSBS1,GalTut}, all arrows need to be reversed.} We introduce an open dense algebraic torus $\Tbr\subset \BR_\br$ called the \emph{Deodhar torus}, so named for its relation to the Deodhar decomposition of Richardson varieties \cite{Deo, MR}. It is defined by requiring the weighted flags $X_c,Y_c$ to be weakly $\wp c$-related, where $\wp c\in W$ is maximal possible subject to~\eqref{eq:braidDiagram} (for each $c=0,1,\dots,m$). The complement $\BR_\br\setminus \Tbr$ is a union of irreducible \emph{mutable Deodhar hypersurfaces} $\{V_c\mid c\in\Jmut\}$. 
We define a partial compactification of $\BR_\br$ so that the complement of $\Tbr$ in it also includes \emph{frozen Deodhar hypersurfaces} $\{V_c\mid c\in\Jfro\}$. We let $\Jo:=\Jfro\sqcup\Jmut$. 
The following definition, suggested by David Speyer, is key to our approach.

\begin{propdef}\label{propdef:intro_cluster}
For $c \in \Jo$, define the \emph{cluster variable} $x_c$ to be the unique character of $\torus(\br)$ that vanishes to order one on $V_c$ and has neither a pole nor a zero on $V_{\cp}$ for $\cp\in\Jo\setminus\{c\}$. We denote the cluster by $\xbr=\{x_c\}_{c \in \Jo}$.
\end{propdef}

We show that the cluster variables thus defined form a basis of the character lattice of $\Tbr$, and that they extend to regular functions on the braid variety $\BR_\br$.  A particular set of generalized minors also form a basis of the character lattice of $\Tbr$; the two sets of functions are related by an upper-triangular monomial transformation.  Since Deodhar hypersurfaces are irreducible, cluster variables are the irreducible factors of these generalized minors.

\subsection{Exchange matrix}
In earlier works on the construction of cluster structures, the exchange matrix $\tB$ is obtained directly from the combinatorics of planar bipartite graphs~\cite{Pos,Scott,GL_cluster} or double wiring diagrams~\cite{BFZ,ShWe}, and can be defined using ``local contributions" from each edge or each crossing, respectively. 
Our construction of $\tB$ uses similar combinatorics, but in addition we must take into account the monomial transformation between cluster variables and generalized minors.
 
In the spirit of Fock and Goncharov \cite{FoGo_X}, the exchange matrix $\tB$ is equivalent to the datum of a $2$-form $\ombr$ on $\BR_\br$.  We define $\ombr$ in terms of generalized minors as a sum of local contributions for each letter of the double braid word $\br$.  We introduce integers $\dbr=(d_c)_{c\in\Jo}$  and then expand $\ombr$ in the basis of cluster variables:
\begin{equation}\label{eq:intro:omcoef_dfn}
  \ombr=\sum_{c,\cp\in\Jo:\ c\leq \cp} d_\cp \omcoef_{c\cp} \dlog x_c \wedge \dlog x_\cp =\sum_{c,\cp\in\Jo:\ c\leq \cp} d_c \omcoef_{\cp c} \dlog x_\cp \wedge \dlog x_c.
\end{equation}
The coefficients $\omcoef_{c\cp}$ %
define a $\Jo\times\Jmut$ integer matrix $\tB:=(\omcoef_{c\cp})$ . %
The principal $\Jmut\times\Jmut$ part of the matrix $\tB$ is skew-symmetrizable, with symmetrizer $\diag(d_c\mid c\in\Jmut)$.
 Therefore $\Sbr:=(\xbr,\tB)$ is a seed of a cluster algebra $\Acalbr$. The content of \cref{thm:main} is that $\Acalbr=\C[\BR_\br]$.
 
The factorization of generalized minors into cluster variables, and thus the exchange matrix $\tB$ itself, is difficult to describe directly in a combinatorial manner.  In the case $G=\SL_n$, we described the factorization in~\cite{GLSBS1} using the combinatorics of \emph{3D plabic graphs}. In this work, we approach it geometrically, via orders of vanishing of the minors on Deodhar hypersurfaces. In \cref{sec:comb-algor}, we give an algorithm, implemented in~\cite{GalTut}, for computing these orders of vanishing --- and thus the entire cluster seed --- using only root-system combinatorics.
 
\subsection{Deletion-contraction induction}\label{sec:intro:overview}
Our proof of \cref{thm:main} is inductive.  We introduce in \cref{sec:dc} a \emph{deletion-contraction recurrence} in the context of cluster varieties, which serves as the main ingredient of the inductive step.  Let $X$ be an algebraic variety and let $\Sigma$ be a cluster seed on $X$. Consider a sink in the mutable part of the quiver, with corresponding cluster variable $x$. We consider two subvarieties $W := \{x\neq0\}$ and $V := \{x=0\}$ of $X$. We show in \cref{thm:abstract-clust-stuc-del-con} that if both $V$ and $W$ are cluster varieties, and if some technical assumptions on $X$ and $\Sigma$ are satisfied, then $X$ is a cluster variety.  We expect that deletion-contraction can be applied in situations beyond braid varieties.

We apply deletion-contraction to $\BR_\br$ when $\br=ii\br'$ starts with a repeated letter; the braids $i\br'$ and $\br'$ correspond to deletion and contraction, respectively. To transform an arbitrary braid word $\br$ to one of this form, we utilize \emph{double braid moves} $\beta \sim \beta'$ on double braid words that induce natural isomorphisms $\BR_\br \cong \BR_{\br'}$. The full flexibility of these moves is the reason we consider double braid words and varieties here. Every double braid variety is isomorphic to some braid variety of~\cite{Mellit_cell,CGGS}, but using double braid words rather than usual braid words gives us access to more seeds.

In \cref{thm:moves}, we prove the key feature of double braid moves $\beta \sim \beta'$: under the isomorphism $\BR_\br \cong \BR_{\br'}$ the corresponding seeds $\Sigma_\br$ and $\Sigma_{\br'}$ are related by mutation. In \cref{sec:2braid}, we prove \cref{thm:moves} in the simply-laced case (i.e., for $G$ of type $A$, $D$, $E$);
 the seeds $\Sigma_\br$ and $\Sigma_{\br'}$ either coincide or are related by a single mutation. 
The proof of \cref{thm:moves} in the multiply-laced case is achieved via folding in \cref{sec:folding,sec:mult_laced}; the seeds $\Sigma_\br$ and $\Sigma_{\br'}$ are related by a sequence of mutations.

\subsection*{Acknowledgments}
We are indebted to David Speyer for his contributions to this project.  %
We thank Roger Casals, Eugene Gorsky, and Daping Weng for inspiring conversations. We are grateful to the anonymous referees for their valuable feedback on the first version of the manuscript.

\section{Deodhar geometry}\label{sec:deodhar-geometry}
We discuss the geometry of the double braid variety $\BR_\br$ with the goal of defining a cluster seed on it. The ingredients of a cluster seed were outlined in \cref{sec:intro:constr}. In \cref{sec:deo_tor}, we construct a \emph{Deodhar torus} $\Tbr\subset\BR_\br$. In \cref{sec:deo_cluster}, we introduce a family $\xbr=\{x_c\}_{c\in\Jo}$ of \emph{cluster variables} and show that they are regular functions on $\BR_\br$. Finally, in \cref{sec:deo_form}, we introduce a $2$-form $\ombr$ on $\Tbr$ from which the $\tB$-matrix can be extracted via~\eqref{eq:intro:omcoef_dfn}.

\subsection{Background}\label{sec:background}
 For each $i\in I$, we fix a group homomorphism 
\begin{equation*}%
  \phi_i:\SL_2\to G,\quad \begin{pmatrix}
1 & t\\
0 & 1
  \end{pmatrix}\mapsto x_i(t),\quad  \begin{pmatrix}
1 & 0\\
t & 1
  \end{pmatrix}\mapsto y_i(t),
\end{equation*} 
where $x_i(t),y_i(t)$ are the exponentiated Chevalley generators. The data $(\H,B_+,B_-,x_i,y_i;i\in I)$ is a \emph{pinning} of $G$; see~\cite[Section~1.1]{Lus2}.

Let $\Phi$ be the root system of $G$, with positive roots $\Phi^+$ corresponding to $B_+$. Let $\chL\H:=\Hom(\H,\Cast)$ be the \emph{character lattice} of $\H$ and $\cochL\H:=\Hom(\Cast,\H)$ be the \emph{cocharacter lattice} of $\H$. Let $\{\alpha_i\}_{i\in I}\subset\chL\H$ (resp., $\{\ach_i\}_{i\in I}\subset\cochL\H$, $\{\om_i\}_{i\in I}\subset\chL\H$) be the simple roots (resp., simple coroots, fundamental weights) of $\Phi^+$. 
 We have a natural pairing $\<\cdot,\cdot\>:\chL\H\times\cochL\H\to\Z$ satisfying $\<\om_i,\ach_j\>=\delta_{ij}$ and $\<\alpha_i,\ach_j\>=a_{ij}$, where $A=(a_{ij})_{i,j\in I}$ is the Cartan matrix of $G$.

Let the Weyl group $W$ have simple generators $\{s_i\}_{i\in I}$, length function $\ell(\cdot)$, and identity $\id\in W$. For $i\in I$, we set
\begin{equation*}%
  \ds_i=\line{s_i}:=\phi_i \begin{pmatrix}
    0 & -1\\ 
    1 & 0
  \end{pmatrix},\quad   \ds_i^{-1}=\lline{s_i}:=\phi_i \begin{pmatrix}
    0 & 1\\ 
    -1 & 0
  \end{pmatrix}.
\end{equation*}
For a reduced expression $w=s_{i_1}s_{i_2}\cdots s_{i_l}$, where $l=\ell(w)$, we set 
\begin{equation}\label{eq:dw_dfn}
  \dw=\line{w}:=\line{s_{i_1}}\cdot \line{s_{i_2}}\cdots \line{s_{i_l}},\quad \lline{w}:=\lline{s_{i_1}}\cdot \lline{s_{i_2}}\cdots \lline{s_{i_l}}.
\end{equation}
The resulting product does not depend on the choice of the reduced expression. 
For $u\in W$ and $h\in\H$, we set $u\cdot h:=\du h\du^{-1}=\line{u} h \line{u}^{-1}=\lline{u} h \lline{u}^{-1}$.
 We also consider elements
\begin{equation}\label{eq:zi_dfn}
  z_i(t):=\phi_i \begin{pmatrix}
	t & -1\\
	1 & 0
\end{pmatrix}
=x_i(t) \ds_i,%
\quad
\z_i(t):=\phi_i \begin{pmatrix}
	t & 1\\
	-1 & 0
\end{pmatrix}
=x_i(-t) \ds^{-1}_i.%
\end{equation}

For each $w\in W$, it is well known~\cite[Proposition~28.1]{Hum} that the multiplication map gives rise to an isomorphism
\begin{equation}\label{eq:U_factorization}
  (\dw^{-1} U_+ \dw \cap U_-) \times (\dw^{-1} U_+ \dw \cap U_+)\xrasim \dw^{-1} U_+ \dw.
\end{equation}

\subsection{Weighted flags}
Recall from \cref{sec:intro:double} that a \emph{weighted flag} is an element $F = gU_+ \in G/U_+$.  Associated to a weighted flag $F$ is the flag $\pi(F)=gB_+$, the image of $F$ in $G/B_+$.

	The following elementary facts can be found in e.g.~\cite[Appendix]{ShWe}; see also~\cite[Section~6.2]{GLSBS1}.
\begin{lemma}\label{lemma:rel-pos-facts} Let $F,F',F''$ be weighted flags. Suppose $v, w \in W$ with $\ell(vw)= \ell(v)+\ell(w)$.
	\begin{enumerate}
		\item $F \Rrel{\id} F'$ if and only if $F=F'$.
		\item If $F \Rrel{v} F' \Rrel{w} F''$, then $F \Rrel{vw}F''$.
		\item\label{rel-pos-facts3}  If $F \Rrel{vw}F''$, then there exists a unique $F'$ such that $F \Rrel{v} F' \Rrel{w} F''$. If $F\Rwrel{vw} F''$ then there exist unique $F'_1,F'_2$ such that $F \Rrel{v} F'_1 \Rwrel{w} F''$ and $F \Rwrel{v} F'_2 \Rrel{w} F''$.
	\end{enumerate}
\end{lemma}

\begin{lemma}\label{lem:paramSingleStep}%
  Suppose $F \Rrel{s_i} F'$ and say $F=gU_+$. Then there exists a unique $t \in \CC$ such that $F'=gz_i(t) U_+$. Similarly, if $F'=g'U_+$, there exists a unique $t' \in \CC$ such that $F=g' \z_i(t')U_+$. The maps $(g,F')\mapsto t$ and $(g',F)\mapsto t'$ are regular on the appropriate subvarieties of $G\times G/U_+$. %
\end{lemma}

\begin{lemma}\label{lem:morePositionFacts} %
	Suppose $F \Rwrel{v} g U_+ \Rrel{s_i} g z_i(t) U_+$ and $F \Rwrel{w} g z_i(t) U_+$. If $vs_i>v$, then $w=vs_i$ for all $t \in \C$. If $vs_i<v$, then there exists $\tzero\in\C$ such that $w=vs_i$ for $t=\tzero$ and $w=v$ for $t \in \C\setminus\{\tzero\}$. %
\end{lemma}

\subsection{Distinguished subexpressions}\label{sec:dist_subexpr} To define the Deodhar torus of $\BR_{\beta}$, we need the following combinatorics.

Fix $\br=i_1i_2\dots i_m\in\pmnm$. We write $[m]:=\{1,2,\dots,m\}$ and $[0,m]:=\{0,1,\dots,m\}$.  
 Recall the notation $s_{i}^+$ and $s_i^{-}$ from \eqref{eq:s_i-notation}.

Given two elements $u, v\in W$ with $u \leq v$ in the Bruhat order, we write $\max(u,v):=v$ and $\min(u,v):=u$. We also define $u * s_i := \max(u, us_i)$.
The \emph{Demazure product} of $\br$ is defined by
\begin{equation}\label{eq:dem}
  \Demprod(\br):= s_{i_m}^- *s_{i_{m-1}}^-*\cdots* s_{i_1}^- * s_{i_1}^+*s_{i_2}^+ * \cdots* s_{i_m}^+ \quad \in W.
\end{equation}
 From now on, we assume that $\Demprod(\br)=\wo$.

\begin{definition}\label{dfn:pds}
  A \emph{$\wo$-subexpression} of $\br$ is a sequence $\bu=(\up0,\up1,\dots,\up m) \in W^{m+1}$ such that $\up0=\id$, $\up m=\wo$, and such that for each $c\in[m]$, we have either $\up{c-1}=\up{c}$ or $\up{c-1}=s_{i_c}^-\up{c}s_{i_c}^+$. Since $\Demprod(\br)=\wo$ there exists a unique ``rightmost'' subexpression, called the \emph{positive distinguished subexpression}; see~\cite[Lemma~3.5]{MR}. It is given by $\up m:=\wo$ and
  $\up{c-1}:=\min(\up c, s_{i_c}^-\up{c} s_{i_c}^+)$
  for all $c=m,m-1,\dots,1$. 
\end{definition}
From now on, we fix $\bu=(\up0,\up1,\dots,\up m)$ to be the positive distinguished $\wo$-subexpression of $\br$.   We also define $\wp c:= \wo\up c $ for $c \in [0, m]$ and $\bw=\wo\bu:=(\wp0,\wp1,\dots,\wp m)$. Note that $\wp0=\wo$.
We set $\Jo:=\{c\in[m]\mid \up c=\up{c-1}\}$. We refer to the indices in $\Jo$ as \emph{solid crossings} and to the indices in $\Hollow$ as \emph{hollow crossings}.  We denote $d(\br) := m - \ell(\wo)=|\Jo|$. 

\subsection{Double braid varieties and the Deodhar torus}\label{sec:deo_tor}
Recall that 
\begin{equation*}%
  \bigBR_{\br}:= \{(X_\bullet, Y_\bullet)\in(G/U_+)^{[0, m]} \times (G/U_+)^{[0, m]}\mid  (\Xbul,\Ybul) \text{ satisfy } \eqref{eq:braidDiagram}\}
\end{equation*}
and that $G$ acts on $\bigBR_\br$ by acting on each weighted flag.

\begin{proposition}\label{prop:smooth_affine_etc}
The $G$-action on $\bigBR_\br$ is free. The quotient $\BR_\br:=\bigBR_\br/G$ is a smooth, factorial, affine, irreducible complex algebraic variety of dimension $d(\br)$.  
\end{proposition}
\begin{proof}
We repeat the argument from~\cite[Proposition~6.13]{GLSBS1}. 
Consider the space of tuples of weighted flags satisfying
\begin{equation}\label{eq:Ugauge}
  \begin{tikzcd}
    X_0 = U_+& \arrow[l,"\textcolor{black}{s_{i_1}^+}"'] X_1& \arrow[l,"\textcolor{black}{s_{i_2}^+}"'] \cdots& \arrow[l,"\textcolor{black}{s_{i_m}^+}"'] 
    X_m \\
    Y_0 \arrow[r,"\textcolor{black}{s_{i_1^\ast}^-}"'] %
    & Y_1 \arrow[r,"\textcolor{black}{s_{i_2^\ast}^-}"'] & \cdots \arrow[r,"\textcolor{black}{s_{i_m^\ast}^-}"'] & Y_m \arrow[u, rightarrow, "\id"']
  \end{tikzcd} 
\end{equation}
This space is an iterated $\C$-bundle and is thus affine. Imposing the condition that $U_+$ and $Y_0$ are weakly $\wo$-related (that is, $Y_0\in B_+ \wo B_+=U_+ \wo B_+$) cuts out a nonempty smooth affine open subset $V$ of the iterated $\C$-bundle. The braid variety $\BR_{\br}$ is the quotient of $V$ by the diagonal action of $U_+=\text{Stab}_G(U_+)$. The group $U_+$ acts freely on $U_+ \dot{w}_0 B_+$ and thus acts freely on $V$. It follows that the quotient $\BR_{\br}$ is also smooth and affine; it is also clearly irreducible. Explicitly, one may fix the $U_+$-action by identifying $\BR_\br$ with the subvariety of $V$ where $Y_m\in \wo B_+$, which is the viewpoint of~\cite{CGGS}.

To see that $\BR_{\br}$ is factorial, note that $\C[V]$ is a UFD, as it is a localization of a polynomial ring. The coordinate ring $\C[\BR_\br]$ is equal to $\C[V]^{U_+}$, which is also a UFD by 
\cite[Theorem 3.17]{PopVin}.

For the dimension, note that 
\[\dim( \bigBR_{\br})= \dim (G/U_+) + m= \dim(G)- \ell(\wo ) + m,\]
so $\dim(\BR_{\br})=\dim( \bigBR_{\br})-\dim(G) = m-\ell(\wo ) = d(\br)$.
\end{proof}

Let $\BigBR_\br$ be a partial compactification of $\bigBR_\br$ obtained by removing the condition $X_0 \Lwrel{\wo} Y_0$ from~\eqref{eq:braidDiagram}.

\begin{definition}
Let $\tTbr\subset\BigBR_\br$ be the set of tuples $(\Xbul,\Ybul)\in\BigBR_\br$ satisfying
\begin{equation}\label{eq:DTcond}
  X_c \Lwrel{\wp c} Y_c \quad\text{ for $c\in[0,m]$.}
\end{equation}
Since $\wp0 = \wo$, we have $\tTbr \subset\bigBR_\br$ and thus $G$ acts freely on $\tTbr$. Define the \emph{Deodhar torus} $\Tbr\subset\BR_\br$ to be the quotient $\Tbr:=\tTbr / G$.
\end{definition}
\noindent We will show in \cref{cor:deo_tor} below that $\Tbr$ is indeed a $d(\br)$-dimensional algebraic torus.

\begin{lemma}\label{lem:tTbr_open_dense}
The subsets $\tTbr\subset \bigBR_\br$ are open dense in $\BigBR_\br$.
\end{lemma}
\begin{proof}
We can parameterize the variety $\BigBR_\br$ as follows. We choose an arbitrary weighted flag $X_m=Y_m$, and then for $c=m,m-1,\dots,1$, assuming $(X_c,Y_c)=(g_cU_+,g'_cU_+)$, we set (cf. \cref{lem:paramSingleStep})
  \begin{equation}\label{eq:propagate_left}
    (X_{c-1},Y_{c-1}):=
    \begin{cases}
      (g_cz_{i_c}(\tparp_c)U_+,g'_cU_+), &\text{if $i_c>0$,}\\
      (g_cU_+,g'_c\z_{|i_c|^*}(\tparp_c)U_+), &\text{if $i_c<0$,}\\
    \end{cases}
  \end{equation}
 for arbitrary parameters $\btparp:=(\tparp_1,\tparp_2,\dots,\tparp_m)\in\C^m$.

For $(\Xbul, \Ybul)$ to be a point in $\tTbr$, \eqref{eq:DTcond} needs to be satisfied for each $c\in[0,m]$. It is clearly satisfied for $c=m$. If it is satisfied for some solid index $c\in[m]$, then it will be satisfied for $c-1$ if and only if the parameter $\tparp_c$ is not equal to the value $\tzero$ from \cref{lem:morePositionFacts}. If the index $c\in[m]$ is instead hollow and~\eqref{eq:DTcond} is satisfied for $c$ it will be satisfied for $c-1$ regardless of the value of $\tparp_c$. 
Thus, $\tTbr$ is indeed an open dense subset of $\BigBR_\br$. Since $\tTbr\subset\bigBR_\br$, $\bigBR_\br$ is dense in $\BigBR_\br$. Since $\bigBR_\br$ is cut out of $\BigBR_\br$ by an open condition $X_0 \Lwrel{\wo} Y_0$, it is open in $\BigBR_\br$.
\end{proof}

\subsection{$\H$-valued functions}\label{sec:torus-valued}
Recall that $\H= B_+ \cap B_-$ is the Cartan torus of $G$. Over the next sections, we will discuss various functions on $\Tbr$ and $\BR_{\br}$. To do so, it is convenient to introduce a regular map $\Tbr \to H$. We also use this map in this section to show that $\Tbr$ is an algebraic torus.

Given $(\Xbul,\Ybul) \in \BR_\br$, let $Z_c := Y_c^{-1} X_c \in U_+ \backslash G / U_+$. Abusing notation, we use double cosets $Z_c \in U_+ \backslash G / U_+$ interchangeably with their representatives in $G$. 
 For $(\Xbul,\Ybul) \in \Tbr$, $Z_c$ belongs to the Bruhat cell $\Bruho_{\wp c}:=B_+\wp c B_+=U_+\wp c \H U_+$ of $G$.\footnote{Here and below, we omit the dot over $\dwp c$ in products such as $B_+ \wp c B_+$ that involve multiplying $\dwp c$ by a subgroup of $G$ containing $\H$.
} %
 There exist unique elements $\hr_c,\hb_c\in\H$ satisfying 
\begin{equation}\label{eq:h_c_dfn}
  Z_c\in U_+ \dwp c \hr_c U_+,\quad Z_c\in U_+ \line{\wo} \hb_c \lline{\up c} U_+, \quad\text{thus,}\quad \hb_c=\up c\cdot \hr_c.
\end{equation}
The third statement follows from the first two since $\dwp  c=\line{\wo}\cdot \lline{\up c}$ and $\up c\cdot  \hr_c=\lline{\up c} \hr_c \lline{\up c}^{-1}$. %
The map $(\Xbul,\Ybul)\mapsto \hrb_c$ is a rational $\H$-valued function on $\BR_\br$ (resp., on $\BigBR_\br$), regular on $\Tbr$ (resp., on $\tTbr$).

\begin{lemma}\label{lemma:tpar}
  There exist rational functions $(\tpar_c)_{c\in\Jo}$ on $\BR_\br$ such that for $c\in[m]$, 
\begin{equation}\label{eq:h_tpar}
  \hr_{c-1}=
  \begin{cases}
    s_{i_c}\cdot  \hr_c, &\text{if $c$ is hollow, $i_c\in I$;}\\
     \alphacheck_{i_c}(\tpar_c) \hr_c, &\text{if $c$ is solid, $i_c\in I$;}\\
  \end{cases} \quad
  \hb_{c-1}=
  \begin{cases}
    s_{|i_c|}\cdot  \hb_c, &\text{if $c$ is hollow, $i_c\in -I$;}\\
     \alphacheck_{|i_c|}(\tpar_c) \hb_c, &\text{if $c$ is solid, $i_c\in -I$.}\\
  \end{cases}
\end{equation}
\end{lemma}
\begin{proof}
Fix $(\Xbul, \Ybul)\in \Tbr$. For $c\in\Jo$, define $\tpar_c$ to be such that if $Z_c=\dwp c \hr_c=\line{\wo} \hb_c \lline{\up c}$ then $Z_{c-1}=Z_c z_{i_c}(\tpar_c)$ if $i_c\in I$ and $Z_{c-1}=\z_{|i_c|^\ast}(\tpar_c)^{-1} Z_c$ if $i_c\in-I$; see \cref{lem:paramSingleStep}. The following identities in $G$ can be checked inside $\SL_2$:
\begin{equation}\label{eq:SL2_xs_sx}
  x_i(t)\ds_i=y_i(1/t)\alphacheck_i(t) x_i(-1/t) \quad\text{and}\quad \ds_i x_i(t)=x_i(-1/t) \alphacheck_i(1/t) y_i(1/t).
\end{equation}

Suppose that $i_c\in I$. We have $Z_{c-1}= \dwp c \hr_c x_{i_c}(\tpar_c) \ds_{i_c}$. If $c$ is hollow then $\dwp c \hr_c x_{i_c}(\tpar_c)\in U_+ \dwp c \hr_c$, and thus $Z_{c-1}\in U_+ \dwp c \ds_{i_c} (s_{i_c}\cdot \hr_c)$. This implies that $\hr_{c-1}=s_{i_c}\cdot \hr_c$. If $c$ is solid then we use the first identity in~\eqref{eq:SL2_xs_sx} to write $Z_{c-1}= \dwp c \hr_c y_{i_c}(1/\tpar_c)\alphacheck_{i_c}(\tpar_c) x_{i_c}(-1/\tpar_c)$. Since $\dwp c\hr_c y_{i_c}(t)\in U_+ \dwp c\hr_c$, we see that $Z_{c-1}\in U_+ \dwp c \hr_c \alphacheck_{i_c}(\tpar_c) U_+$, which implies that $\hr_{c-1}=\hr_c \alphacheck_{i_c}(\tpar_c)$. 

The case when $i_c\in -I$ is handled similarly. When $c$ is solid, we use the second identity in~\eqref{eq:SL2_xs_sx} together with $\ach_{|i_c|^\ast}(1/\tpar_c) \line{\wo}=\line{\wo} \ach_{|i_c|}(\tpar_c)$; see~\cite[Equation~(1.2)]{FZ_double}. 
\end{proof}

\begin{corollary}\label{cor:h_stays}
Suppose $c$ is hollow. If $i_c\in I$ then $\hb_{c-1}=\hb_c$, and if $i_c\in -I$ then $\hr_{c-1}=\hr_c$. 
\end{corollary}

\begin{corollary}\label{cor:deo_tor}
The Deodhar torus $\Tbr \subset \BR_{\br}$ is isomorphic to an algebraic torus of dimension $d(\br)$, and the functions $(t_c)_{c\in\Jo}$ form a basis of the character lattice of $\Tbr$.
\end{corollary}
\begin{proof}
For $c\in\Jo$, the function $\tpar_c$ is regular on $\Tbr$ by \cref{lem:paramSingleStep}.  With notation as in the proof of \cref{lemma:tpar}, we have $Z_c=\dwp c \hr_c=\line{\wo} \hb_c \lline{\up c}$ and $Z_{c-1}=Z_c z_{i_c}(\tpar_c)$ if $i_c\in I$ and $Z_{c-1}=\z_{|i_c|^\ast}(\tpar_c)^{-1} Z_c$ if $i_c\in-I$.  It follows that $t_c \neq 0$ if and only if $Z_{c-1} \in \Bruho_{\wp{c-1}}=\Bruho_{\wp{c}}$ (thus, $\tzero = 0$ in the notation of \cref{lem:morePositionFacts}).  Thus, we get a regular map $\Tbr\to(\Cx)^{\Jo}$, $(\Xbul,\Ybul)\mapsto (\tpar_c)_{c\in\Jo}$.

To show that this map is an isomorphism, we construct the inverse: given $(\tpar_c)_{c\in\Jo}$, we explain how to recover $(\Xbul,\Ybul)\in\Tbr$. We have $\hr_m=\hb_m=1\in\H$, and for $c=m-1,m-2,\dots,0$, we recover $\hr_c,\hb_c$ using~\eqref{eq:h_tpar}. Next, since $(\Xbul,\Ybul)$ is defined up to $G$-action, we may assume that $Y_0 = U_+$ and $X_0 = \dwo\hr_0$. Suppose that for $c=1,2,\dots,m$, $(X_{c-1},Y_{c-1})$ has been recovered.
 If $c$ is hollow then $(X_c,Y_c)$ is recovered uniquely using part~\eqref{rel-pos-facts3} of \cref{lemma:rel-pos-facts}. If $c$ is solid then the parameter $\tparp_c$ in~\eqref{eq:propagate_left} is recovered uniquely from $t_c$ by an argument similar to the proof of \cref{lemma:tpar}. (Alternatively, one can recover $\tparp_c$ as a ratio of grid minors discussed in the next subsection.) We have constructed the inverse map, and it is clearly regular on the torus $(\Cx)^{c\in\Jo}$.
\end{proof}

\subsection{Grid and chamber minors}\label{sec:minors} In this section, we introduce a basis of characters of $\Tbr$ consisting of certain generalized minors called \emph{grid minors}. This basis will be crucial in our definition of the exchange matrix. We will show later in \cref{cor:cluster_vars} that grid minors are related to cluster variables by an invertible monomial transformation.

Recall that $A = (a_{ij})_{i,j\in I}$, $a_{ij}:= \< \alpha_i, \ach_j \>$ is the Cartan matrix, and that $d_i a_{ij} = d_j a_{ji}$.  For $i,j \in \pm I$, we define $a_{ij} = 0$ if $i,j$ have different signs, and $a_{ij}= a_{(-i)(-j)}$ otherwise.  Also set $d_{-i} := d_{i}$ for $i\in I$. Abusing notation, given a double braid word $\br$, we 
let 
\begin{equation}\label{eq:d_dfn}
  \dbr=(d_c)_{c\in\Jo}, \quad\text{where}\quad d_c := d_{i_c} \quad\text{for $c\in\Jo$}.
\end{equation}

Following~\cite[Definition~1.4]{FZ_double}, for $i \in I$ and $v,w \in W$, we have a \emph{generalized minor} $\Delta_{v\omega_i, w\omega_i} : G \to \C$. It is a regular function satisfying
\begin{align}%
\label{eq:delta(yxy)}
  \Delta_{\om_i,\om_i}(y_-xy_+)&=  \Delta_{\om_i,\om_i}(x) \quad\text{for all $(y_-,x,y_+)\in U_-\times G\times U_+$;}\\
  \Delta_{v\om_i,w\om_i}(x)&=\Delta_{\om_i,\om_i}\left(\lline{v^{-1}} x\line{w}\right)=\Delta_{\om_i,\om_i}\left(\dv^{-1} x\dw\right);
\end{align}
see~\cite[Section~1.4]{FZ_double}. For $h\in\H$, we have $\Delta_{\om_i,\om_i}(h)=h^{\om_i}$. For $x\in G$, we also have~\cite[Equation~(2.14)]{FZ_double}
\begin{equation}\label{eq:domk(xh)}
  \domk(xh)=\domk(hx)=h^{\om_k}\domk(x).
\end{equation}

\begin{definition}\label{dfn:grid_minors}
	For $c \in [0,m]$ and $k \in I$, we define the \emph{grid minors} %
	\begin{equation}\label{eq:DeltaGrid}
          \grid_{c, k}(\Xbul,\Ybul)= \Delta_{\wp c\omega_k, \omega_k}(Z_c) \quad\text{and}\quad \grid_{c,-k}(\Xbul,\Ybul)=\Delta_{\wo \omega_k, u\pd{c}^{-1}\omega_k}(Z_c).
	\end{equation}
\end{definition} 
\noindent The \emph{chamber minors} are defined as $\crossing{c}:=\Delta_{c-1, i_{c}}$, for $c\in[m]$.

\begin{lemma}\label{lemma:hc}
For $c \in [0,m]$ and $k\in I$, the grid minors $\grid_{c,k}$ and $\grid_{c,-k}$ are well-defined regular functions on $\tTbr$. For $(\Xbul,\Ybul)\in\tTbr$, we have
\begin{equation}\label{eq:hc}
  \grid_{c,k}(\Xbul,\Ybul) =  (\hr_c)^{\om_k} \quad\text{and}\quad \grid_{c,-k}(\Xbul,\Ybul) = (\hb_c)^{\om_k}.
\end{equation}
\end{lemma}
\begin{proof}
Recall that we view $Z_c \in \UGU$ as an element of $G$, and that for $(\Xbul,\Ybul)\in\tTbr$, we have $Z_c\in\Bruho_{\wp c}=U_+\wp c \H U_+$ for all $c\in[0,m]$.  Write $Z_c=y_+'\dwp c \hr_c y_+''$ for $y_+',y_+''\in U_+$. We have
\begin{equation*}%
  \grid_{c,k}(\Xbul,\Ybul)= \Delta_{\wp c\omega_k, \omega_k}(Z_c)
=\domk\left(\lline{\wp c^{-1}} y_+'\dwp c \hr_c y_+''\right)
=\domk(\dwp c^{-1} y_+' \dwp c \hr_c).
\end{equation*}
Factorizing $\dwp c^{-1} y_+' \dwp c=b_-b_+$ for $(b_-,b_+)\in U_-\times U_+$ using~\eqref{eq:U_factorization}, using $(\hr_c)^{-1}b_+\hr_c\in U_+$, and applying~\eqref{eq:domk(xh)}, we get the first identity in~\eqref{eq:hc}. In particular, since the result does not depend on $y_+',y_+''\in U_+$, we see that $\grid_{c,k}(Z_c)$ is invariant under the $U_+\times U_+$-action on $Z_c$, and thus descends to a well-defined function on $U_+ \backslash \Bruho_{\wp c} / U_+$.
 The proof for $\grid_{c,-k}$ is similar.
\end{proof}

\begin{corollary}\label{cor:grid_bigBR}
For each $c\in[0,m]$ and $k\in I$, the grid minors $\grid_{c,k}$ and $\grid_{c,-k}$ are well-defined regular functions on $\BigBR_\br$, $\bigBR_\br$, and $\BR_\br$. 
 These regular functions commute with the quotient map $\bigBR_\br\to\BR_\br$ and the inclusion map $\bigBR_\br\hookrightarrow \BigBR_\br$.
\end{corollary}
\begin{proof}
For $(\Xbul,\Ybul)\in\tTbr$, we have $Z_c\in\Bruho_{\wp c}$, and the proof of \cref{lemma:hc} implies that $\Delta_{\wp c\omega_k, \omega_k}$ is well defined on $U_+ \backslash \Bruho_{\wp c} / U_+$. Since $\tTbr$ is dense in $\BigBR_\br$ by \cref{lem:tTbr_open_dense}, it follows that for $(\Xbul,\Ybul)\in\BigBR_\br$, we have $Z_c\in\Bruh_{\wp c}$. By continuity, it follows that $\Delta_{\wp c\omega_k, \omega_k}$ is well defined on $U_+ \backslash \Bruh_{\wp c} / U_+$. Thus, $\grid_{c,k}$ (and similarly $\grid_{c,-k}$) is well defined on $\BigBR_\br$. The fact that it commutes with the quotient map $\bigBR_\br\to\BR_\br$ follows since the map $(X_c,Y_c)\mapsto Z_c=Y_c^{-1}X_c$ commutes with the $G$-action $(X_c,Y_c)\mapsto (gX_c,gY_c)$, $g\in G$, on $\bigBR_\br$. It also trivially commutes with the inclusion map $\bigBR_\br\hookrightarrow \BigBR_\br$.
\end{proof}

Combining \cref{lemma:hc} with~\eqref{eq:h_tpar}, we get the following. 
\begin{corollary}\label{cor:crossing_changes_one_minor}
If $c$ is solid and $k\in\pmn$ has the same sign as $i_c$ then 
\begin{equation}\label{eq:crossing_changes_one_minor}
  \grid_{c-1,k}=
  \begin{cases}
    \tpar_c\grid_{c,k}, &\text{if $k= i_c$;}\\
    \grid_{c,k}, &\text{if $k\neq i_c$.}\\
  \end{cases}
\end{equation}
\end{corollary}
\begin{proposition}\ \label{prop:deodharIsTorus} 
\begin{enumerate}
\item\label{item:deo_tor3} The grid minors are characters of $\Tbr $.
\item\label{item:deo_tor2}  The solid chamber minors $(\Delta_c)_{c\in\Jo}$ form a basis of the character lattice of $\Tbr$.
\end{enumerate}
\end{proposition}
\begin{proof}

We relate the parameters $(\tpar_c)_{c\in\Jo}$ from \cref{cor:deo_tor} to the grid and \cham minors by combining~\eqref{eq:h_tpar} with~\eqref{eq:hc}. Suppose that $i_c\in I$ and let $k\in I$. If $k\neq i_c$ then $\grid_{c-1,k}=\grid_{c,k}$ by \cref{cor:crossing_changes_one_minor}. If $c$ is hollow then 
\begin{equation*}%
  \grid_{c-1,i_c}=(\hr_{c-1})^{\om_{i_c}}=(s_{i_c}\cdot \hr_c)^{\om_{i_c}}=(\hr_c)^{s_{i_c} \om_{i_c}}. 
\end{equation*}
Expand $s_{i_c} \om_{i_c}=\om_{i_c}-\alpha_{i_c}$ in the basis of fundamental weights using $\alpha_{i_c}=\sum_{j\in I} a_{i_cj}\om_j$. This gives
\begin{equation}\label{eq:grid_mon_reln}
  \grid_{c-1,i}\grid_{c,i}\prod_{j\neq i} \grid_{c,j}^{a_{ij}}=1, \quad\text{if $c$ is hollow and $i:=i_c$},
\end{equation}
which holds for $i_c \in \pm I$.  If $c$ is solid, \cref{cor:crossing_changes_one_minor} yields $\grid_{c-1,i_c}=\tpar_c\grid_{c,i_c}$.  Thus,
 \begin{enumerate}[label=(\roman*)]
 \item\label{item:gridCrossingMono0} For each solid $c\in\Jo$, $\tpar_c=\grid_{c-1,i_c}/\grid_{c,i_c}$ is a ratio of two grid minors.
 \item\label{item:gridCrossingMono1} For each solid $c\in\Jo$, the grid minors $(\grid_{c-1, j})_{j\in\pmn}$ are Laurent monomials in the grid minors $(\grid_{c, k})_{k\in\pmn}$ and the \cham minor $\crossing{c}=\grid_{c-1, i_c}$.
 \item\label{item:gridCrossingMono2} For each hollow $c\in[m]\setminus\Jo$, the grid minors $(\grid_{c-1, j})_{j\in\pmn}$ are Laurent monomials in the grid minors $(\grid_{c, k})_{k\in\pmn}$.
 \item\label{item:gridCrossingMono3} Every grid minor $\grid_{c,j}$ is a Laurent monomial in the solid \cham minors $(\crossing{\cp})_{\cp\in\Jo}$.
 \end{enumerate}
We have already shown~\ref{item:gridCrossingMono0}--\ref{item:gridCrossingMono2}, and~\ref{item:gridCrossingMono3} follows from~\ref{item:gridCrossingMono1}--\ref{item:gridCrossingMono2}. This implies the result.
\end{proof}

\subsection{Almost positive sequences and Deodhar hypersurfaces}
Here we discuss the complement of the Deodhar torus and its irreducible components, the Deodhar hypersurfaces. We first introduce additional combinatorics, which we will use to describe an open subset of the Deodhar hypersurface. 

Recall that we have $\up{c-1}=\min(\up c,s_{i_c}^-\up{c}s_{i_c}^+)$ and $\wp{c-1}=\max(\wp c, s_{i_c^\ast}^-\wp{c} s_{i_c}^+)$ for all $c\in[m]$. 
\begin{definition}%
\label{def:almostPos}
  Let $\cp \in \Jo$. Let $\vp{m}\apd{\cp}:=\wo$, and for $c=m,m-1,\dots,1$, define
  \[\vp{c-1}\apd{\cp}:=\begin{cases}
      \max(\vp \cp\apd\cp,s_{i_\cp}^- \vp{\cp}\apd{\cp}  s_{i_\cp}^+),& \text{if $c=\cp$,}\\
      \min(\vp c\apd\cp, s_{i_{c}}^- \vp{c}\apd{\cp} s_{i_{c}}^+),& \text{otherwise}.
    \end{cases} \]
  We call the sequence $\bvap\apd \cp:=(\vp0\apd{\cp}, \dots, \vp m\apd{\cp})$ the \emph{$\langle \cp \rangle$-almost positive sequence}. We set $\wp c\apd \cp:=\wo\vp c\apd \cp$ for all $c\in[0,m]$, and write 
$\bw\apd \cp=\wo\bvap\apd \cp:=(\wo\vp0\apd{\cp}, \dots, \wo\vp m\apd{\cp})$.
\end{definition}
\begin{definition}\label{dfn:mutable}
We say that $\cp\in\Jo$ is \emph{mutable} if $\vp0\apd \cp=\id$. Otherwise, $\cp$ is \emph{frozen}. We let $\Jmut$ (resp., $\Jfro$) denote the set of mutable (resp., frozen) indices.
\end{definition}

\begin{definition}\label{dfn:V_d}
Let $\cp \in \Jo$. Define the \emph{open Deodhar hypersurface} $\tVo_\cp\subset\BigBR_\br$ by
\begin{equation}\label{eq:tVo_dfn}
  \tVo_\cp := \{(\Xbul,\Ybul)\in\BigBR_\br\mid X_{c} \Lwrel{\wp{c}\apd{\cp}} Y_c \quad\text{for all $c \in [0,m]$} \}.
\end{equation}
Define the \emph{(closed) Deodhar hypersurface} $\tV_\cp \subset \BigBR_{\br}$ to be the closure of $\tVo_\cp$.
\end{definition}

It follows that an index $\cp\in\Jo$ is mutable (resp., frozen) if and only if $\tVo_\cp\subset\bigBR_\br$ (resp., $\tVo_\cp\cap\bigBR_\br=\emptyset$). If $\cp$ is mutable then the $G$-action on $\tVo_\cp$ is free. In this case, we set $\Vo_\cp:=\tVo_\cp/G$ and let $V_\cp$ be the closure of $\Vo_\cp$ in $\BR_\br$.

\begin{proposition}\label{prop:V_d_irr}
The closed subset $\BigBR_\br\setminus \tTbr$ is a union of the Deodhar hypersurfaces $\tV_\cp$ for $\cp\in\Jo$. Each $\tV_\cp$ is irreducible and has codimension one in $\BigBR_\br$, and the hypersurfaces $\tV_\cp,\tV_{\cpp}$ are distinct for distinct $\cp,\cpp\in\Jo$.
\end{proposition}
\begin{proof}
 We prove the second sentence. 
As explained in the proof of \cref{lem:tTbr_open_dense}, $\BigBR_\br$ is an iterated fiber bundle over $G/U_+$, where each fiber is either $\C$ (if $c$ is hollow) or $\Cx$ (if $c$ is solid). Similarly, $\tVo$ is an iterated fiber bundle over $G/U_+$, where each fiber is either $\C$, $\Cx$, or (in the case of the crossing $c = \cp$) a point. It follows that $\tV_\cp$ is an irreducible subvariety of $\BigBR_\br$ of codimension one: we have $\dim\BigBR_\br = \dim(G/U_+) + m$ and $\dim\tV_\cp = \dim\tVo_\cp = \dim(G/U_+) + m-1$. For distinct $\cp,\cpp\in\Jo$, any point $(\Xbul,\Ybul)$ in $\tVo_{\cpp}$ satisfies $Z_{\cp-1} \in\Bruho_{\wp{\cp-1}}$. Meanwhile, any point $(\Xbul,\Ybul)$ in $\tV_{\cp}$ satisfies $Z_{\cp-1} \in\Bruh_{\wp{\cp-1}\apd{\cp}}$. Since $\wp{\cp-1}\apd{\cp}<\wp{\cp-1}$ in the Bruhat order, we see that $\tVo_{\cpp}\not\subset\tV_\cp$, and thus $\tV_{\cpp}\neq\tV_\cp$.

We now prove the first sentence. First, recall from \cref{lem:tTbr_open_dense} that $\tTbr$ is open in $\BigBR_\br$, so $\BigBR_\br\setminus\tTbr$ is closed. For each $\cp\in[m]$, we introduce an auxiliary subset
\begin{equation}\label{eq:tVgeq_dfn}
  \tVgeq_\cp := \{(\Xbul,\Ybul)\in\BigBR_\br\mid X_{c} \Lwrel{\wp{c}} Y_c \quad\text{for all $c \geq \cp$,\quad  and\quad 
$X_{\cp-1} \Lwrel{w'} Y_{\cp-1}$ for $w'\neq \wp{\cp-1}$.} \}.
\end{equation}
Recall that if $\cp$ is hollow then $X_{\cp} \Lwrel{\wp{\cp}} Y_\cp$ implies $X_{\cp-1} \Lwrel{\wp{\cp-1}} Y_{\cp-1}$. Thus, $\tVgeq_\cp$ is empty unless $\cp\in\Jo$. For $\cp\in\Jo$, the element $w'$ in~\eqref{eq:tVgeq_dfn} must be equal to $\wp{\cp-1}\apd{\cp}$ because there are only two possibilities for the relative position of $(X_{\cp-1}, Y_{\cp-1})$ given that $X_{\cp} \Lwrel{\wp{\cp}} Y_{\cp}$.

Let $(\Xbul,\Ybul)\in\BigBR_\br\setminus\tTbr$. Then~\eqref{eq:DTcond} must fail for $(\Xbul,\Ybul)$ for some index $c\in[0,m]$. We always have $X_m\Lwrel{\wp{m}}Y_m$, so there exists a unique $\cp\in[m]$ such that $(\Xbul,\Ybul)\in\tVgeq_\cp$. 
Applying an iterated fiber bundle argument as above, we see that $\tVo_\cp$ is an open dense subset of $\tVgeq_\cp$, and therefore 
\begin{equation}\label{eq:tVo_inside_tV}
  \tVo_\cp\subset\tVgeq_\cp \subset \tV_\cp.
\end{equation}
Thus, $(\Xbul,\Ybul)\in\tV_\cp$. We have shown that $\BigBR_\br\setminus\tTbr=\bigcup_{\cp\in\Jo} \tV_\cp$. 
\end{proof}

\subsection{Cluster variables}\label{sec:deo_cluster}
By \cref{prop:V_d_irr}, the irreducible components of $\BigBR_\br \setminus \tTbr$
 are the Deodhar hypersurfaces $\tV_\cp$, $\cp \in \Jo$. 
 For a grid minor $\grid_{c,k}$ and $\cp\in\Jo$, we denote by $\ord_{V_\cp}\Delta_{c,k}\in\Z$ the order of vanishing of $\Delta_{c,k}$ on the hypersurface $\tV_\cp$; cf. \cref{cor:grid_bigBR}. Since $\grid_{c,k}$ is regular on $\BigBR_\br$, we have that $\ord_{V_\cp}\grid_{c,k}\geq0$. In this section, we utilize properties of $\ord_{V_\cp}\grid_{c,k}$ to define a new basis of characters of $\Tbr$, the cluster variables.

We have the following basic unitriangularity property.
\begin{proposition}\label{prop:triang}
For $\cp\in\Jo$ solid, $c\in[0,m]$, and $k\in\pmn$, we have
\begin{equation}\label{eq:ord_matrix}
  \ord_{V_\cp}\grid_{c,k}=
  \begin{cases}
    0, &\text{if $\cp\leq c$;}\\
    1, &\text{if $(c,k)=(\cp-1,i_\cp)$, i.e., $\Delta_{c,k}=\Delta_{\cp}$.}\\
  \end{cases}
\end{equation}
\end{proposition}
\begin{proof}
 Suppose that $\cp\leq c$. Let $(\Xbul,\Ybul)$ be a generic point in $\tV_\cp$. Then we have $X_c \Lwrel{\wp c} Y_c$, and thus $\grid_{c,k}(\Xbul,\Ybul)\neq0$. It follows that $\ord_{V_\cp}\grid_{c,k}=0$ when $\cp\leq c$.

Suppose now that $(c,k)=(\cp-1,i_\cp)$. Recall that we have introduced an open dense subset $\tVgeq_\cp\subset\tV_\cp$ in~\eqref{eq:tVgeq_dfn}. The subset $\tTbr\cup\tVgeq_\cp\subset\tTbr\cup \tV_\cp$ is thus also open dense, and recall that for $(\Xbul,\Ybul)\in\tTbr\cup\tVgeq_\cp$, we have $X_{\cpp}\Lwrel{\wp{\cpp}} Y_{\cpp}$ for all $\cpp\geq \cp$. 

We may parameterize $\BigBR_\br$ using parameters $(\bt',X_m=Y_m)$ as in~\eqref{eq:propagate_left}. We specialize all of these parameters except for $t:=\tparp_\cp$ to some fixed generic values, and view the resulting tuple $(\Xbul,\Ybul)=(\Xbul(t),\Ybul(t))$ as a function of $t$. Since the parameters $\btgeq_\cp :=(\tparp_{\cp+1},\dots,\tparp_{m},X_m=Y_m)$ are generic, we have $(\Xbul,\Ybul)\in\tTbr\cup\tVgeq_\cp$ (and therefore $X_\cp\Lwrel{\wp \cp} Y_\cp$) for all $t\in\C$.

Assume that $i_\cp\in I$. Thus, we have $Z_\cp\in U_+\dwp \cp h U_+$ for $h = \hr_\cp\in\H$. The proof of \cref{lemma:hc} implies that $\dwp \cp^{-1}Z_\cp\in U_-h U_+$. Let us write $\dwp \cp^{-1}Z_\cp=y_- h y_+$ for $(y_-,y_+)\in U_-\times U_+$. Since $\cp\in\Jo$ is solid, we have $\wp{\cp-1}=\wp \cp$. We find
\begin{equation*}%
  \grid_{\cp-1,k}(\Xbul,\Ybul)=\Delta_{\wp{\cp-1}\om_k,\om_k} (Z_\cp z_k(t)) = \domk(\dwp{\cp}^{-1}Z_\cp z_k(t))
=\domk(y_- h y_+ z_k(t)).
\end{equation*}
Recall that $z_k(t)=x_k(t)\ds_k$. Let $\Psi:=\Phi^+\setminus\{\alpha_k\}$ and let $U_+(\Psi):=(\ds_k^{-1} U_+ \ds_k)\cap U_+$ be the corresponding \emph{root subgroup}; see~\cite[Theorem~26.3]{Hum}. We have $x_k(-t)U_+(\Psi)x_k(t)\subset U_+(\Psi)$ by~\cite[Lemma~32.5]{Hum}. Next, we have $\ds_k^{-1} U_+(\Psi) \ds_k\subset U_+(\Psi)$, since $s_k$ permutes $\Psi$. Using~\eqref{eq:U_factorization}, we can factorize $y_+=x_k(p) y_+'$ for some $p\in\C$ and $y_+'\in U_+(\Psi)$. (Here $p$ depends only on the parameters in $\btgeq_\cp$ and not on $t$.) We therefore get $y_+'x_k(t)\ds_k\in x_k(t)\ds_k U_+$. Using~\eqref{eq:delta(yxy)}, we get
\begin{equation*}%
  \grid_{\cp-1,k}(\Xbul,\Ybul)=\domk(y_- h x_k(p)y_+' x_k(t)\ds_k) = \domk(h x_k(p+t)\ds_k).
\end{equation*}
It is clear that if $p+t=0$ then $\grid_{\cp-1,k}(\Xbul,\Ybul)=0$. If $p+t\neq0$, applying the first identity in~\eqref{eq:SL2_xs_sx} to $x_k(p+t)\ds_k$ and using~\eqref{eq:domk(xh)}, we find
\begin{equation}\label{eq:grid=p+t}
  \grid_{\cp-1,k}(\Xbul,\Ybul)=(p+t)\domk(h).
\end{equation}
Thus,~\eqref{eq:grid=p+t} holds regardless of whether $p+t=0$, and we have $p+t=0$ if and only if the condition $X_{\cp-1} \Lwrel{\wp{\cp-1}} Y_{\cp-1}$ fails, i.e., $(\Xbul,\Ybul)\in\tVgeq_\cp$. By~\eqref{eq:grid=p+t}, since $\domk(h)\neq0$, we have $p+t=0$ if and only if $\grid_{\cp-1,k}(\Xbul,\Ybul)=0$. Since $\grid_{\cp-1,k}$ is of degree $1$ in $t$, we find that $\ord_{V_\cp}\grid_{\cp-1,k}\leq 1$. On the other hand, we have shown that $\grid_{\cp-1,k}$ vanishes on $\tVgeq_\cp$, and thus on $\tV_\cp$ (cf.~\eqref{eq:tVo_inside_tV}), so $\ord_{V_\cp}\grid_{\cp-1,k}\geq 1$.
\end{proof}

The integers $\ord_{V_\cp} \grid_{c, k}$ are nonnegative. Our next result shows that whether $\ord_{V_\cp} \grid_{c, k}$ is zero or positive is determined by the almost positive subexpression $\bu\apd{\cp}$.  The stronger result that $\ord_{V_\cp} \grid_{c, k} \in \{0,1\}$ holds when $G =\SL_n$~\cite[Proposition~7.10]{GLSBS1}.  The precise value of $\ord_{V_\cp} \grid_{c, k}$ for $G$ of arbitrary type is given in \cref{sec:comb-algor}.

\begin{proposition} \label{prop:APS-gives-var-nonappearance} 
For all $c\in[0,m]$, $\cp\in\Jo$, and $k\in I$, we have 
\begin{equation*}%
  \ord_{V_\cp} \grid_{c, k}=0 \quad \Longleftrightarrow \quad u\pd{c} \omega_k = u\pd{c}\apd{\cp}\omega_k
\quad\text{and} \quad
  \ord_{V_\cp} \grid_{c, -k}=0 \quad \Longleftrightarrow \quad u\pd{c}^{-1} \omega_k = (u\pd{c}\apd{\cp})^{-1}\omega_k.
\end{equation*}
\end{proposition}
\begin{proof}
Let $\leq$ denote the Bruhat order on $W$ and the quotient order on the orbit $W\om_k$ of the fundamental weight $\om_k$.  Comparing \cref{dfn:pds,def:almostPos}, we see that $\up c\leq \vp c\apd \cp$ and $\wp c\geq \wp c\apd \cp$ for all $c\in[0,m]$.  Thus
$  \wp c\om_k\geq \wp c\apd \cp\om_k$ for all $c\in[0,m]$ and $k\in I$.
For  $(\Xbul,\Ybul)\in\tVo_\cp$, we have $Z_c\in\Bruho_{\wp c\apd \cp}\subset \Bruh_{\wp c}$ for all $c\in[0,m]$, because $\wp c\apd \cp\leq \wp c$. Recall that $\grid_{c,k}(\Xbul,\Ybul)=\Delta_{\wp c\omega_k, \omega_k}(Z_c)$. It is well known that the function $\Delta_{\wp c\omega_k, \omega_k}(Z_c)$, when restricted to $Z_c\in \Bruh_{\wp c}$, does not vanish at $Z_c\in\Bruho_{\wp c\apd \cp}$ if and only if $\wp c\apd \cp\om_k=\wp c \om_k$; this can be shown by e.g. adapting the proof of~\cite[Proposition~2.4]{FZ_double}. Similarly, we consider $\grid_{c,-k}(\Xbul,\Ybul)=\Delta_{\wo \omega_k, u\pd{c}^{-1}\omega_k}(Z_c)$ and observe that this function does not vanish at $Z_c\in\Bruho_{\wo\up c\apd \cp}\subset\Bruh_{\wo\up c}$ if and only if $u\pd{c}^{-1} \omega_k = (u\pd{c}\apd{\cp})^{-1}\omega_k$.
\end{proof}

\begin{corollary}\label{cor:upper-tri-change-of-basis}
The $\Jo\times\Jo$ matrix $\ordM=(\ord_{V_\cp} \grid_c)_{c,\cp\in\Jo}$ is upper unitriangular.
\end{corollary}
\noindent Inverting the matrix $\ordM$, we arrive at the following definition, which is crucial for our analysis; cf. \cref{propdef:intro_cluster}. Recall from \cref{prop:deodharIsTorus}  that a character on $\Tbr$ is just a Laurent monomial in the solid chamber minors $\{\crossing{c}\}_{c\in\Jo}$.

\begin{definition}
For $c\in \Jo$, the \emph{cluster variable} $x_c$ is the character of $\Tbr$ satisfying 
\begin{equation}\label{eq:ord_V_cp_xc}
  \ord_{V_\cp}x_c=
  \begin{cases}
    1, &\text{if $c=\cp$,}\\
    0, &\text{otherwise,}
  \end{cases} \quad\text{for all $\cp\in\Jo$}
\end{equation} 
which can be written explicitly as a Laurent monomial in the chamber minors using $\ordM^{-1}$ (see \cref{cor:upper-tri-change-of-basis}). We denote the \emph{cluster} by $\xbr=\{x_c\}_{c \in \Jo}$.
\end{definition}
 The conditions~\eqref{eq:ord_V_cp_xc} are equivalent to 
\begin{equation}\label{eq:grid_vs_x_c}
  \grid_{c,k} = \prod_{\cp\in\Jo} x_\cp^{\ord_{V_\cp}\grid_{c,k}} \quad\text{for all $c\in\Jo$ and $k\in\pmn$}.
\end{equation}

\begin{corollary}\ \label{cor:cluster_vars}
\begin{enumerate}
\item For each $c\in\Jo$, there exists a unique character $x_c\in\chL\Tbr$ satisfying the conditions~\eqref{eq:ord_V_cp_xc}.
\item For $c\in\Jo$, the character $x_c\in\chL\Tbr$ extends to a regular function on $\BR_\br$, $\bigBR_\br$, $\BigBR_\br$. 
\item\label{item:frozen_invertible} For $c\in\Jfro$, the frozen cluster variable $x_c$ is invertible in $\C[\BR_\br]$ and $\C[\bigBR_\br]$.
\item The cluster variables in $\xbr$ are irreducible and algebraically independent.
\end{enumerate}
\end{corollary}
\begin{proof}
The existence and uniqueness of $x_c$ follows from the invertibility of $\ordM$: writing $\ordM^{-1}=(m_{\cp,c})_{\cp,c\in\Jo}$, we have $x_c = \prod_{\cp\in\Jo} \crossing{c}^{m_{\cp,c}}$. 
For $c\in\Jo$, $x_c$ is a rational function on $\BigBR_\br$ which does not have a pole on $\tTbr$ or on $\tV_\cp$ for all $\cp\in\Jo$. This implies that $x_c$ is regular on $\BR_\br$, $\bigBR_\br$, $\BigBR_\br$. By \cref{prop:V_d_irr}, $x_c$ is irreducible. Since $\xbr=\{x_c\}_{c\in\Jo}$ is a basis of the character lattice of $\Tbr$, we see that the cluster variables in $\xbr$ are algebraically independent. Finally, for $c\in\Jfro$, the function $1/x_c$ is regular on $\bigBR_\br\setminus \tVo_c$, and as we mentioned after \cref{dfn:V_d}, we have $\tVo_c\cap \bigBR_\br=\emptyset$ for $c\in\Jfro$.
\end{proof}

Recall from \cref{prop:triang} that $\ord_{V_\cp}\grid_{c,k}$ can only be nonzero when $\cp>c$; see also~\eqref{eq:grid_vs_x_c}. The next result follows from the parametrization \eqref{eq:propagate_left} and will be used later in the proof. We denote $\ord_{V_\cp}\grid_{c,k}$ by $\ordbr_{V_\cp} \grid_{c,k}$ to emphasize dependence on $\br$.
\begin{lemma} \label{lem:extend-word}
The integer $\ord_{V_\cp}\grid_{c,k}$ only depends on $i_{c+1},\dots,i_m$. That is, suppose that $\br=i_1i_2\cdots i_m$ and $\br'=i'_1i'_2\cdots i'_{m'}$ are two double braid words in the alphabet $\pmn$ such that for some $c\in[m]$ and $c'\in[m']$ (with $m-c=m'-c'$), we have $i_{c+1}\cdots i_m=i'_{c'+1}\cdots i'_{m'}$. Then we have
\begin{equation*}%
\ordbr_{V_\cp}\grid_{c,k}=\ordbrp_{V_{\cp'}}\grid_{c',k}
\end{equation*} 
 for all $k\in\pmn$, $\cp>c$, and $\cp'>c'$ such that $m-\cp=m'-\cp'$.
\end{lemma}

\subsection{A two-form on the braid variety, and a seed} \label{sec:deo_form}
We now introduce a two-form, which, together with the cluster variables, determines an exchange matrix via~\eqref{eq:intro:omcoef_dfn}. At the end of the section, we put everything together to define a seed for the braid variety.

We start by introducing a family of $1$-forms on $\Tbr$. 

For $i,j \in \pm I$, recall that $a_{ij} = 0$ if $i,j$ have different signs, and $a_{ij}= a_{(-i)(-j)}$ otherwise, and that $d_i := d_{|i|}$.  For each $c\in[0,m]$ and $i\in \pm I$, we set
\begin{equation}\label{eq:Lci-defn}
  \Lci ci:=\dlog \left(\prod_{k \in \pm I} \grid_{c,k}^{a_{i k}}\right)
=\frac{1}{2} \sum_{k \in \pm I} a_{i k} \dlog \grid_{c,k}
= \frac{1}{2} \sum_{k \in \pm I,\ \cp\in\Jo}a_{i k} ~\left(\ordbr_{V_\cp}\grid_{c,k} \right)~\dlog x_\cp.
\end{equation}
Consider the following $2$-forms on $\Tbr$:
\begin{equation}\label{eq:ombr_dfn}
  \ombrc:=\sign(i) ~d_i ~\Lci{c-1}i \wedge \Lci ci \quad\text{for $c\in[m]$ and $i:=i_c$,} \qquad\text{and}\qquad \ombr:=\sum_{c\in [m]} \ombrc.
\end{equation}
Note that $\dlog \grid_{c-1,j} \wedge \dlog \grid_{c,k}$ only contributes to $\ombrc$ if $j=k$ or $|j|,|k|$ are adjacent in the Dynkin diagram for $G$. Written in terms of $\dlog \grid_{c-1,j} \wedge \dlog \grid_{c,k}$, $\ombrc$ is essentially the same as \cite[Definition of $\epsilon_{(n)}$]{ShWe}, which generalizes \cite[Definitions 2.2, 2.3]{BFZ}.

\begin{remark}\label{rmk:constant_coefs}
By~\eqref{eq:Lci-defn}, the form $\ombr$ has constant coefficients when expressed in terms of $(\dlog x_c\wedge \dlog_\cp)_{c,\cp\in\Jo}$.
\end{remark}

Since $\Tbr$ is open dense in $\BR_\br$, the forms $\ombr$ and $\ombrc$ are rational $2$-forms on $\BR_\br$. Though it is not apparent from the above formula, it will follow from our main result (\cref{thm:main}) combined with the results of~\cite{Muller_WP_form} that $\omega_\br$ extends to a regular $2$-form on the entire $\BR_\br$.

Recall that~\eqref{eq:grid_mon_reln} holds for $i_c\in \pmn$. Taking $\dlog$ of both sides of~\eqref{eq:grid_mon_reln}, we get 
\begin{equation}\label{eq:L_mon_reln}
  \Lci{c-1}{i_c}+\Lci c{i_c} = 0 \quad\text{if $c$ is hollow.}
\end{equation}
Thus, $\ombrc=0$ for all $c\in[m]\setminus\Jo$, which implies the following result.
\begin{corollary}\label{cor:hollow}
We have $\ombr=\sum_{c \in \Jo} \ombrc$.
\end{corollary}

The next proposition follows immediately from \cref{cor:deo_tor,cor:cluster_vars}. 

\begin{proposition}
	Let $\br \in (\pmn)^m$ be such that $\Demprod(\br) = \wo$. 
 The tuple
	\begin{equation}\label{eq:sig-br-def}
		\Sbr:=(\Tbr, \xbr, \dbr, \ombr)
	\end{equation}
	is an abstract seed on $\BR_{\br}$ in the sense of \cref{dfn:abstract_seed} below. 
\end{proposition}

\section{Cluster algebras}

Cluster algebras were discovered by Fomin and Zelevinsky \cite{FZ}.  We consider skew-symmetrizable cluster algebras, relying on formalism similar to \cite{FoGo_ensembles}.

\subsection{Background}\label{sec:cluster_backgr}

\begin{definition}\label{dfn:abstract_seed}
	A rank $n$ and dimension $n+m$ \emph{(abstract) seed} is a quadruple $\Sigma = (T, \x, \d, \omega)$, where
	\begin{enumerate}
		\item $T$ is a complex algebraic torus of dimension $\npm$,
		\item $\x =  (x_1,\ldots,x_{\npm})$ is an ordered basis of $\chL{T}$, where $x_1,\ldots,x_n$ (resp., $x_{n+1},\ldots,x_{\npm}$) are mutable (resp., frozen) variables,
		\item $\d = (d_1,\ldots,d_{\npm})$ is a collection of positive integers,
		\item $\omega$ is a 2-form on $T$ of the form
\begin{equation}\label{eq:om_dfn_cluster}
\omega =\sum_{i \leq j} d_j \tB_{ij} \dlog x_i \wedge \dlog x_j= \sum_{i \leq j} d_i \tB_{ji} \dlog x_j \wedge \dlog x_i,
\end{equation}
where $\tB_{ii}=0$ for $i\in[n+m]$ and $\tB_{ij}\in\Q$ for all $i,j\in[\npm]$. 
	\end{enumerate}
\end{definition}

The matrix 
$\tB=(\tB_{ij})_{(i,j)\in[\npm]\times[n]}$
 is the usual $(\npm) \times n$ \emph{extended exchange matrix} in the theory of cluster algebras. 
We have $d_j \tB_{ij} = -d_i\tB_{ji}$ for $i,j \in [\npm]$; in particular, the top $n \times n$ \emph{principal part} $B$ of $\tB$ is skew-symmetrizable. %

\begin{definition}\label{dfn:mut}
	Let $\Sigma = (T,\x,\d,\omega)$ be a seed and $k$ a mutable index. We say that $\Sigma$ is \emph{integral at $k$} if $\tB_{jk}\in\Z$ for all $j\in[\npm]$. In this case, we define
\begin{equation}\label{eq:mut_dfn}
  x_k':=\frac{\prod_{\tB_{jk}>0} x_j^{\tB_{jk}} + \prod_{\tB_{jk}<0} x_j^{-\tB_{jk}}}{x_k}.
\end{equation}
   The \emph{mutation} of $\Sigma$ in the direction $k$ is the seed $\mu_k(\Sigma) = (T',\x',\d,\omega')$ where $T'$ is the algebraic torus with basis of characters $\x'= (x_1,\ldots,x'_k,\ldots,x_{\npm})$ and the 2-form $\omega'$ on $T'$ is the pullback of $\omega$ via the natural rational map $T' \to T$.
\end{definition}
\noindent We say that $\Sigma$ is \emph{integral} if it is integral at each $k\in[n]$, i.e., if 
 $\tB_{ik} \in \Z$ for all $i\in[\npm]$ and $k\in[n]$.
\begin{remark}
As discussed in~\cite{FoGo_ensembles}, $\omega'$ may be expressed in the form~\eqref{eq:om_dfn_cluster} using another matrix $\tB' = \mu_k(\tB)$ obtained from $\tB$ via the usual cluster mutation of exchange matrices as defined in e.g.~\cite[Definition~2.7.6]{FWZ_book}. In particular, if $\Sigma$ is integral (resp., integral at $k$) then so is $\mu_k(\Sigma)$.
\end{remark}

For the rest of this subsection, we assume that all seeds are integral.  Following~\cite[Section~5.1]{LS}, a seed $\Sigma$ is \emph{full rank} if $\tB$ has rank $n$ and is \emph{really full rank} if the rows of $\tB$ span $\Z^n$ over $\Z$. We will prove the seeds $\Sbr$ from~\eqref{eq:sig-br-def} are really full rank in \cref{cor:really_full_rank}. Note that begin (really) full rank is invariant under mutation.

Let $X$ be an irreducible complex algebraic variety of dimension $\npm$.  A \emph{seed on $X$} is an abstract seed $\Sigma = (T, \x, \d, \omega)$ together with an identification $T \subset X$ of $T$ with an open dense subset of $X$. The inclusion $T \hookrightarrow X$ induces an identification of the field $\C(X)$ of rational functions on $X$ with the field of rational functions $\C(\x) := \C(x_1,\ldots,x_\npm)$ in the initial cluster variables $\x$. In practice, we abuse notation and write $\x$ for a tuple of elements in $\C(X)$.

The \emph{cluster algebra} $\A(\Sigma)$ is the subring of $\C(\x)$ generated by all cluster variables together with inverses of frozen variables. We let $\Vcal(\Sigma) := \Spec(\A(\Sigma))$ denote the \emph{cluster variety}.  We say that \emph{$(X, \Sigma)$ is a cluster variety} if $X$ is an affine variety and the coordinate ring $\C[X]$ is identified with $\A(\Sigma)$ under the identification $\C(X) \cong \C(\x)$.

We will need the following property of cluster variables.
\begin{proposition}[{\cite[Theorem~3.1]{GLS}}]\label{prop:irreducible}
Each cluster variable is an irreducible element of $\Acal(\Sigma)$. %
\end{proposition}

Our proofs will utilize some notions on cluster algebras that we now recall.

\begin{definition}\label{defn:freezing}
  Let $\Sigma$ be an abstract seed of rank $n$ and dimension $\npm$, and let $\Iset\subset[n]$. The \emph{freezing of $\Sigma$ at $\Iset$}, denoted $\ifreeze{\Sigma}[\Iset]$, is the seed obtained from $\Sigma$ by declaring the variables $\{x_c\}_{c \in \Iset}$ to be frozen.
  It is a seed of rank $n-|\Iset|$ and dimension $\npm$. For $k\in[n]$, we denote $\ifreeze\Sigma[k]:=\ifreeze\Sigma[\{k\}]$. 
\end{definition}

To a seed $\Sigma$ we associate a directed graph $\Digrt(\Sigma)$ with vertex set $[n+m]$ and an arrow $i\to j$ whenever $\tB_{ij}>0$. We let $\Digr:=\Digr(\Sigma)$ be the \emph{mutable part} of $\Digrt(\Sigma)$, i.e., the induced subgraph of $\Digrt(\Sigma)$ with vertex set $[n]$. We say that a mutable index $\si\in[n]$ is a \emph{sink} if it has no outgoing arrows in $\Digr$. Let $\ND$ denote the set of vertices of $\Digr$ having an arrow to $\si$, and denote $\NDH:=\ND\cup\{\si\}$.  The following definition is a variation of locally-acyclic seeds~\cite{Mul} and Louise seeds~\cite{LS}; see also~\cite[Section~5.4 and Remark~5.14]{GL_plabic}.

\begin{definition}\label{def:sinkrec}
	The class of \emph{\sinkrec seeds} is defined recursively as follows.
	\begin{itemize}
		\item Any seed $\Sigma$ such that $\Digr(\Sigma)$ has no arrows is \sinkrec.
		\item Any seed that is mutation equivalent to a \sinkrec seed is \sinkrec.
		\item Suppose that $\Sigma$ is a seed with a sink $\si\in[n]$ such that the seeds $\ifreeze\Sigma[\si]$ and $\ifreeze\Sigma[\NDSH]$ are \sinkrec. Then $\Sigma$ is \sinkrec.
	\end{itemize}
\end{definition}

The \emph{upper cluster algebra} $\Ucal(\Sigma)\subset\C[\x^{\pm1}]$ is the intersection of Laurent polynomial rings $\cap_{\Sigma'}\C[\bx'^{\pm 1}]$, wher $\bx'$ is the cluster of $\Sigma'$ and the intersection is over all $\Sigma'$ which can be obtained from $\Sigma$ by a sequence of mutations. If $\Sigma$ is full rank, then $\Ucal(\Sigma)=\C[\x^{\pm1}]\cap \bigcap_{k\in[n]} \C[\mu_k(\x)^{\pm1}]$~\cite[Corollary 1.9]{BFZ}.
\begin{proposition}\label{prop:A=U}
Suppose that $\Sigma$ is a \sinkrec seed. Then $\Acal(\Sigma)=\Ucal(\Sigma)$.
\end{proposition}
\begin{proof}
It follows from induction and~\cite[Lemma~5.3]{Mul} that \sinkrec seeds are locally acyclic in the sense of~\cite{Mul,MullerLA2}.  By~\cite[Theorem~2]{MullerLA2}, we have $\Acal(\Sigma)=\Ucal(\Sigma)$. 
\end{proof}

We will also need the notion of quasi-equivalence of seeds, which was first studied in~\cite{Fraser}. We adapt the definition of \cite{Fraser} to our conventions.
\begin{definition}\label{dfn:quasi_eq}
Two seeds $\Sigma=(T,\x,\d,\om)$ and $\tilde\Sigma=(\tilde T,\tilde\x,\tilde\d,\tilde\om)$ of rank $n$ and dimension $\npm$ are \emph{quasi-equivalent}, denoted $\Sigma \quasi \tilde\Sigma$, if the following conditions are satisfied:
\begin{enumerate}[leftmargin=30pt]
\item\label{quasi1} $T=\tilde T$, $\d=\tilde\d$, $\om=\tilde\om$;
\item\label{quasi2} the sublattice of $\chL T$ spanned by the frozen variables $x_{n+1},\dots,x_{\npm}$ coincides with the sublattice spanned by $\tilde x_{n+1},\dots,\tilde x_{\npm}$;
\item\label{quasi3} for each $k\in[n]$, we have $\tilde x_k=x_k M_k$, where $M_k$ is a Laurent monomial in $x_{n+1},\dots,x_{\npm}$.
\end{enumerate}
\end{definition}

It is easy to see that if $\Sigma$ is integral and $\Sigma\quasi \tilde \Sigma$ then $\tilde\Sigma$ is integral. 
 The following is also straightforward to check.

\begin{lemma}\label{lem:quasimutate}
If $\Sigma $ and $ \tilde \Sigma$ are quasi-equivalent seeds then $\mu_k(\Sigma) \quasi \mu_k(\tilde{\Sigma})$ for all mutable $k$.
\end{lemma}

\begin{corollary}\label{cor:quasi}
Suppose two seeds $\Sigma,\tilde\Sigma$ are quasi-equivalent.
  Then they define the same cluster algebra $\A(\Sigma) = \A(\Sigma') \subset \C(T)$.
\end{corollary}
\begin{proof}
It follows from \cref{lem:quasimutate} that each cluster variable in $\A(\Sigma)$ differs from the corresponding cluster variable in $\A(\tilde\Sigma)$ by a factor equal to a Laurent monomial in the frozen variables. 
\end{proof}

\subsection{Deletion-contraction}\label{sec:dc}

We give an inductive criterion for a pair $(X,\Sigma)$ to be a sink-recurrent cluster variety, which is a key part of our proof of \cref{thm:main}. In \cref{sec:braid_dc}, we will apply this criterion to the seeds $\Sbr$ from~\eqref{eq:sig-br-def}.  See \cite[Corollary 5.15]{GL_plabic} for a different application suggesting our nomenclature. %

\begin{assumption}\label{assn:dc}
For the remainder of this section, we let $\Sigma = (T,\x,\d,\omega)$ be an abstract seed of rank $n$ and dimension $\npm$. Let  $\Digr:=\Digr(\Sigma)$. We assume that $\Sigma$ is \sinkrec, with sink $\si$ in $\Digr$ such that $\Sigma$ is integral at $\si$. Further, we assume that there exists a frozen index $\froi$ such that $\tB_{\froi \si}=\pm1$ and $\tB_{\froi j}=0$ for $j\in[n]\setminus\{\si\}$. 
 Suppose that the exchange relation for $x_\si$ in $\Sigma$ is given by $x_\si x_\si'=M_1+x_\froi M_2$ for some monomials $M_1,M_2$ in $\{x_j\}_{j\in[n+m]\setminus\{\si,\froi\}}$.
\end{assumption}

\begin{definition}
Suppose that $\si$ has $q:=|\ND|$ mutable neighbors.  
 The \emph{contraction} $\icon\Sigma[\si]=(\icon T[\si],\icon\x[\si],\icon\d[\si],\icon\om[\si])$ is a seed of rank $n-q-1$ and dimension $\npm-2$ defined as follows. 

\begin{enumerate}
\item $\icon\x[\si]$ is obtained from $\x$ by omitting $x_\si$ and $x_\froi$ and declaring the indices in $\ND$ to be frozen.
\item $\icon T[\si]$ is an algebraic torus with character lattice generated by $\icon\x[\si]$.
\item $\icon\d[\si]$ is obtained by restricting the sequence $\d$ to the set $[\npm]\setminus\{\si,\froi\}$. 
\item $\icon\om[\si]$ is obtained from $\om$ by writing it in the form~\eqref{eq:om_dfn_cluster} and substituting $\dlog x_\si:=0$ and $\dlog x_\froi:=\dlog M_1-\dlog M_2$.
\end{enumerate}

The \emph{deletion}\footnote{The terminology ``deletion-contraction'' comes from an analogous construction for matroids and hyperplane arrangements. Note that, in the case of cluster seeds, ``deletion'' corresponds to \emph{freezing} $x_s$ (thus deleting $s$ from the mutable part $\Digr$) and not to deleting $x_s$ from the seed.}
 $\ifreeze\Sigma[\si]$ is the seed of rank $n-1$ and dimension $\npm$ obtained by declaring $x_\si$ to be frozen (cf. \cref{defn:freezing}).
\end{definition}

In our next result, we use the following notation. If the same cluster variable $x$ is viewed as a function on two different varieties $U$ and $U'$ (for example, $U = \Vcal(\Sigma)$ and $U' = \Vcal(\ifreeze\Sigma[\si])$), we denote the corresponding two functions by $x|_{U}$ and $x|_{U'}$. 
\begin{theorem}[Deletion-contraction recurrence]\label{thm:abstract-clust-stuc-del-con}
Let $X$ be an affine, normal, irreducible, complex algebraic variety, and let $\Sigma=(T,\x,\d,\omega)$ be a seed on $X$ with a sink $\si$ satisfying \cref{assn:dc}. Assume that all cluster variables in $\x\sqcup\{x_\si'\}$ are regular on $X$. Define subvarieties $W:=\{x_\si\neq0\}$ and $V:=\{x_\si=0\}$ of $X$.
 Suppose that we have isomorphisms $W \cong \Vcal(\ifreeze\Sigma[\si])$ and $V \cong V_1\times V_2$ with $V_1 := \Spec(\C[x'_\si])\cong\C$ and $V_2 := \Vcal(\icon\Sigma[\si])$. Let $\Vproj_1:V\to V_1$, $\Vproj_2:V\to V_2$, $\iota_W:W \hookrightarrow X$, and $\iota_V:V \hookrightarrow X$ denote the natural projections and inclusions. Suppose that:
\begin{itemize}
\item for each cluster variable $x$ of $\ifreeze\Sigma[\si]$, we have $\iota_W^\ast(x|_X) = x|_W$;
\item for each cluster variable $x$ of $\icon\Sigma[\si]$, we have $\iota_V^\ast(x|_X) = \Vproj_2^\ast(x|_{V_2})$;
\item for the cluster variable $x_\si'$, we have $\iota_V^\ast(x_\si'|_X) = \Vproj_1^\ast(x_\si'|_{V_1})$.
\item the seed $\ifreeze\Sigma[\si]$ is (really) full rank.
\end{itemize}
 Then $(X,\Sigma)$ is a cluster variety and $\Sigma$ is (really) full rank.
\end{theorem}

\begin{proof}
First, the condition that $\ifreeze\Sigma[\si]$ is integral is included in the assumption that $(W,\ifreeze\Sigma[\si])$ is a cluster variety. Since $\Sigma$ is integral at $\si$, this implies
$\Sigma$ is also integral.

Second, the assumption that $\ifreeze\Sigma[\si]$ is (really) full rank together with \cref{assn:dc} imply that $\Sigma$ is (really) full rank. Indeed, row $\froi$ of the exchange matrix of $\Sigma$ contains a single nonzero entry equal to $\pm1$ in column $\si$. The exchange matrix of $\ifreeze\Sigma[\si]$ is obtained from that of $\Sigma$ by removing column $\si$. 

Let $j\in[n]\setminus\{\si\}$ be a mutable index. Clearly, the (pullback under $\iota_W$ of the) exchange relation for $x_j$ in $\Sigma$ coincides with the exchange relation for $x_j$ in $\ifreeze\Sigma[\si]$. Thus, the mutated variable $x_j'$ is regular on~$W$. Next, assume that $j\notin \ND$. By assumption, $j$ is not connected to $\si,\froi$ in $\Digr$, and thus the terms involving $\dlog x_j$ are unchanged when passing from $\om$ to $\icon\om[\si]$. Thus, the pullback of the exchange relation for $x_j$ under $\iota_V$ is still the exchange relation for $x_j$ in $\icon\Sigma[\si]$, and therefore the mutated variable $x_j'$ is regular on $V$. For $j\in\ND$, we claim that $x_j'$ must also be regular on $V$. Indeed, by~\eqref{eq:mut_dfn}, $x_j'$ is regular on $V$ if $x_j^{-1}$ is regular on $V$. But $\iota_V^\ast(x_j|_X) = \Vproj_2^\ast(x_j|_{V_2})$, and $x_j|_{V_2}$ is a frozen variable in $\icon\Sigma[\si]$, so indeed $x_j$ is invertible on $V$. It follows that for all $j\in[n]\setminus\{\si\}$, the mutated variable $x_j'$ is a regular function on $X$ since it is regular on both $V$ and $W$. For $j=\si$, $x_\si'$ is regular on $X$ by assumption.

Next, we show that $\C[X]\subset \Ucal(\Sigma)$. As $\Sigma$ is full-rank, \cite[Corollary 1.9]{BFZ} implies that $\Ucal(\Sigma) = \C[\bx^{\pm 1}] \cap \bigcap_{k\in[n]} \C[\mu_k(\bx)^{\pm 1}]$. The containment $\C[X]\subset \Ucal(\Sigma)$ is equivalent to constructing inclusions $T\hookrightarrow X$ and $\mu_j(T)\hookrightarrow X$ for all $j\in[n]$. Since $\Sigma$ is a seed on $X$, we have $T\subset X$.  For the tori $\mu_j(T)$, we show that the subset $X_j\subset X$ where the regular functions in $\mu_j(\x)$ are all non-vanishing is isomorphic to an algebraic torus $\mu_j(T)\cong (\Cast)^{\npm}$ via the map $\varphi_j:X_j\to\mu_j(T)$ sending $y\in X_j$ to $z:=(x_1(y),\dots,x_j'(y),\dots,x_{\npm}(y))$.  If $j\in[n]\setminus\si$ then we have $X_j\subset W$, and thus the statement follows since $W$ is a cluster variety. So let $j=\si$. Consider the torus $\mu_\si(T)\cong (\Cast)^{\npm}$. Let $p:=M_1+x_\froi M_2\in\C[X]$ be the exchange binomial for $x_\si$ (cf. \cref{assn:dc}). Since $p$ does not involve $x_\si$ and $x_\si'$, we can also view $p$ as a regular function on $\mu_\si(T)$ compatible with pullback under $\varphi_\si$. Let $z=(z_1,\dots,z_{\npm})\in\mu_\si(T)$. Our goal is to show that $z$ has a unique preimage under $\varphi_\si$. Suppose first that $p(z)\neq0$. Then $\varphi_\si^{-1}(z)\subset T$, and the result follows. Suppose now that $p(z)=0$. Then $\varphi_\si^{-1}(z)\subset V$. Recall that $V\cong V_1\times V_2$. Since $x'_\si=z_\si$, the first coordinate $\Vproj_1\circ\varphi_\si^{-1}(z)$ of the preimage
is uniquely determined by $z$. The second coordinate $\Vproj_2\circ\varphi_\si^{-1}(z)$ of the preimage is uniquely determined by $(z_i)_{i\in[\npm]\setminus\{\si,\froi\}}$. We have shown that $z$ has a unique preimage under $\varphi_\si$, which completes the proof of the inclusion $\C[X]\subset \Ucal(\Sigma)$. The statement of the theorem now follows from \cref{prop:BFZ} below. 
\end{proof}

\begin{proposition}\label{prop:BFZ}
Let $X$ be an affine, factorial, irreducible, complex algebraic variety, and let $\Sigma=(T,\x,\d,\omega)$ be an integral sink-recurrent seed on $X$.
  Suppose that $\C[X]\subset \Ucal(\Sigma)$. 
  Then $(X,\Sigma)$ is a cluster variety.
\end{proposition}
\begin{proof}
The inclusions $\C[X] \subset \C[x_1^{\pm 1},\ldots,(x'_j)^{\pm 1},\ldots,x_\npm^{\pm 1}]$ give tori $\mu_j(T)\cong X_j \subset X$ as in the proof of \cref{thm:abstract-clust-stuc-del-con}. By a standard argument, this implies that 
 the complement of $T \cup \bigcup_{j\in[n]} X_j$ has codimension greater than or equal to two in $X$; see~\cite[Section 3]{Zel}, \cite[Proof of Theorem~2.10]{BFZ}, or~\cite[Lemmas~9.5--9.8]{GLSBS1}.
 Since $X$ is factorial and in particular normal, we have $\C[X] = \C[T \cup \bigcup_{j\in[n]} X_j] \supset \Ucal(\Sigma)$. We have the reverse containment by assumption, so $\C[X] = \Ucal(\Sigma)$.  By assumption, $\Sigma$ is sink-recurrent, so \cref{prop:A=U} implies $\Acal(\Sigma) = \Ucal(\Sigma) = \C[X]$.
\end{proof}

\section{Double braid moves}\label{sec:2braid}

Recall from \cref{sec:deodhar-geometry} that we have constructed a single seed $\Sbr$ for each double braid variety $\BR_{\br}$. In this section, we first study natural isomorphisms between braid varieties corresponding to \emph{double braid moves}, and show that pullbacks along them are well-behaved on seeds. In \cref{sec:braid_dc}, we show how to apply \cref{thm:abstract-clust-stuc-del-con} on deletion-contraction to $\BR_{\br}$ for $\br$ of a particular form. Finally, in \cref{subsec:pf_thm_main_simply_laced}, we use double braid moves and deletion-contraction to prove \cref{thm:main} in simply-laced types.

Double braid moves are defined as follows:

\begin{enumerate}[label=\nrmfont{(B\arabic*)}]%
	\item \label{bm1} $i j \leftrightarrow j i$ if $i, j\in\pmn$ have different signs;
	\item \label{bm2} $i j \leftrightarrow j i$  if $i, j\in\pmn$ have the same sign and $(s_{|i|}s_{|j|})^2 = 1$;
	\item \label{bm3} $\underbrace{iji\dots}_{\text{$m_{ij}$ letters}}  \leftrightarrow \underbrace{jij\dots}_{\text{$m_{ij}$ letters}}$  if $i, j\in\pmn$ have the same sign and $(s_{|i|}s_{|j|})^{m_{ij}}=1$ with $m_{ij}\geq3$;
	\item \label{bm4} $\br_0i\leftrightarrow \br_0(-i^\ast)$ for $i\in \pmn$ and $\br_0\in(\pmn)^{m-1}$;
\item \label{bm5} $i\br_0\leftrightarrow (-i)\br_0$ for $i\in \pmn$ and $\br_0\in(\pmn)^{m-1}$.
\end{enumerate}

If double braid words $\br$ and $\br'$ are related by one of the moves~\ref{bm1}--\ref{bm5}, there is a natural isomorphism $\mis:\BR_\br\xrasim\BR_{\br'}$, discussed below.

\begin{definition}\label{dfn:mis_B1_B3}
Suppose that $\br$ and $\br'$ are related by one of the moves~\ref{bm1}--\ref{bm3}. If this move involves indices $l,l+1,\dots,r$, the isomorphism $\mis$ sends $(\Xbul,\Ybul)\in\BR_\br$ to the unique tuple $(\Xbul',\Ybul')\in\BR_{\br'}$ such that $X_c'=X_c$ and $Y_c'=Y_c$ for $0\leq c<l$ or $r\leq c\leq m$. The remaining weighted flags $X_l',\dots,X_{r-1}',Y_l',\dots,Y_{r-1}'$ are uniquely determined by \cref{lemma:rel-pos-facts}. 
\end{definition}
\noindent For the moves~\ref{bm4} and~\ref{bm5}, the isomorphism $\mis$ is described in \cref{sec:bm4,sec:bm5}, respectively.
The main result of this section is the following.

\begin{theorem}\label{thm:moves}
Suppose that $\br$ and $\br'$ are related by one of the moves~\ref{bm1}--\ref{bm5}. 
If $(\BR_\br,\Sbr)$ is a cluster variety then so is $(\BR_{\br'},\Sbrp)$. %
\end{theorem}
\noindent We then use \cref{thm:moves} and \cref{thm:abstract-clust-stuc-del-con} to prove \cref{thm:main}; see \cref{thm:br_del_con,subsec:pf_thm_main_simply_laced,subsec:pf_thm_main_mult_laced}.

The proof of \cref{thm:moves} will occupy \crefrange{sec:2braid}{sec:mult_laced}. Along the way, we will construct a seed $\Sigma'=(T',\x',\d',\om')$ obtained from $\Sbr=(T,\x,\d,\om)$ by one or several mutations, 
followed by a relabeling. We will show the following for moves~\bmref{bm1}--\bmref{bm5}:
\begin{enumerate}%
\item[{\crtcrossreflabel{\nrmfont{(F)}}[IS0]}]
  The $2$-form is invariant: $\mis^*\om_{\br'}=\ombr$. %
\item[{\crtcrossreflabel{\nrmfont{(Q)}}[IS2]}]
  Suppose that $(\BR_\br,\Sbr)$ is a cluster variety. Then the seeds $\Sigma'$ and $\mis^*\Sbrp$ are quasi-equivalent. 
\end{enumerate}
Here, for a seed $\Sbrp=(T_{\br'},\x_{\br'},\d_{\br'},\om_{\br'})$, $\mis^*\Sbrp=(T\pb,\bx\pb,\d\pb,\om\pb)$ is an abstract seed on $\BR_\br$ defined by 
\begin{equation}\label{eq:Sbrp_dfn}
  T\pb:=\mis^{-1}(T_{\br'}),\quad \x\pb:=\mis^* \x_{\br'},\quad \d\pb:=\d_{\br'}, \quad\text{and}\quad \om\pb:=\mis^\ast\om_{\br'}.
\end{equation}
Note that~\ref{IS2} immediately implies \cref{thm:moves}. %

\begin{definition}
  A \ref{bm1}--\ref{bm3} move is \emph{solid} if all indices involved are solid. For $i, j\in I$, the \ref{bm1} move $(-i) j \leftrightarrow j (-i)$ on indices $c, c+1$ is \emph{special} if $u\pd{c} s_i= s_j u\pd{c}$ and \emph{solid-special} if it is both solid and special. A \ref{bm3} move with $m_{ij}>3$ is \emph{\llong}; all other moves are \emph{\short}. Finally, a \ref{bm1}--\ref{bm5} move is a \emph{mutation move} if it involves at least one cluster mutation; otherwise it is a \emph{non-mutation move}.
\end{definition}

\begin{remark}
 As we will show in \cref{sec:B1_solid_special}, a solid-special \ref{bm1} move corresponds to a single mutation, at the rightmost index involved in the move. The move \ref{bm3} involving $q$ solid indices corresponds to a sequence of $\binom{q-1}{2}$ mutations on the rightmost $m_{ij}-2$ indices involved in the move (see \cref{sec:B3_solid,sec:mult_laced}). We will show that all other moves are non-mutation moves.
\end{remark}

We will show \ref{IS0}, \ref{IS2} for short moves directly. This will complete the proof of \cref{thm:moves} in simply-laced types. We then use this and folding to show \ref{IS0}, \ref{IS2} for long moves in \cref{sec:folding,sec:mult_laced}.

Throughout the rest of this section, we fix $\br, \br'$ related by a short move and thus an isomorphism
 $\mis: \BR_{\br} \xrasim \BR_{\br'}$.  For a rational function or a form $f$ on $\BR_{\br'}$, we use the shorthand $f\pb:=\mis^*f$. %

\begin{remark}\label{rmk:all_hollow}
If all indices involved in a move~\ref{bm1}--\ref{bm3} are hollow, then the statements \ref{IS0}, \ref{IS2} follow trivially; cf. \cref{cor:hollow}.
\end{remark}

\subsection{Mutation move: \ref{bm1}, solid-special}\label{sec:B1_solid_special}
Consider the case of a solid-special move~\ref{bm1} on indices $c,c+1$. Since both indices are solid, we denote $u:=\up{c-1}=\up{c}=\up{c+1}$ and $w:=\wp{c-1}=\wp c=\wp{c+1}$. The indices $i,j\in\pmn$ are of opposite signs; we assume that $i\in -I$ and $j\in I$ as the other case is similar. The solid-special condition yields %
\begin{equation}\label{eq:solid_spec_u_si_sj}
  u<s_{|i|}u=us_{j} \quad\text{and}\quad s_{|i|^*}w=ws_j<w.
\end{equation}

To show \ref{IS0}, we will utilize the following relation among grid minors. 

\begin{proposition}[{\cite[Theorem~1.17]{FZ_double}}]\label{prop:det_B1}
We have\footnote{Our Cartan matrix $a_{ij}:= \< \alpha_i, \ach_j \>$ is the transpose of that of~\cite{FZ_double}; see~\cite[Equation~(2.27)]{FZ_double}.}
\begin{equation}\label{eq:M1identity}
  \grid_{c, j} \grid_{c,j}\pb = \grid_{c+1, j} \grid_{c-1, j} + \prod_{k\neq j} \Delta_{c,k}^{-a_{jk}}.
\end{equation}
\end{proposition}
\begin{proof}
 We may choose $t,t'\in\C$ such that $Z_c=Z_{c+1} z_j(t)$, $Z_c\pb=\z_{|i|^*}(t')^{-1} Z_{c+1}$, and $Z_{c-1}=Z_{c-1}\pb=\z_{|i|^*}(t')^{-1} Z_{c+1} z_j(t)=\ds_{|i|^*} x_{|i|^*}(t') Z_{c+1} x_j(t) \ds_j$. Let $Z:=x_{|i|^*}(t') Z_{c+1} x_j(t)$. By \cite[Theorem~1.17]{FZ_double}, we have
\begin{equation}\label{eq:M1_FZ}
  \Delta_{w\om_j,s_j\om_j}(Z)   \Delta_{ws_j\om_j,\om_j}(Z) = 
\Delta_{ws_j\om_j,s_j\om_j}(Z) \Delta_{w\om_j,\om_j}(Z) + \prod_{k\neq j} \Delta_{ws_j\om_k,\om_k}(Z)^{-a_{jk}}.
\end{equation}
Using properties of generalized minors from \cref{sec:minors}, one can check that each term of~\eqref{eq:M1identity} equals the corresponding term of \eqref{eq:M1_FZ}. For example, we have 
\begin{equation*}%
  \grid_{c,j}=\Delta_{w\om_j,\om_j}(Z_{c+1} x_j(t) \ds_j)
=\Delta_{\om_j,\om_j}(\dw^{-1} Z_{c+1} x_j(t) \ds_j)
=\Delta_{\om_j,\om_j}(\dw^{-1} Z \ds_j)=\Delta_{w\om_j,s_j\om_j}(Z),
\end{equation*}
where we have used $\dw^{-1} x_{|i|^*}(t')\in U_- \dw^{-1}$; cf.~\eqref{eq:delta(yxy)} and~\eqref{eq:solid_spec_u_si_sj}. For $\Delta_{c,k}^{-a_{jk}}$, $k\neq j$, we additionally used that $s_j\om_k=\om_k$.
\end{proof}

We shall use the following analog of \cite[Lemma 8.10]{GLSBS1}.
\begin{lemma}\label{lemma:special} 
For $\cp\in[0,m]$ and $-i, j \in I$ such that $\up\cp s_{j}=s_{|i|}\up\cp$, we have
\begin{equation}\label{eq:special}
  \prod_{k\in \pmn} \Delta_{\cp,k}^{a_{ik}} =  \prod_{k\in\pmn} \Delta_{\cp,k}^{\eps a_{jk}} \quad\text{and}\quad \Lci \cp i=\eps \Lci \cp j, \quad\text{where}\quad \eps:=
  \begin{cases}
    1, &\text{if $\up\cp<\up\cp s_j$,}\\
    -1, &\text{if $\up\cp>\up\cp s_j$.}\\
  \end{cases}
\end{equation}
\end{lemma}
\begin{proof}
 We have 
$\alpha_j = \sum_{k \in I} a_{jk} \omega_k$
 and similarly for $\alpha_{|i|}$.  By~\eqref{eq:hc}, the first identity in~\eqref{eq:special} therefore becomes $(\hb_\cp)^{\alpha_{|i|}}=(\hr_\cp)^{\eps\alpha_j}$, which follows from the assumption 
$\up\cp \alpha_{j} = \eps\alpha_{|i|}$ together with $\hb_\cp=\up\cp\cdot \hr_\cp$; cf.~\eqref{eq:h_c_dfn}. Taking $\dlog$ of both sides, we obtain the second identity.
\end{proof}

\begin{remark}
Equations~\eqref{eq:M1identity} and~\eqref{eq:special}
 are true as stated in the case $i,-j\in I$ as well.
\end{remark}

\begin{proof}[Proof of~\ref{IS0} for~\ref{bm1}, solid-special]
Only the terms $\ombrc$ and $\ombrx{c+1}$ change when applying the move \ref{bm1}.  
 Note that we must have $d_{i} = d_j$ because the simple roots $\alpha_{|i|},\alpha_j$ are related by the action of $u\in W$ (which preserves the lengths of roots). Applying~\eqref{eq:special}, we get
\begin{align*}
\frac{1}{d_j}(\ombr-\ombrp\pb)&=\frac{1}{d_j}\left(\ombrc+\ombrx{c+1} - \ombrpc\pb-\ombrpx{c+1}\pb\right) \\
&=-L_{c-1,i} \wedge L_{c,i} +L_{c,j} \wedge L_{c+1,j} -
L_{c-1,j}\pb \wedge L_{c,j}\pb + L_{c,i}\pb \wedge L_{c+1,i}\pb\\
&=-L_{c-1,j} \wedge L_{c,j} +L_{c,j} \wedge L_{c+1,j} -
L_{c-1,j} \wedge L_{c,j}\pb + L_{c,j}\pb \wedge L_{c+1,j}\\
&= (L_{c,j} + L_{c,j}\pb) \wedge (L_{c-1,j}+L_{c+1,j}).
\end{align*}
For $\cp\in\{c-1,c,c+1\}$, let $M_\cp:=\prod_{k\neq j} \Delta_{\cp,k}^{-a_{jk}}$. Thus, $M_c$ is the third term in~\eqref{eq:M1identity}. By~\eqref{eq:h_tpar}, we have $\hr_c=\ach_j(t_{c+1})\hr_{c+1}$ and $\hb_{c-1}=\ach_{|i|}(t_c)\hb_c$. This implies that $\hr_{c-1}=\ach_j(t_ct_{c+1})\hr_{c+1}$ since $\hb_c=u\cdot \hr_c$. Thus, we have $M:=M_{c-1}=M_c=M_{c+1}$. Since $M_{c+1}\pb=M_{c+1}$, we get that $M=M_{c-1}\pb=M_c\pb=M_{c+1}\pb$. Set $A :=  \frac{\grid_{c, j}\grid_{c,j}\pb}{M}$ and $B:=\frac{\grid_{c-1, j}\grid_{c+1, j}}{M}$.  Then~\eqref{eq:M1identity} gives $A=B+1$.  Thus, $dA=dB$, and so $\dlog A\wedge \dlog B=0$. It remains to prove that $\dlog A=L_{c,j}+L_{c,j}\pb$ and $\dlog B=L_{c-1,j}+L_{c+1,j}$. Indeed, by~\eqref{eq:Lci-defn}, for $\cp\in\{c-1,c,c+1\}$, we have
\begin{equation}\label{eq:Lci_vs_grid}
  \Lci{\cp}{j} = \frac12\dlog \left(\frac{\grid_{\cp,j}^2}{M_\cp}\right) = \dlog \left(\frac{\grid_{\cp,j}}{M^{1/2}}\right)
  \quad\text{and}\quad
  \Lci{\cp}{j}\pb = \frac12\dlog \left(\frac{(\grid_{\cp,j}\pb)^2}{M_\cp\pb}\right) = \dlog \left(\frac{\grid_{\cp,j}\pb}{M^{1/2}}\right),
\end{equation}
since $M_{\cp}=M_{\cp}\pb=M$. Using the additivity of $\dlog$, we get $\dlog A=L_{c,j}+L_{c,j}\pb$ and $\dlog B=L_{c-1,j}+L_{c+1,j}$. 
\end{proof}

\begin{proof}[Proof of~\ref{IS2} for~\ref{bm1}, solid-special]
We do not use the assumption that $(\BR_\br,\Sbr)$ is a cluster variety until the last paragraph of this proof. 
 Let $x:=x_{c+1}$ and $V:=V_{c+1}$. Applying \cref{prop:triang,prop:APS-gives-var-nonappearance}, we see that 
\begin{equation}\label{eq:ord_B1_special}
  \ord_V\grid_{c,j}=\ord_V\grid_{c,i}=1 \quad\text{and}\quad \ord_V\grid_{\cp,k}=0 \quad\text{for $(\cp,k)\in[0,m]\times(\pmn) \setminus \{(c,j),(c,i)\}$.}
\end{equation}
In particular, $\dlog x$ appears in $\ombr$ only in the terms $\Lci cj$ and $\Lci ci$ in 
 $\ombrx{c+1}=d_j\Lci cj\wedge \Lci{c+1}j$ and $\ombrc=-d_i\Lci {c-1}i\wedge\Lci ci$, respectively. Recall from~\eqref{eq:special} that we actually have $\Lci cj=\Lci ci$. Using this and $d_i=d_j$, we find
\begin{equation*}%
  \ombr - \ombrest = \ombrc + \ombrx{c+1} = d_j \Lci cj\wedge \left(\Lci{c+1}j + \Lci{c-1}i\right),
\end{equation*}
where $\ombrest := \sum_{\cp\in\Jo\setminus\{c,c+1\}} \ombrcp$ by \cref{cor:hollow}. Expanding the forms $\Lci{\cp}k$ in terms of $(\dlog x_{\cpp})_{\cpp\in\Jo}$ via~\eqref{eq:Lci-defn}, we see from~\eqref{eq:ord_B1_special} that $\dlog x$ appears in $\Lci cj$ with coefficient $1$ and that $\dlog x$ does not appear in $\ombrest$. Using~\eqref{eq:Lci_vs_grid}, we get
\begin{equation}\label{eq:dlog_x_coef_B1_special}
  \ombr - \ombrestp = d_j\dlog x\wedge (\Lci{c+1}j+\Lci{c-1}i)=  
  d_j\dlog x\wedge \left( \dlog(\grid_{c+1,j}\grid_{c-1,j}) - \dlog \prod_{k\neq j} \grid_{c,k}^{-a_{jk}}\right),
\end{equation}
where $\ombrestp$ is a linear combination of terms $\dlog x_{\cp}\wedge \dlog x_{\cpp}$ for $x_{\cp},x_{\cpp}\neq x$.
By \cref{prop:triang}, a cluster variable $x_{\cp}$ for $\cp\in\Jo$ may appear on the right-hand side of~\eqref{eq:dlog_x_coef_B1_special} only for $\cp\geq c$. Moreover, we have already observed that $d_{c}=d_{i}=d_j=d_{c+1}$. Let us denote $p_{\cp}:=\ord_{V_{\cp}} (\grid_{c+1,j}\grid_{c-1,j})$ and $q_{\cp}:=\ord_{V_{\cp}}\prod_{k\neq j} \grid_{c,k}^{-a_{jk}}$. Clearly, $p_{\cp},q_{\cp}\geq0$. 

Expanding $\ombr$ as a linear combination of terms $\dlog x_{\cp}\wedge \dlog x_{\cpp}$ for $\cp<\cpp$ (resp., $\cp>\cpp$) via~\eqref{eq:Lci-defn}, we see from~\eqref{eq:om_dfn_cluster} that for each $\cp\in\Jo\setminus\{c+1\}$, $d_j\tB_{\cp,c+1}$ equals the coefficient of $\dlog x_\cp\wedge \dlog x_{c+1}$, regardless of whether $\cp<c+1$ or $\cp>c+1$. Since $x=x_{c+1}$ does not appear in $\ombrestp$, we see from~\eqref{eq:dlog_x_coef_B1_special} that $d_j\tB_{\cp,c+1} =q_{\cp}-p_{\cp}$ for all $\cp\in\Jo\setminus\{c+1\}$. In fact, this identity also holds for $\cp=c+1$ since in this case $\tB_{\cp,c+1} = q_\cp = p_\cp = 0$. 

Combining~\eqref{eq:solid_spec_u_si_sj} and \cref{def:almostPos}, we see that $\vp{\cp}\apd{c+1} = \vp{\cp}$ for all $\cp\neq c$. In particular, by \cref{dfn:mutable}, the cluster variable $x$ is mutable. Thus, the mutated variable $x':=x'_{c+1}$ satisfies
\begin{equation}\label{eq:xx'_B1_special}
  x x' = \prod_{\cp\in\Jo: p_{\cp}>q_{\cp}} x_{\cp}^{p_{\cp}-q_{\cp}} + \prod_{\cp\in\Jo: q_{\cp}>p_{\cp}} x_{\cp}^{q_{\cp}-p_{\cp}}.
\end{equation}

We have $V_\cp=V_\cp\pb$ and $x_{\cp}=x_\cp\pb$ for all $\cp\in\Jo\setminus\{c+1\}$. Let $V\pb:=V_{c+1}\pb$ and $x\pb:=x_{c+1}\pb$. A generic point $(\Xbul,\Ybul)\in V$ satisfies $X_{c-1}\Lwrel{ws_j} Y_{c+1}$ and $X_{c+1}\Lwrel{w} Y_{c-1}$, while a generic point $(\Xbul,\Ybul)\in V\pb$ satisfies $X_{c+1}\Lwrel{ws_j} Y_{c-1}$ and $X_{c-1}\Lwrel{w} Y_{c+1}$. Thus, $V\neq V\pb$.

For $\cp\in\Jo$, applying $\ord_{V_\cp}$ to both sides of~\eqref{eq:M1identity}, we get
\begin{equation}\label{eq:ord_geq_B1_special}
  \ord_{V_\cp}\grid_{c, j} + \ord_{V_\cp} \grid_{c,j}\pb \geq \min(p_\cp,q_\cp).
\end{equation}
For $\cp=c+1$, we have $\ord_{V}\grid_{c, j}=1$, $\ord_{V} \grid_{c,j}\pb=0$ (since $V\neq V\pb$), and $p_{c+1}=q_{c+1}=0$ by~\eqref{eq:ord_B1_special}. Similarly, $\ord_{V\pb}\grid_{c,j}\pb=1$, $\ord_{V\pb}\grid_{c,j}=0$, and the order of vanishing of $\grid_{c+1,j}\pb\grid_{c-1,j}\pb$ and $\prod_{k\neq j} (\grid_{c,k}\pb)^{-a_{jk}}$ at $V\pb$ is zero. 

 Dividing both sides of~\eqref{eq:M1identity} by $\prod_{\cp\in\Jo\setminus\{c+1\}} x_{\cp}^{\min(p_\cp,q_\cp)}$, we get
\begin{equation*}%
  x x\pb \prod_{\cp\in\Jo\setminus\{c+1\}} x_{\cp}^{r_\cp} = \prod_{\cp\in\Jo: p_{\cp}>q_{\cp}} x_{\cp}^{p_{\cp}-q_{\cp}} + \prod_{\cp\in\Jo: q_{\cp}>p_{\cp}} x_{\cp}^{q_{\cp}-p_{\cp}},
\end{equation*}
where $r_{\cp}:=\ord_{V_\cp}\grid_{c, j} + \ord_{V_\cp} \grid_{c,j}\pb -\min(p_\cp,q_\cp)\geq0$. By~\eqref{eq:xx'_B1_special}, we get 
\begin{equation}\label{eq:B1_x'_vs_x_pb}
  x'=x\pb \prod_{\cp\in\Jo\setminus\{c+1\}} x_{\cp}^{r_\cp}.
\end{equation}

Now, assume that $(\BR_\br,\Sbr)$ is a cluster variety. We get from \cref{prop:irreducible} that the mutated cluster variable $x'$ is irreducible in $\C[\BR_\br]$. The function $x\pb$ vanishes on $V\pb\subset\BR_\br$ and therefore is not a unit in $\C[\BR_\br]$. It follows that $r_\cp=0$ for all mutable $\cp$, i.e., 
\begin{equation}\label{eq:ord=ord}
  \ord_{V_\cp}\grid_{c, j} + \ord_{V_\cp} \grid_{c,j}\pb = \min(p_\cp,q_\cp) \quad\text{for $\cp\in\Jmut$.}
\end{equation}
Thus, $x'$ and $x\pb$ differ by a monomial in the frozen variables: we have
\begin{equation}\label{eq:B1_x'_vs_x_pb2}
  x'=x\pb \prod_{\cp\in\Jfro} x_{\cp}^{r_\cp}.
\end{equation}
 We claim that the mutated seed $\Sigma':=\mu_{c+1}\Sigma_\br$ is quasi-equivalent to the pulled back seed $\Sigma_{\br'}\pb$. Recall that $\Sigma_{\br'}\pb=(T\pb,\x\pb,\d\pb,\om\pb)$ was defined in~\eqref{eq:Sbrp_dfn} while $\Sigma'= (T',\x',\d',\om')$ was defined in \cref{dfn:mut}. To show that these seeds are quasi-equivalent, we check each condition in \cref{dfn:quasi_eq}. We have $\d\pb = \d'$ since $d_i = d_j$. We have $\om'=\ombr$ by \cref{dfn:mut} and $\om\pb =  \ombr$ by~\ref{IS0}. The tori $T\pb = T'$ are both obtained as the subset of $\BR_\br$ where the cluster variables in $\{x_\cp\}_{\cp\in\Jo\setminus\{c+1\}}\cup\{x'\}$ are nonzero in view of~\eqref{eq:B1_x'_vs_x_pb2}. Thus, condition~\eqref{quasi1} in \cref{dfn:quasi_eq} is satisfied. The set of frozen variables has not changed, so condition~\eqref{quasi2} is satisfied trivially. Condition~\eqref{quasi3} is satisfied by~\eqref{eq:B1_x'_vs_x_pb2}.
\end{proof}

\subsection{Non-mutation move: \ref{bm1}, not solid-special}\label{sec:bm1_non_spec}
We continue to assume that the move involves indices $c,c+1$, and that $i\in -I$, $j\in I$.

\subsubsection{\ref{bm1}, special, non-solid}
Suppose that at least one of the indices is hollow, and that the move is special. Then it follows that $c+1$ is hollow and $c$ is solid in both $\br$ and $\br'$. By~\eqref{eq:L_mon_reln}, $\Lci{c}j=-\Lci{c+1}j$ and $\Lci{c}i\pb=-\Lci{c+1}i\pb$. Applying~\eqref{eq:special} with $\eps=1$ for $\cp=c-1,c$ and $\eps=-1$ for $\cp=c+1$ and using $d_{i}=d_j$, we obtain 
\begin{equation*}%
\frac{\ombrc}{d_j}
=-\Lci{c-1}i\wedge \Lci ci
=\Lci{c-1}i\wedge \Lci{c+1}i
=\Lci{c-1}i\pb\wedge\Lci{c+1}i\pb
=-\Lci{c-1}j\pb\wedge\Lci{c+1}j\pb
=\Lci{c-1}j\pb\wedge\Lci{c}j\pb
=\frac{\ombrpc\pb}{d_{i}},
\end{equation*}
which proves~\ref{IS0}. The clusters $\x_\br$ and $\x_{\br'}\pb$ are identical, which proves~\ref{IS2}.

\subsubsection{\ref{bm1}, non-special}
We start by introducing a formalism for working with the forms $\Lci\cp k$. Let $\la:=\sum_{k\in I} b_k \om_k$ with $b_k\in\Q$, and let $h$ be an $\H$-valued rational function on $\BR_\br$. We introduce a rational $1$-form
\begin{equation*}%
  \dlog h^\la:=\sum_{k\in I} b_k  \dlog (h^{\om_k}).
\end{equation*}
It is clear that 
\begin{equation}\label{eq:dlog_h_additive}
  \dlog h^{\la_1+\la_2}=\dlog h^{\la_1} + \dlog h^{\la_2} \quad\text{and}\quad
\dlog (h_1 h_2)^{\la}=\dlog h_1^\la + \dlog h_2^\la.
\end{equation}
For $\cp\in [0,m]$ and $k\in I$, \cref{lemma:hc} gives %
\begin{equation}\label{eq:Lci_vs_dlog_h}
  \Lci \cp k = \dlog (\hr_\cp)^{\alpha_k/2} = \dlog(\hb_\cp)^{\up\cp\alpha_k/2} ,\quad%
\Lci \cp{-k} = \dlog (\hb_\cp)^{\alpha_k/2} = \dlog (\hr_\cp)^{\up\cp^{-1}\alpha_k/2}.
\end{equation}
Finally, suppose that $h_1=h_2\ach_k(t)$. Then we have
\begin{equation}\label{eq:la_ach_scalar}
  \dlog h_1^\la=\dlog h_2^\la + \<\la,\ach_k\> \dlog t.
\end{equation}

\begin{proof}[Proof of~\ref{IS0} and \ref{IS2} for~\ref{bm1}, non-special]
Suppose as before that the move involves indices $c,c+1$, and that $i\in -I$, $j\in I$. Assume first that both $c,c+1$ are solid, and let $u:=\up{c-1}=\up{c}=\up{c+1}$. Let $a:=\<u^{-1}\alpha_{|i|}/2,\ach_j\>$ and $a':=\<u\alpha_j/2,\ach_{|i|}\>$. Using~\eqref{eq:Lci_vs_dlog_h}--\eqref{eq:la_ach_scalar} and~\eqref{eq:h_tpar}, we get
\begin{align}%
\label{eq:IS0_bm1_non_special_1}  \Lci ci&=\Lci{c+1}i+a\dlog\tpar_{c+1}, &\Lci{c-1}i&=\Lci ci+\dlog\tpar_c,  &\Lci cj&=\Lci{c+1}j+\dlog\tpar_{c+1};\\
\label{eq:IS0_bm1_non_special_2}  \Lci cj\pb &=\Lci{c+1}j\pb+a'\dlog\tpar_{c+1}\pb, &\Lci{c-1}j\pb&=\Lci cj\pb+\dlog\tpar_c\pb,  &\Lci ci\pb&=\Lci{c+1}i\pb+\dlog\tpar_{c+1}\pb.
\end{align}
Since the move is non-special, the coroots $\ach_j$ and $u^{-1}\ach_{|i|}$ are linearly independent, which implies $\tpar_c\pb=\tpar_{c+1}$ and $\tpar_{c+1}\pb=\tpar_c$. Note also that we have $\Lci{c+1}i\pb=\Lci{c+1}i$ and $\Lci{c+1}j\pb=\Lci{c+1}j$. Using~\eqref{eq:IS0_bm1_non_special_1}--\eqref{eq:IS0_bm1_non_special_2} to express each $1$-form $\Lci\cp k$ in terms of $\Lci{c+1}i$, $\Lci{c+1}j$, $\dlog\tpar_c$, and $\dlog\tpar_{c+1}$, we find
\begin{equation*}%
  \ombrc+\ombrx{c+1} - \ombrpc\pb - \ombrpx{c+1}\pb = (d_ja'-d_ia)\dlog\tpar_c\wedge\dlog\tpar_{c+1}.
\end{equation*}
Since $d_ja'=d_ia$, we get that $\ombr=\ombrp\pb$. The clusters $\x_\br$ and $\x_{\br'}\pb$ differ by a relabeling  $c\leftrightarrow c+1$.

Suppose now that one of $c,c+1$ is hollow. For instance, let $c\notin\Jo$ and $c+1\in\Jo$. By \cref{cor:h_stays}, we have $\hr_{c}=\hr_{c-1}$, and thus $\Lci cj=\Lci{c-1}j$. Similarly, $\Lci cj\pb=\Lci{c+1}j\pb$. Recall that $\Lci{c\pm 1}j\pb=\Lci{c\pm1}j$. Thus, $\ombrx{c+1}=\ombrpc\pb$, and so $\ombr=\ombrp\pb$. The case where $c\in\Jo$ and $c+1\notin\Jo$ is similar. The clusters $\x_\br$ and $\x_{\br'}\pb$ differ by a relabeling $c\leftrightarrow c+1$. For the case $c,c+1\notin\Jo$, see \cref{rmk:all_hollow}.
\end{proof}

\subsection{Non-mutation move: \ref{bm2}} Suppose that the move involves indices $c,c+1$. We have $\ombrc=\ombrpx{c+1}\pb$ and $\ombrx{c+1}=\ombrpc\pb$, so $\ombr=\ombrp\pb$. The chamber minors satisfy $  \grid_{c} = \grid_{c+1}\pb$ and $\grid_{c+1}=\grid_{c}\pb$.
Thus, the clusters $\x_\br$ and $\x_{\br'}\pb$ differ by a relabeling $c\leftrightarrow c+1$. This shows~\ref{IS0} and~\ref{IS2}.

\subsection{Mutation move: \ref{bm3}, solid, \short}\label{sec:B3_solid}
We proceed analogously to the case of solid-special \ref{bm1} in \cref{sec:B1_solid_special}. Suppose that the move $\br\to\br'$, $iji\to jij$, involves indices $c-1,c,c+1$, and that all three indices are solid. Suppose in addition that $i,j\in I$; the case $i,j\in -I$ is similar.

For \ref{IS0}, we will use the following relation among grid minors. 
\begin{proposition}[{\cite[Theorem~1.16(1)]{FZ_double}}]\label{prop:det_B3}
We have
  \begin{equation}\label{eq:M2identity}
    \grid_{c, i} \grid_{c, j}\pb = \grid_{c+1, i} \grid_{c-2, j} + \grid_{c-2, i}\grid_{c+1, j}.
  \end{equation}
\end{proposition}
\begin{proof}
We have $Z_{c-2}=Z_{c+1}z_{i}(t_1)z_{j}(t_2)z_{i}(t_3)$ for some $t_1,t_2,t_3\in\C$. We have  $z_{i}(t_1)z_{j}(t_2)z_{i}(t_3)=z_{j}( t_3)z_{i}(t_2')z_{j}(t_1 )$ for $t_2':=t_1 t_3 -t_2$, which can be checked inside $\SL_3$. Thus, $Z_{c-1}=Z_{c+1}z_{i}(t_1)z_{j}(t_2)$, $Z_c=Z_{c+1}z_{i}(t_1 )$, and $Z_c\pb=Z_{c+1}z_{j}(t_3)$. Let $Z:=Z_{c-2} (\ds_i\ds_j\ds_i)^{-1}$. Let $w:=\wp{c-1}=\wp c=\wp{c+1}$. By~\cite[Theorem~1.16(1)]{FZ_double},
\begin{equation}\label{eq:M2_FZ}
	\Delta_{w\om_i,s_i\om_i}(Z) \Delta_{w\om_j, s_j\om_j}(Z) = 
	\Delta_{w\om_i,\om_i}(Z) \Delta_{w\om_j, s_is_j\om_j}(Z) +
	\Delta_{w\om_i,s_js_i\om_i}(Z) \Delta_{w\om_j, \om_j}(Z).
\end{equation}
Similarly to the proof of \cref{prop:det_B1}, we observe that each term in~\eqref{eq:M2identity} equals the corresponding term in~\eqref{eq:M2_FZ}.
\end{proof}

\begin{proof}[Proof of~\ref{IS0} for~\ref{bm3}, solid, \short]
 Let $\tbeta = \beta \wo$ and $\tbeta' = \beta' \wo$.  By definition, $\sum_{j \in J_{\tbeta}} \omega_j(\tbeta)$ and $\sum_{j \in [m]} \omega_j(\br)$ are identical when expressed in terms of the symbols $\Delta_{c,i}$.  It is known (\cite[Proposition 3.25]{ShWe} or \cite{BFZ}) that \cref{prop:det_B3} implies $\omega_{\tbeta} = \omega_{\tbeta'}$.  Since the same identity for grid minors in \cref{prop:det_B3} holds on $\BR_\beta$, we deduce that $\omega_\br= \omega_{\br'}$.
\end{proof}

\begin{proof}[Proof of~\ref{IS2} for~\ref{bm3}, solid, \short]
 Let $x:=x_{c+1}$ and $V:=V_{c+1}$. By \cref{prop:triang,prop:APS-gives-var-nonappearance}, $x$ is mutable, $\ord_V\grid_{c-1,i}=\ord_V\grid_{c,i}=1$, and $\ord_V\grid_{\cp,k}=0$ when $\cp\notin\{c-1,c\}$ or $k\in I\setminus\{i\}$.

Note that $d_i=d_j$. Collecting the terms of $\ombrx{c-1}+\ombrx{c}+\ombrx{c+1}$ involving $\dlog x$, we get
\begin{equation}\label{eq:B3_dlog_x}
  d_i\dlog x\wedge \left(\Lci{c+1}i-\Lci{c-2}i + \frac12(\Lci{c-1}j-\Lci cj)\right).
\end{equation}
Applying~\eqref{eq:Lci-defn} and using \cref{cor:crossing_changes_one_minor}, we get
\begin{align*}%
  \Lci{c+1}i-\Lci{c-2}i&=\dlog(\grid_{c+1,i})-\dlog(\grid_{c-2,i})+\frac12\dlog(\grid_{c-2,j})-\frac12\dlog(\grid_{c+1,j});\\
  \Lci{c-1}j-\Lci cj &= \dlog(\grid_{c-1,j})-\dlog\grid_{c,j}= \dlog(\grid_{c-2,j})-\dlog\grid_{c+1,j}.
\end{align*}
Thus,~\eqref{eq:B3_dlog_x} becomes $ d_i\dlog x\wedge \left(\dlog(\grid_{c+1,i}\grid_{c-2,j}) - \dlog(\grid_{c-2,i}\grid_{c+1,j})\right)$.
The rest of the proof is entirely analogous to the argument for solid-special \ref{bm1} given at the end of \cref{sec:B1_solid_special}, using~\eqref{eq:M2identity} in place of~\eqref{eq:M1identity}. %
\end{proof}

\subsection{Non-mutation move: \ref{bm3}, non-solid, short}\label{sec:B3_non_solid}
Suppose that at least one of the indices $c-1,c,c+1$ is hollow. By \cref{rmk:all_hollow}, we may assume that there are either one or two hollow indices in $\{c-1,c,c+1\}$. Explicitly, underlining the hollow crossings, the possible moves are
$i\,\unj\,i\leftrightarrow j\,i\,\unj$ and 
 $i\,\underline{j}\,\underline{i}\leftrightarrow \underline{j}\,\underline{i}\,j$ (or the moves obtained from these by swapping the roles of $i$ and $j$).

For $l\in\{i,j\}$ and $\cp\in\Jo$, let us denote
\begin{equation}\label{eq:B3_non_solid_notation}
  A_l:=\dlog\prod_{k\neq i,j} \grid_{c+1,k}^{a_{lk}},\quad B_l:=\dlog\grid_{c+1,l}, \quad\text{and}\quad T_\cp:=\dlog\tpar_\cp.
\end{equation}
Using~\eqref{eq:crossing_changes_one_minor}--\eqref{eq:grid_mon_reln}, we can express the $\dlog$s of grid minors $\grid_{\cp,l}$ for $l\in\{i,j\}$ and $\cp\in\{c-1,c,c+1\}$ in the symbols~\eqref{eq:B3_non_solid_notation}. Using $\dlog \grid_{c-2,l}\pb=\dlog\grid_{c-2,l}$ for $l\in\{i,j\}$, we express $T_{\cp}\pb$ in terms of $T_{\cpp}$ for all indices $\cp\in\{c-1,c,c+1\}$ which are solid in $\br'$. 
Thus, we can express the forms $\ombrx{\cp}$, $\ombrpx{\cp}\pb$, $\cp\in\{c-1,c,c+1\}$ in terms of the symbols~\eqref{eq:B3_non_solid_notation}. Using a straightforward computation, we check $\ombr=\ombrp\pb$. 

We observe using \cref{cor:crossing_changes_one_minor} that the clusters $\x_\br$ and $\x_{\br'}\pb$ differ by a relabeling, which shows~\ref{IS2}.

\subsection{Non-mutation move: \ref{bm4}}\label{sec:bm4}
Suppose that $i\in I$. The isomorphism $\mis:\BR_\br\xrasim\BR_{\br'}$ sending $(\Xbul,\Ybul)\mapsto (\Xbul',\Ybul')$ is given by $X'_{m-1}=X'_m=Y'_m:=X_{m-1}$, $Y'_{m-1}:=Y_{m-1}$, and $(X'_c,Y'_c):=(X_c,Y_c)$ for all $0\leq c<m-1$. 
The last crossing in $\br_0i$ and $\br_0(-i^\ast)$ is always hollow, and thus the statements~\ref{IS0} and~\ref{IS2} follow trivially. 

\subsection{Non-mutation move: \ref{bm5}}\label{sec:bm5}
Suppose that $i\in I$. The isomorphism $\mis:\BR_\br\xrasim\BR_{\br'}$ sending $(\Xbul,\Ybul)\mapsto (\Xbul',\Ybul')$ is defined as follows. For $c\in[m]$, we set $(X'_c,Y'_c):=(X_c,Y_c)$ and $X'_0:=X'_1$. Note that $Y_0=Y_1=Y_1'$ and recall $X_0\Rwrel{\wo} Y_0$. 
We let $Y_0'$ be the unique weighted flag satisfying $X_0\Rwrel{\wo s_{i^\ast}} Y_0' \Rrel{s_{i^\ast}} Y_0$. It follows that $X'_0\Lwrel{\wo }Y'_0$ and $Y'_0\Rrel{s_{i^\ast}} Y'_1$, so $(\Xbul',\Ybul')\in\BR_{\br'}$.  The inverse map is defined similarly: $X_0$ is the unique weighted flag satisfying $X'_0\Rrel{s_i} X_0 \Rwrel{s_i\wo } Y'_0$.

The statement~\ref{IS0} is trivial if the first crossing of $\br$ is hollow. If the first crossing of $\br$ is solid, we have 
\begin{equation*}%
  X_0'=X_1'=X_1 \Rrel{s_i} X_0 \Rwrel{s_i\wo } Y_0'\Rrel{s_{i^\ast}} Y_0=Y_1=Y_1'.
\end{equation*}
It follows that after acting on all these flags by some $g\in G$, we can find $t,t'\in\C$ and $h\in\H$ such that
\begin{equation*}%
  X_0'=X_1'=X_1=\dwo \ds_ih \z_i(t) U_+,\quad X_0=\dwo \ds_ihU_+,\quad Y_0'=U_+,\quad Y_0=Y_1=Y_1'=z_{i^\ast}(t')U_+.
\end{equation*}
Here, we have $X_0 \Rwrel{s_i\wo } Y_0'$ and thus $X_0 \Lwrel{\wo s_i} Y_0'$, and we have used a representative $\dwo\ds_i$ of $\wo s_i$ in $N_G(\H)$. Let us denote $h_0:=\hr_0=\hb_0$ and $h_0\pb:=(\hr_0)\pb=(\hb_0)\pb$. We have $Z_0=Y_0^{-1}X_0=z_{i^\ast}(t')^{-1}\dwo \ds_ih$, and thus, proceeding as in the proof of \cref{lemma:tpar}, we get $h_0=h$. Similarly, $Z_0'=(Y_0')^{-1}X_0'= \dwo  \ds_i h \z_i(t)$, so $h_0\pb=s_{i} \cdot h$.
Applying~\eqref{eq:Lci_vs_dlog_h}, we find
\begin{equation*}%
  \Lci0{-i}\pb=\dlog(s_i\cdot h_0)^{\alpha_i/2}=\dlog(h_0)^{-\alpha_i/2}=- \Lci0i.
\end{equation*}
 Applying~\eqref{eq:special} for $\cp=0,1$, we obtain $\Lci0i=\Lci0{-i}$ and  $\Lci1i=\Lci1{-i}$. Recall that $\Lci1i=\Lci1i\pb$. Thus, we get $\ombrx1=\ombrpx1\pb$, and therefore $\ombr=\ombrp\pb$, finishing the proof of~\ref{IS0}.

We now prove~\ref{IS2}. Let $\br = i \br_0$ and $\br' = (-i) \br_0$. If the first crossing is hollow, the claim is trivial. Suppose that the first crossing is solid.  We have $\grid_{c,k}=\grid_{c,k}\pb$ for all $c\geq1$ and $k\in\pmn$. Thus, $x_c=x_c\pb$ for all $c\in\Jo$ such that $c>1$. Since $h_0\pb=s_i\cdot h_0$, \cref{lemma:hc} implies that $x_1\pb=x_1^{-1}M$, where $M$ is a Laurent monomial in the grid minors $\grid_{0,k}$ for $k\neq i$ of the same sign as $i$. It follows from \cref{prop:triang,prop:APS-gives-var-nonappearance} that $M$ is a Laurent monomial in the frozen variables other than $x_1$. This shows~\ref{IS2}.

\subsection{Deletion-contraction for double braid varieties} \label{sec:braid_dc}
\cref{thm:moves} tells us that it is enough to show $(\BR_{\br}, \Sbr)$ is a cluster variety for a single $\br$ in a double-braid-move equivalence class. We now explain how the 
 cluster algebraic results from \cref{sec:dc} apply to $(\BR_{\br}, \Sbr)$ for a special shoice of $\br$.%

\begin{theorem}\label{thm:br_del_con} 
Let $i \in I$ and consider a double braid word $\br= i i \br'$ on positive letters. If $(\BR_{i\br'}, \Sigma_{i\br'})$ and $(\BR_{\br'}, \Sigma_{\br'})$ are sink-recurrent cluster varieties and $\Sigma_{i\br'}$ is really full rank, then $(\BR_{\br}, \Sigma_{\br})$ is a sink-recurrent cluster variety and $\Sigma_{\br}$ is really full rank.
\end{theorem}
\begin{proof}
Suppose first that at least one of the first two crossings in $\br$ is hollow, in which case $1$ must be solid and $2$ must be hollow. 
Consider an arbitrary point $(\Xbul,\Ybul)\in\BR_\br$. Since the letters in $\br$ are positive, we have $Y_0=Y_1=\dots=Y_m=X_m$. 
Since $\wp2\leq\wo s_i$ and $\wp0=\wo$, we must have $X_1\Lwrel{\wo}X_m$ and $X_2\Lwrel{\wo s_i} X_m$. It follows that $\hrb_1$ and $\hrb_2$ are regular functions on $\BR_\br$.  Choose a representative $Z_2=\dwo \hb_2 \ds_i^{-1}$ as in~\eqref{eq:h_c_dfn}, and let $t,t'\in\C$ be such that $Z_1=Z_2z_i(t)$ and $Z_0=Z_2z_i(t)z_i(t')$. Thus, $t,t'$ are regular functions on $\BR_\br$. Proceeding as in the proof of \cref{lemma:tpar}, we find $\hr_1=\hb_2$ and $\hr_0=\hb_2\ach_i(t')$, where $\hr_0,\hr_1$ are regular on $\BR_\br$. It follows that $\grid_{0,i}=t'\grid_{1,i}$. For any $\cp\in\Jo$ such that $\cp>1$, the function $x_\cp$ depends on $Z_2,Z_3,\dots,Z_m$ but does not depend on $t,t'$. By \cref{prop:triang}, we have $\grid_{0,i}=x_1M$ for some monomial $M$ in $\{x_\cp\}_{\cp>1}$. The Deodhar hypersurface $V_1$ is clearly given by the equation $t'=0$. We conclude that $x_1=t'$. We thus have an isomorphism
\begin{equation}\label{eq:r_iso_dc}
  \ris:\BR_\br\xrasim \BR_{i\br'}\times \Cast,\quad (\Xbul,\Ybul)\mapsto\left((X_1,\dots,X_m,Y_1,\dots,Y_m),x_1\right).
\end{equation} 
Moreover, since $\grid_{1,i}=M$ involves only frozen variables, we see that $1$ is connected to only frozen indices in $\Digrt(\Sbr)$. It follows that the principal parts of $\Sbr$ and $\Sigma_{i\br'}$ agree, and therefore $(\BR_{\br}, \Sigma_{\br})$ is a sink-recurrent cluster variety. Moreover, since $\Sigma_{i\br'}$ is really full rank, so is $\Sbr$.

Suppose now that the first two crossings are both solid. %
Our goal is to apply \cref{thm:abstract-clust-stuc-del-con}. Recall from \cref{prop:smooth_affine_etc} that $\BR_\br$ is affine, factorial, and irreducible. We now show that $\Sigma_{\br}$ is sink-recurrent. Let $\Digr:=\Digr(\Sbr)$. The seed $\ifreeze{\Sigma_{\br}}[2]$ is obtained from $\Sigma_{i\br'}$ by adding an isolated frozen variable $x_1$, and $\ifreeze{\Sigma_{\br}}[\NDH]$ is obtained from $\Sigma_{\br'}$ by adding isolated frozen variables $x_1$ and $x_2$, so both of these seeds are sink-recurrent.

Next, the variable $x_1$ is frozen in $\Sigma_\br$. We claim that $\tB_{12}= 1$, and $\tB_{1c}=0$ for mutable $c>2$. Indeed, since the first two crossings are solid, we have $u\pd{0} = u\pd{1} = u\pd{2} = \id$. Moreover, since $i_1=i_2=i$, by \cref{def:almostPos}, we get that $\vp{0}\apd{c} = \vp{1}\apd{c}$ for all $c\in\Jo$ such that $c>2$. (For $c=2$, we have $\vp{0}\apd{2} = \id$ and $\vp{1}\apd{2}=s_i$.) Let $c>2$ be mutable so that $\vp{0}\apd{c} = \id$. Then $\vp{1}\apd{c}=\id$, and thus by \cref{prop:APS-gives-var-nonappearance}, we have $\ord_{V_c}\grid_{0,k} = \ord_{V_c}\grid_{1,k} = 0$ for all $k\in\pmn$. Thus, $\dlog x_c$ does not contribute to $\Lci 1i$, and so $\tB_{1c}= 0$. By \cref{prop:triang,prop:APS-gives-var-nonappearance}, we have $\ord_{V_1}\grid_{0,i} = \ord_{V_2}\grid_{1,i} = 1$, $\ord_{V_1}\grid_{0,j} = \ord_{V_2}\grid_{1,j} = 0$ for all $j\in I\setminus \{i\}$, and $\ord_{V_1}\grid_{1,j} = \ord_{V_2}\grid_{0,j}  = 0$ for all $j\in I$. Thus, $\tB_{12}= 1$.

Next, by \cref{cor:crossing_changes_one_minor}, the sum of terms of $\ombr$ involving $x_2$ is clearly of the form $d_i\dlog x_2\wedge\dlog M$ for a Laurent monomial $M$ in $\bx$, and thus $\Sbr$ is integral at $2$. We have shown that $\Sbr$ satisfies \cref{assn:dc}. 

Next, we show that the mutated cluster variable $x_2'$ is regular on $\BR_\br$. We apply the moves
$ii\br'\xrightarrow{\ref{bm5}}(-i)i\br' \xrightarrow{\ref{bm1}} i(-i)\br'$. 
 Denote $\dot\x:=\x_{(-i)i\br'}$ and $\ddot\x:=\x_{i(-i)\br'}$. It follows from the argument in \cref{sec:B1_solid_special} that $\dot x_2$ is mutable in $\Sigma_{(-i)i\br'}$, and by~\eqref{eq:B1_x'_vs_x_pb}, its mutation $\dot x_2'$ is regular on $\BR_{(-i)i\br'}$, as it equals the pullback $\ddot x_2\pb$ times a monomial in the other cluster variables in $\dot\x$ with nonnegative exponents.  As explained in \cref{sec:bm5}, the seeds $\Sbr$ and $\Sigma_{(-i)i\br'}$ are quasi-equivalent. By \cref{lem:quasimutate}, we find that the mutation $x_2'$ differs from $\dot x_2'$ by a unit (cf. part~\eqref{item:frozen_invertible} of \cref{cor:cluster_vars}), and thus $x_2'$ is regular on $\BR_\br$. 

Let $W := \{x_2 \neq 0\}$ $V := V_2= \{x_2 = 0\}$ be the open-closed covering of $\BR_{\br}$ coming from $x_2$. Our final goal is to construct isomorphisms
\begin{equation*}%
  W \cong \BR_{i\br'} \times \C^\times \cong \Vcal(\ifreeze{\Sigma_{\br}}[2]) \quad \text{and} \quad V \cong V_1\times V_2 = \Spec(\C[x_2']) \times  \BR_{\br'} \cong \C\times \Vcal(\icon{\Sigma_{\br}}[2])
\end{equation*}
satisfying the conditions of \cref{thm:abstract-clust-stuc-del-con}. %
Recall that by \cref{prop:APS-gives-var-nonappearance}, for $\cp\in\Jo$, $\cp>2$, we have $\ord_{V_\cp}\grid_{1,i}=0$ if and only if $\ord_{V_\cp}\grid_{0,i}=0$. Moreover, the same proposition implies $\ord_{V_2}\grid_{0,i}=0$. It follows by \cref{prop:triang} that $\grid_{1,i}$ is equal to $x_2$ times a monomial in the frozen variables, and that $\grid_{0,i}$ is equal to $x_1$ times a monomial in the same set of frozen variables. 
 Since $W$ is the complement of $V_2$, we see that $(\Xbul,\Ybul)\in W$ if and only if $X_1\Lwrel{\wo} Y_1=X_m$. Thus, $\hr_1$ is a regular function on $W$. We choose a representative $Z_1=\dwo \hr_1$ and  let $t\in\C$ be such that $Z_0=Z_1z_i(t')$. Then we get $t'=\grid_{0,i}/\grid_{1,i}=Mx_1/x_2$, where $M$ is a Laurent monomial in the frozen variables other than $x_1$. Similarly to~\eqref{eq:r_iso_dc}, we let $\risW:W \to \BR_{i\br'} \times \C^\times$ be the map sending $(\Xbul,\Ybul)$ to $\left((X_1,\dots,X_m,Y_1,\dots,Y_m),Mx_1/x_2\right)$. By assumption, we have $\BR_{i\br'}\cong \Vcal(\Sigma_{i\br'})$. The frozen index $1$ is only connected to other frozen indices in $\Digrt(\Sbr)$. Thus, the seed $\ifreeze{\Sigma_{\br}}[2]$ is obtained from $\Sigma_{i\br'}$ by adding an isolated frozen vertex, and therefore $\Vcal(\ifreeze{\Sigma_{\br}}[2])\cong \Vcal(\Sigma_{i\br'})\times \Cx$. Adjusting the isolated frozen variable by a Laurent monomial in the other frozen variables, we see that the pullbacks of $x_1,\dots, x_{\npm}$ under the inclusion $\Vcal(\ifreeze{\Sigma_{\br}}[2])\cong W\hookrightarrow X$ are indeed the same-named cluster variables in $\ifreeze{\Sigma_{\br}}[2]$. This verifies the assumptions on $\iota_W$ in \cref{thm:abstract-clust-stuc-del-con}.

 Now suppose that $(\Xbul,\Ybul)\in V$. We have $X_0\Lwrel{\wo} X_m$ but not $X_1\Lwrel{\wo} X_m$, so we must have $X_1\Lwrel{\wo s_i} X_m$, and therefore $X_2\Lwrel{\wo} X_m$. Consider the map $\risV:V\to \BR_{\br'}\times \C$ sending $(\Xbul,\Ybul)$ to $\left((X_2,\dots,X_m,Y_2,\dots,Y_m),x_2'\right)$. We claim that this map is an isomorphism. To construct an inverse, we need to show how to recover $X_0,X_1,Y_0,Y_1$ from the image of $\risV$. We have $Y_0=Y_1=X_m$. Also, $X_1$ is uniquely determined by $Y_1=Y_2$ and $X_2$, since $Y_1\Lwrel{s_i\wo} X_1 \Lrel{s_i} X_2$. It remains to recover $X_0$. Note also that we can recover the cluster variables $x_\cp$, $\cp>2$, as well as the mutated cluster variable $x'_2$, from the image of $\risV$. Since $(\Xbul,\Ybul)\in V$, the frozen variable $x_1$ is also recovered from the exchange relation $0=x_2x_2'=M_1+x_1M_2$ for $x_2$. 

In order to recover $X_1$, we apply moves $ii\br'\xrightarrow{\ref{bm5}}(-i)i\br' \xrightarrow{\ref{bm1}} i(-i)\br'$ as we did above. 
Let $(\ddot\Xbul,\ddot\Ybul)$ denote the image of $(\Xbul,\Ybul)$ in $\BR_{i(-i)\br'}$ under this isomorphism $\mis$, and let $\ddot\x:=\x_{i(-i)\br'}$.  As in \cref{sec:bm5}, let $Y_0'$ be the unique weighted flag satisfying $X_0\Rwrel{\wo s_{i^\ast}} Y_0' \Rrel{s_{i^\ast}} Y_0$.  Then $\ddot X_2=X_2$, and $\ddot Y_2=Y_2$, and $(\ddot X_1,\ddot X_0,\ddot Y_0,\ddot Y_1)=(X_2,X_1,Y_0',Y_0')$.  We have that $Y'_0$ is uniquely determined by $X_2$, $Y_2$, and $\ddot\x$: if $\ddot x_2=0$ then $Y_0'$ is uniquely determined by $Y_2\Lrel{s_{i^\ast}} Y_0'\Lwrel{\wo s_{i^\ast}} X_2$; otherwise, we have $X_2\Lwrel{\wo} Y_0'$, and the values of $\ddot\x$ uniquely fixes the $U_+\times U_+$-double coset $\ddot Z_1:=(Y_0')^{-1}X_2$ which determines $Y_0'$.  The weighted flag $X_0$ is then uniquely determined by $X_1\Rrel{s_i} X_0\Rwrel{s_i\wo}Y_0'$.  It thus suffices to show that $\ddot\x$ is uniquely determined by the image of $\risV$. For $\cp\in\Jo$, $\cp>2$, we have $\ddot x_\cp=x_\cp$. Moreover, $\ddot x_1=M/x_1$ for some monomial $M$ in the frozen variables $x_\cp$ other than $x_1$ (all of which must satisfy $\cp>2$ since $x_2$ is mutable). Finally, by~\eqref{eq:B1_x'_vs_x_pb}, $\ddot x_2$ differs from $x_2'$ by a monomial in the cluster variables other than $x_2$. We are done with verifying the assumptions on $\iota_V$ in \cref{thm:abstract-clust-stuc-del-con}.

Finally, we verify that $\ifreeze{\Sigma_{\br}}[2]$ is really full rank. This is immediate as $\Sigma_{i\br'}$ is really full-rank by assumption and we obtain $\ifreeze{\Sigma_{\br}}[2]$ by adding an isolated frozen.

We have verified all conditions in \cref{assn:dc} and \cref{thm:abstract-clust-stuc-del-con}. Thus, $(\BR_\br,\Sbr)$ is a cluster variety and $\Sbr$ is really full rank. We have already shown that it is sink-recurrent.
\end{proof}

\subsection{Proof of Theorem~\ref{thm:main} for $G$ simply-laced} \label{subsec:pf_thm_main_simply_laced}
We proceed by induction on the number $m$ of indices in $\br$, and prove the stronger statement that $(\BR_\br,\Sbr)$ is a cluster variety and $\Sbr$ is really full rank. Recall that we always assume $\Demprod(\br)=\wo$. The base case is $m=\ell(\wo)$, where all indices are hollow. The cluster algebra is $\Abr=\C$, which has no seeds and thus any seed of $\Abr$ is vacuously really full rank. The braid variety $\BR_\br$ is a point.

Suppose now that $m>\ell(\wo)$. Applying \ref{bm1} and \ref{bm4}, we can assume that all letters of $\br$ belong to $I$. Since $G$ is simply-laced, all braid moves are automatically short. Applying \ref{bm2}--\ref{bm3}, we may therefore transform $\br$ into a braid word of the form $\br_1ii\br_2$ for some braid words $\br_1,\br_2$ and $i\in I$. We can also apply \emph{conjugation moves} to $\br$: if $\br=j\br_0$, the conjugation move consists of the moves
\begin{equation}\label{eq:rotn}
  \br=j\br_0\xrightarrow{\text{\bmref{bm5}}} (-j)\br_0 \xrightarrow{\text{\bmref{bm1}}} \cdots \xrightarrow{\text{\bmref{bm1}}} \br_0(-j) \xrightarrow{\text{\bmref{bm4}}} \br_0 j^\ast.
\end{equation}
Applying conjugation moves, we may further transform $\br$ into the form $\br':=ii\br_2\br_1^\ast$, where $\br_1^\ast$ is obtained from $\br_1$ by applying the map $j\mapsto j^\ast$ to each letter. Utilizing the inductive hypothesis and applying \cref{thm:br_del_con} to $\br'$, we find that $(\BR_{\br'},\Sigma_{\br'})$ is a cluster variety and $\Sigma_{\br'}$ is really full-rank. It follows from \cref{thm:moves} (for short moves) and \cref{cor:quasi} that $(\BR_\br,\Sbr)$ is therefore also a cluster variety and $\Sbr$ is really full-rank. \qed

\begin{remark}\label{rmk:full_rank_simply_laced}
It follows from our proof that the seed $\Sbr$ is really full rank when $G$ is simply-laced.
\end{remark}

Finally, we show that for $G$ simply-laced, double braid moves correspond to mutation equivalence.
\begin{proposition}\label{prop:moves_are_mutations_simply_laced}
	Suppose that $G$ is simply-laced and $\br, \br'$ are related by a braid move \ref{bm1}--\ref{bm4}. The seeds $\Sbr, \Sbrp\pb$ are mutation equivalent (up to relabeling cluster variables).
\end{proposition}

\begin{proof} By \cref{thm:main} for simply-laced $G$, $(\BR_\br,\Sbr)$ is a cluster variety. By~\ref{IS2}, there is a seed $\Sigma'$, which differs from $\Sbr$ by mutation and possibly relabeling, such that $\Sigma' \sim \Sbrp\pb$. We claim that these seeds are actually identical. Indeed, choose a double braid word $\br_0$ such that all cluster variables of $\br,\br'$ become mutable in $\tbeta:=\br_0\br$, $\tbeta':=\br_0\br'$; cf. \cref{lem:extend-word}. Let $\tilde\Sigma'$ be obtained from $\Sigma_{\tbeta}$ using the same mutations and relabeling by which $\Sigma'$ was produced from $\Sbr$. Now, $(\BR_\tbeta,\Sigma_{\tbeta})$ is also a cluster variety, so by~\ref{IS2}, $\tilde\Sigma' \sim \Sigma_{\tbeta'}\pb$. Since all frozen variables of $\Sigma'$ are mutable in $\tilde\Sigma'$, it follows that the seeds $\Sigma'$ and $\Sbrp\pb$ are identical. 
\end{proof}

\section{Folding}\label{sec:folding}

Before completing the proof in \cref{sec:mult_laced}, we review some background on folding. We first compare Deodhar geometry (\cref{sec:deodhar-geometry}) in the case of a multiply-laced group $G$ to the case of the ``unfolded'' simply-laced group $\UG$. We then review folding on seeds.

\subsection{Pinnings}
Let $G$ be a complex, simple, simply-connected algebraic group. Choose a pinning
$(\H,B_+,B_-,x_i,y_i;i\in I)$. Then there exists an algebraic group $\UG$ of simply-laced type with pinning $(\UH,\UB_+,\UB_-,\ux_\ui,\uy_\ui;\ui\in\UI)$; see~\cite[\S1.6]{Lus2}. We have a bijection $\sig:\UI\to\UI$ which extends to an automorphism $\sig:\UG\to\UG$, and a map $\iota: G\to\UG$ which yields algebraic group isomorphisms
\begin{equation*}%
  \iota: G\xrasim \UG^\sig,\quad \H\xrasim \UH^\sig, \quad B_{\pm}\xrasim (\UB_{\pm})^\sig, \quad U_{\pm}\xrasim (\UU_{\pm})^\sig.
\end{equation*}
The maps $gB_+\mapsto \iota(g) \UB_+$ and $gU_+\mapsto \iota(g)\UU_+$ induce isomorphisms of %
varieties:
\begin{equation}\label{eq:iota_GB}
  \iota: G/B_+\xrasim (\UG/\UB_+)^\sig \quad\text{and}\quad G/U_+\xrasim (\UG/\UU_+)^\sig.
\end{equation}
For the first isomorphism, see~\cite[\S8.8]{Lus2}. 
 The surjectivity and injectivity of the second map follow from that of the first by a straightforward computation.

For an element $i\in I$, we denote by $\orb[i]\subset \UI$ the associated $\sigma$-orbit, i.e., the orbit under the cyclic group generated by $\sigma$. We also let $\orb[-i]:=\{-\ui\mid \ui\in\orb[i]\}\subset -\UI$. 
 We let $\{\ualpha_{\ui}\mid \ui\in\UI\}$, $\{\uach_{\ui}\mid \ui\in\UI\}$, and $\{\uom_\ui\mid\ui\in\UI\}$ be the simple roots, simple coroots, and fundamental weights of the root system of $\UG$. Letting $\ucartan_{\ui\uj}:=\<\ualpha_\ui,\uach_\uj\>$ be the entries of the associated Cartan matrix (and setting $\ucartan_{(-\ui)(-\uj)}:=\ucartan_{\ui\uj}$ and $\ucartan_{(\pm\ui)(\mp\uj)}:=0$ as before), we have 
\begin{equation}\label{eq:d_i=|orb|}
  d_i=|\orb[i]| \quad\text{and}\quad a_{ij}=\sum_{\uj\in\orb[j]} \ucartan_{\ui\uj} \quad\text{for all $i,j\in\pm I$ and  $\ui\in\orb[i]$.}
\end{equation}
 The Coxeter generators of the Weyl group $\UW$ of $\UG$ are denoted by $\{\us_\ui\mid\ui\in\UI\}$. Restricting $\iota$ to the normalizer of $\H$, we get a group isomorphism
\begin{equation}\label{eq:iota_W}
  \iota: W\xrasim \UW^\sig,\quad s_i\mapsto\prod_{\ui\in\orb[i]} \us_\ui.
\end{equation}
Here the order inside $\orb[i]$ is immaterial since the corresponding elements $\us_\ui$ commute. It follows that the longest element $\wo\in W$ gets mapped under~\eqref{eq:iota_W} to the longest element $\uwo$ of $\UW$, because $\sig:\UW\to\UW$ preserves Coxeter length and therefore $\uwo\in\UW^\sig$. The following result is immediate.

\begin{lemma}\label{lemma:iota_rel}
Let $B_1,B_2\in G/U_+$. If $B_1\xrightarrow{w} B_2$ then  $\iota(B_1)\xrightarrow{\iota(w)} \iota(B_2)$. If $B_1\Rwrel{w} B_2$ then $\iota(B_1)\Rwrel{\iota(w)} \iota(B_2)$. In particular, if $B_1\Rwrel{\wo} B_2$ then $\iota(B_1)\Rwrel{\uwo} \iota(B_2)$.
\end{lemma}

\subsection{Braid varieties}\label{sec:folding_braid}
Let $\br=i_1i_2\dots i_m\in\DRW$ be a double braid word. 
 Let $\bru=\uisub_1\uisub_2\dots\uisub_{\um}\in(\pm \UI)^{\um}$ be obtained by concatenating the letters in $\orb[i_1],\orb[i_2],\dots,\orb[i_m]$ (choosing the order inside each $\orb[i_c]$ arbitrarily), where $\um:=|\orb[i_1]|+|\orb[i_2]|+\dots+|\orb[i_m]|$. We let $\relab:[\um]\to[m]$ denote the unique order-preserving map satisfying $|\relab^{-1}(c)|=|\orb[i_c]|$ for all $c\in[m]$. It is clear that an index $c\in[m]$ is solid (resp., hollow) if and only if all indices in $\relab^{-1}(c)$ are solid (resp., hollow). In other words, the set $\UJo$ of solid crossings for $\bru$ is given by
\begin{equation}\label{eq:relab_J}
  \UJo=\relab^{-1}(\Jo).
\end{equation}

Let $\bigBR'_\br$ be the variety of tuples $(\UXbul,\UYbul)$ of weighted flags in $\UG/\UU_+$ satisfying
\begin{equation*}%
\begin{tikzcd}
		\UX_0& \arrow[l,"{\iota(s_{i_1}^+)}"'] \UX_1& \arrow[l,"{\iota(s_{i_2}^+)}"'] \cdots& \arrow[l,"{\iota(s_{i_m}^+)}"'] 
		\UX_m \\
		\UY_0 \arrow[r,"{\iota(s_{i_1^\ast}^-)}"'] \arrow[u, Rightarrow, "\uwo"']& \UY_1 \arrow[r,"{\iota(s_{i_2^\ast}^-)}"'] & \cdots \arrow[r,"{\iota(s_{i_m^\ast}^-)}"'] & \UY_m. \arrow[u, equal]
\end{tikzcd}
\end{equation*}
Let $\BigBR'_\br$ be obtained by omitting the condition $\UX_0\Lwrel{\uwo} \UY_0$. \Cref{lemma:rel-pos-facts} yields isomorphisms $\bigBR'_\br\cong \bigBR_\bru$ and $\BigBR'_\br\cong \BigBR_\bru$.
 Let $\BR'_\br$ be the quotient of $\bigBR'_\br$ by the free $\UG$-action. Then $\BR'_\br\cong \BR_\bru$. %

The map $\sig$ acts on the varieties $\bigBR'_\br$, $\BigBR'_\br$, and $\BR'_\br$ termwise by acting on each $\UX_c$ and $\UY_c$. Let $T'_\br\subset \BR'_\br$ be the image of the Deodhar torus $T_{\bru}\subset \BR_\bru$ under the isomorphism $\BR_\bru\cong \BR'_\br$. We have the following straightforward result.
\begin{proposition}
Applying $\iota$ termwise yields isomorphisms
\begin{equation}\label{eq:iota_BR}
  \bigBR_\br\xrasim (\bigBR'_\br)^\sig, \quad \BigBR_\br\xrasim (\BigBR'_\br)^\sig, \quad \BR_\br\xrasim (\BR'_\br)^\sig, \quad\text{and}\quad T_\br\xrasim (T'_\br)^\sig.
\end{equation}
\end{proposition}

\subsection{Grid minors}\label{sec:folding_grid}
Recall that we have the character and cocharacter lattices $\chL\H:=\Hom(\H,\Cast)$, $\cochL\H:=\Hom(\Cast,\H)$. The map $\iota:\H\to\UH$ induces a map $\iota_*:\cochL\H\to\cochL\UH$ sending $\ach_i\mapsto \sum_{\ui\in\orb[i]} \uach_{\ui}$ for $i\in I$, so that $\iota(\ach_i(t))=\prod_{\ui\in\orb[i]} \uach_{\ui}(t)$ for $t\in\Cast$. It also induces a map $\iota^\ast:\chL\UH\to\chL\H$ sending $\uom_\ui\mapsto \om_i$ for all $i\in I$ and  $\ui\in\orb[i]$, so that $\iota(h)^{\uom_\ui}=h^{\om_i}$ for $h\in\H$. It follows that for all $g\in G$, $v,w\in W$, $i\in I$, and $\ui\in\orb[i]$, we have
\begin{equation}\label{eq:iota_mnr}
  \Delta_{v\om_i,w\om_i}(g)=\Delta_{\iota(v) \uom_\ui,\iota(w)\uom_\ui} (\iota(g)).
\end{equation}

Let  $(\Xbul,\Ybul)\in\BR'_\br$. As usual, for $c=0,1,\dots,m$, we denote $Z_c:=Y_c^{-1}X_c$. Let $\uup c:=\iota(\up c)$ and $\wwp c:=\iota(\wp c)$. For $\ui\in \UI$,  consider analogs of grid minors for $\BR'_\br$:
\begin{equation}\label{eq:ugrid}
          \ugrid_{c, \ui}(\Xbul,\Ybul)= \Delta_{\wwp c\uom_\ui, \uom_\ui}(Z_c), \qquad   \ugrid_{c,-\ui}(\Xbul,\Ybul)=\Delta_{\uwo \uom_\ui, \uup c^{-1}\uom_\ui}(Z_c).
\end{equation}

Comparing~\eqref{eq:iota_mnr}--\eqref{eq:ugrid} to \cref{dfn:grid_minors}, we find that the grid minors on $\BR_\br$ are pullbacks of the minors defined in~\eqref{eq:ugrid}: for $c=0,1,\dots,m$, $i\in \pm I$, and $\ui\in\orb[i]$, we have
\begin{equation}\label{eq:iota_grid}
  \iota^\ast \ugrid_{c, \ui}= \grid_{c,i}.
\end{equation}
Using \cref{cor:crossing_changes_one_minor}, we obtain the following description of chamber minors on $\BR_\bru$, which we denote by $\ucrossing{\uc}$, $\uc\in\UJo$; cf.~\eqref{eq:relab_J}.
\begin{lemma}\label{lemma:iota_crossing}
  Let $\uc\in\UJo$ be a solid crossing for $\bru$. Set $\ui:=\uisub_\uc$ and $c:=\relab(\uc)\in J$. Then the isomorphism $\BR_\bru \cong \BR'_\br$ sends the chamber minor $\ucrossing{\uc}$ to $\ugrid_{c-1,\ui}$, and we have $\iota^\ast \ugrid_{c-1,\ui}=\crossing{c}$.
\end{lemma}

\subsection{$2$-form}
 Our next goal is to show that the two-form also folds.
\begin{lemma}\label{lem:om_folding}
Let $\omp_\br$ be the pullback of the $2$-form $\om_{\bru}$ on $\BR_{\bru}$ under the isomorphism $\BR'_\br\cong \BR_\bru$. We have $    \iota^\ast\omp_\br=\om_\br$.
\end{lemma}
\begin{proof}
Recall from~\eqref{eq:Lci-defn}--\eqref{eq:ombr_dfn} that we have $1$-forms $\Lci ci=\frac{1}{2} \sum_{k \in \pm I} a_{i k} \dlog \Delta_{c,k}$ on $\BR_\br$ for $(c,i)\in[m]\times (\pm I)$, and that for $c\in[m]$ and $i:=i_c$, we set $\om_c(\br):=\sign(i) d_i \Lci{c-1}i \wedge \Lci ci$. For $\ui\in\pm\UI$, introduce a $1$-form $\Lpci c\ui:=\frac{1}{2} \sum_{\uj \in \pm \UI} \ucartan_{\ui \uj} \dlog \Delta'_{c,\uj}$. By \cref{cor:crossing_changes_one_minor}, we have 
\begin{equation}\label{eq:folding_omp_br}
  \omp_\br=\sum_{c\in\Jo} \sign(i_c)\sum_{\ui\in\orb[i_c]} \Lpci{c-1}\ui\wedge\Lpci c\ui.
\end{equation}
Next, applying~\eqref{eq:iota_grid} and~\eqref{eq:d_i=|orb|}, we see that for all $c\in[m]$, $i\in \pmn$, and $\ui\in\orb[i]$, we have
\begin{equation}\label{eq:folding_Lci}
  \iota^\ast  \Lpci c\ui
= \iota^\ast \left(  \frac12 \sum_{\uj\in \pm\UI} \da_{\ui\uj} \dlog\ugrid_{c,\uj}\right)
=\frac12 \sum_{j\in \pmn} \left( \sum_{\uj\in \orb[j]} \da_{\ui\uj}\right) \dlog\grid_{c,j}
= \Lci ci.
\end{equation}
The result follows by combining~\eqref{eq:folding_omp_br}--\eqref{eq:folding_Lci} with~\eqref{eq:d_i=|orb|}.
\end{proof}

\subsection{Folding seeds} We briefly review the notion of folding seeds, following \cite[Section 4.4]{FWZ_book}, though translating into our conventions.

\begin{definition}\label{def:admiss}
	Let $\USigma=(\UT, \ux, \ud, \uom)$ be a seed with $\ud=(1, \dots, 1)$, with mutable indices $\UJmut$ and frozen indices $\UJfro$. Let $\sig$ be a bijection acting on $\UJ:=\UJmut\sqcup\UJfro$. Let $J$ be the set of $\sig$-orbits, and for $j\in J$, we denote the corresponding orbit by $\orb[j]$. 
 An orbit is \emph{mutable} (resp., \emph{frozen}) if it consists entirely of mutable (resp., frozen) indices. 
 The bijection $\sig$ also acts on the set of cluster variables by $\sig(\ux_{\uj})= \ux_{\sig(\uj)}$.
	We call $\USigma$ \emph{weakly $\sig$-admissible}\footnote{Our notion of weak $\Sheta$-admissibility differs from the notion of admissibility in \cite[Definition 4.4.1]{FWZ_book} in that we do not require that for any $\ua, \ua'$ in the same orbit and any $\uk$ mutable, $\utB_{\ua\uk}\tB_{\ua'\uk}\geq 0$.} if:
	\begin{enumerate}
		\item\label{admiss1} Every orbit is either mutable or frozen.
		\item\label{admiss2} The 2-form $\omega$ is invariant under the $\Sheta$-action.
		\item\label{admiss3} For all $\ua,\ua'\in\UJmut$ in the same $\sig$-orbit, $\utB_{\ua\ua'}=0$, where $\utB$ is the exchange matrix of $\USigma$.
	\end{enumerate}
\end{definition}
Part~\eqref{admiss1} implies a natural decomposition $J=\JJmut\sqcup \JJfro$. The map $\Sheta$ also acts on the torus $\UT$ by permuting coordinates. Notice that $\UT^{\Sheta}$ is isomorphic to $(\C^\times)^{|\orbSet|}$. We denote by $\iota:\UT^\sig\hookrightarrow \UT$ the inclusion map.

\begin{definition}\label{def:folded}
	Suppose $\USigma$ is weakly $\Sheta$-admissible, with notation as in \cref{def:admiss}.  
 The \emph{folded seed} is a seed with index set $J=\JJmut\sqcup \JJfro$, defined as $\iota^\ast\USigma=\fseed \Sigma:=(T, \x, \d, \om)$ where
	\begin{itemize}
        \item $T=\UT^\sig$;
		\item $\x=(x_{\orbit j})_{\orbit j \in \orbSet}$, where for $j\in J$, $x_{\orbit{j}}:=\iota^* \ux_\uj$ for any $\uj \in \orb[j]$;
		\item $d_{\orbit j}= |\orb[j]|$ for $j\in J$;
		\item $\om=\iota^* \uom.$
	\end{itemize}
\end{definition}

Note that $x_{\orbit j}$ is well-defined, since $\iota^* \ux_\uj=\iota^* \ux_{\sheta (\uj)}$ for all $\uj\in\orb[j]$. The exchange matrix $\fmtx \tB$ of $\fseed \Sigma$ is therefore written in terms of the exchange matrix $\utB$ of $\USigma$ as
\begin{equation}\label{eq:foldedB}
  \tB_{ab}=\sum_{\ua \in \orb[a]} \utB_{\ua \ub}, \quad\text{where $\ub\in\orb[b]$ is arbitrary}.
\end{equation}
In particular, if $\USigma$ is integral then so is $\Sigma$. For the rest of this subsection, we assume that $\USigma$ and $\Sigma$ are integral.

For weakly $\Sheta$-admissible $\USigma$ and $j\in\JJmut$, we denote by $\mu_{\orb[j]} \USigma=\mu_{\orb[j]} (\USigma)$ the result of mutating $\USigma$ once at each index in $\orb[j]$, and call $\mu_{\orb[j]}$ an \emph{orbit-mutation}. Note that $\mu_{\orb[j]} \USigma$ does not depend on the order of mutation. The seed $\mu_{\orb[j]} \USigma$ may not be weakly $\Sheta$-admissible; we introduce the following notion to avoid such mutations.

\begin{definition}\label{def:quasi-admiss}
	Let $\USigma$ be a weakly $\Sheta$-admissible seed and $j\in\JJmut$. 
 We call $\mu_{\orb[j]}$ \emph{quasi-admissible} if for all $k\in\JJmut$, we have $\tB_{\uk\uj} \tB_{\uk'\uj} \geq 0$ for all $\uk, \uk' \in \orb[k]$ and $\uj \in \orb[j]$.
\end{definition}

The name ``quasi-admissible'' is justified by the following proposition.

\begin{prop}\label{prop:quasi-admiss-mut}
	Let $\USigma$ be a weakly $\Sheta$-admissible seed and $j\in\JJmut$. If $\mu_{\orb[j]}$ is quasi-admissible, then $\mu_{\orb[j]}(\USigma)$ is weakly $\Sheta$-admissible and
\begin{equation*}%
  \mu_{\orbit j}(\iota^\ast \USigma)\quasi \iota_j^\ast (\mu_{\orb[j]} \USigma),
\end{equation*}
where $\iota_j$ 
 is an inclusion of the associated tori. 
\end{prop}
\begin{proof}
	We use the notation of \cref{def:admiss,def:folded}. In particular, let $\Sigma:=\iota^\ast\USigma$ be the folded seed on index set $J$. Since the exchange matrix $\utB$ of $\USigma$ is skew-symmetric, it is equivalent to a quiver $\UQ$; we will use the two interchangeably. 
	
It is clear that $\mu_{\orb[j]} \USigma$ satisfies condition~\eqref{admiss1} of \cref{def:admiss}. Because there are no arrows between vertices in $\orb[j]$, mutating at all vertices of $\orb[j]$ shows 
 that the exchange matrix of $\mu_{\orb[j]} \USigma$ satisfies $(\mu_{\orb[j]} \utB )_{\ua\ub}= (\mu_{\orb[j]} \utB )_{\sheta(\ua) \sheta(\ub)}$ for $\ua,\ub\in\UJ$. The assumption that $\mu_{\orb[j]}$ is quasi-admissible implies that for $k\in\JJmut$, $(\mu_{\orb[j]} \utB )_{\uk\uk'}=0$ for all $\uk,\uk'\in\orb[k]$. 
 Thus, $\mu_{\orb[j]} \USigma$ is weakly $\Sheta$-admissible.

	Let $\Sigma_1:=\mu_{\orbit j}(\Sigma)$ and $\Sigma_2:=\iota^\ast_j(\mu_{\orb[j]} \USigma)$. Let $\by$ and $\bz$ be the clusters in $\Sigma_1$ and $\Sigma_2$, respectively. For $\orbit k \neq \orbit j$, $y_{\orbit k}=z_{\orbit k}$ because both are equal to the cluster variable $x_k$ of $\Sigma$. 

	To analyze the relationship between $y_{\orbit j}$ and $z_{\orbit j}$, we need the following notions. Let $a,k\in J$ and  choose $\ua\in\orb[a]$, $\uk \in \orb[k]$. We call a path $\ua \to \uk \to \ua'$ in $\UQ$ a \emph{bad path} if $\ua, \ua'$ are in the same orbit; condition~\eqref{admiss2} of \cref{def:quasi-admiss} implies that no bad path in $\UQ$ begins or ends in a mutable orbit. Let $P_{\uk}$ be a maximal (by inclusion) collection of arrow-disjoint bad paths with middle vertex $\uk$.
	
	In $\USigma$, for $\uj\in\orb[j]$, the mutation $\ux_\uj'$ of $\ux_\uj$ is defined by the exchange relation
	\[\ux_\uj \ux_\uj'= M' N' + M'' N'', \quad\text{where}\]
\[N':= \prod_{(\ua \to \uj \to \ua')\in P_\uj} \ux_{\ua} \qquad \text{and} \qquad N'':= \prod_{(\ua \to \uj \to \ua')\in P_\uj} \ux_{\ua'},\]
 and $M', M''$ are the appropriate monomials in the cluster variables of $\USigma$. Notice that if $\ux_{\ua}$ appears in $M'$ and $\ux_{\ub}$ appears in $M''$ for $\ua\in\orb[a]$, $\ub\in\orb[b]$, then $\orb[a] \neq \orb[b]$, by the maximality of $P_\uj$. Notice also that by assumption, $N'$ and $N''$ are monomials in the frozen variables. We set $N:=\iota^*(N')= \iota^*(N'')$.  Using \eqref{eq:foldedB}, we have%
\begin{equation}\label{eq:z_orbit_j}
   \iota^* (\ux_\uj')= N  \frac{\iota^* M' + \iota^* M''}{x_j}, \qquad \text{and} \qquad \mu_j(x_{\orbit j})= \frac{\iota^* M' + \iota^* M''}{x_{\orbit j}}.
\end{equation}
 This shows that the tori and the lattices spanned by the frozens of $\Sigma_1, \Sigma_2$ agree, and that cluster variables differ by Laurent monomials in frozens. The multipliers $\d$ of both seeds are the same by definition.  The 2-forms of the two seeds agree by the functoriality of pullbacks. 
\end{proof}

\section{Proof of Theorem~\ref{thm:main} for $G$ multiply-laced}\label{sec:mult_laced}

Fix multiply-laced braid words $\br,\br'$ related by a long braid move \bmref{bm3}, so that
\[\br= \br_1\underbrace{iji\dots}_{\text{$m_{ij}$ letters}}\br_2= \br_1 \midbr \br_2 \quad\text{and}\quad \br'= \br_1\underbrace{jij\dots}_{\text{$m_{ij}$ letters}}\br_2= \br_1 \midbr' \br_2.\]
By \cref{dfn:mis_B1_B3}, we have an isomorphism 
$\mis: \BR_{\br} \xrightarrow{\sim} \BR_{\br'}.$ The goal of this section is to show \ref{IS0}, \ref{IS2}, and thus \cref{thm:moves}, for the geometrically defined seeds $\Sigma_{\br}$ and $\mis^*\Sigma_{\br'}$, and a particular mutation $\Sigma'$ of $\Sigma_{\br}$ (defined in \eqref{eq:Sig''-multiply-laced}). We will then show \cref{thm:main} for $G$ multiply-laced in \cref{subsec:pf_thm_main_mult_laced}, and discuss consequences of \cref{thm:main} in \cref{sec:lefschetz}

 Note that we already have shown \cref{thm:main,thm:moves} for simply-laced braid varieties.

\subsection{Proof of \ref{IS0} for long braid moves}
Let $\brl_1, \brl_2, \midbrl, \midbrl'$ be lifts of $\br_1, \br_2, \midbr, \midbr'$, respectively, following the conventions of \cref{sec:folding_braid}. Define $\brl:= \brl_1 \midbrl \brl_2$, $\brl':= \brl_1  \midbrl' \brl_2$ which are lifts of $\br$ and $\br'$, respectively. There is at least one sequence of short braid moves~\ref{bm3} relating $\brl$ and $\brl'$; fix such a sequence of braid moves and denote the corresponding isomorphism of braid varieties by $\ml: \BR_{\brl} \xrightarrow{\sim} \BR_{\brl'}$. 

Throughout this section, we abuse notation and use $\iota$ to denote the compositions
\[\iota: \BR_{\br} \xrightarrow{\sim} (\BR'_{\br} )^\sig \hookrightarrow \BR'_{\br} \xrightarrow{\sim} \BR_{\brl}  \qquad \text{and}\qquad \iota: \BR_{\br'} \xrightarrow{\sim} (\BR'_{\br'} )^\sig\hookrightarrow \BR'_{\br'} \xrightarrow{\sim} \BR_{\brl'}.\]

\begin{lemma}\label{lem:lifting_braid_iso}
	We have the equality $\iota \circ \mis = \ml \circ \iota$.
\end{lemma}

\begin{proof}
Observe that the words $\midbrl,\midbrl'$ are reduced. Thus, the sequence of moves~\ref{bm3} from $\midbrl$ to $\midbrl'$ fixes the weighted flags to the left and right of the indices involved in $\midbrl,\midbrl'$, and all weighted flags in between are uniquely determined; cf. \cref{dfn:mis_B1_B3}.
\end{proof}

\begin{lemma}\label{lemma:IS0_mult_laced}
	We have $\om_{\br}= \mis^*\om_{\br'}$; that is, \ref{IS0} holds for $\br, \br'$.
\end{lemma}
\begin{proof}
	We have $\om_{\br}=\iota^* \om_{\brl}= \iota^* \ml^* \om_{\brl'}= \mis^* \iota^* \om_{\brl'}= \mis^* \om_{\br'}$,
	where we have used \cref{lem:om_folding}, \ref{IS0} for $\brl, \brl'$, \cref{lem:lifting_braid_iso}, and \cref{lem:om_folding} again (in that order).
\end{proof}

\subsection{Proof of \ref{IS2} for long braid moves}
 We continue to use the notation established earlier in this section. Without loss of generality, we assume that either $\delta=ijij$ (in the case when $\alpha_i,\alpha_j$ form a root subsystem of type $B_2$ or $C_2$, where $|\orb[i]|=2$ and $|\orb[j]|=1$), or $\delta=121212$ (in the case of $G=G_2$, where $|\orb[1]|=3$ and $|\orb[2]|=1$).

The words $\midbr,\midbr'$ involve indices $r+1,\dots,r+p$, and for convenience, we decrease all indices by $r$ so that $\midbr,\midbr'$ are supported on $1,\dots,p$. 
 Similarly, we assume that $\midbrl,\midbrl'$ involve indices $1,\dots, \ptil$. 
 We define the seed 
\begin{equation} \label{eq:Sig''-multiply-laced}
	\Sigma':= \perm_{\fold} \circ \mu_{\fold}(\Sigma),
\end{equation}
 where $\perm_{\fold}$ is a permutation and $\mu_{\fold}$ is a sequence of mutations involving $1,\dots,p$. We list $\perm_{\fold},\mu_{\fold}$ in \tabref{tab:fold}(a--b). 
 In our tables, we only list the restriction of $\perm_{\fold}$ to the solid crossings in $1,\dots,p$. We would like to show that $\Sigma'$ and $\mis^* \Sigma_{\br'}$ are quasi-equivalent. To do so, we will eventually fold the seeds $\Sigma_{\brl}$ and $\Sigma_{\brl'}$, and then establish a chain of quasi-equivalences involving the folded seeds, $\Sigma'$ and $\mis^* \Sigma_{\br'}$.

As a first step, we fix a particular sequence $S$ of braid moves~\ref{bm3} between the lifts $\brl, \brl'$, and thus also fix $\ml$. By \cref{thm:main,thm:moves}, there is a corresponding mutation sequence $\mu_{\braid}$ and relabeling $\perm_{\braid}$ such that
\begin{equation}\label{eq:mu_braid-def}
	\perm_{\braid} \circ \mu_{\braid} (\USigma_{\brl})= \ml^* \USigma_{\brl'},
\end{equation}
where the seeds $\USigma_{\brl},\USigma_{\brl'}$ are the seeds denoted $\Sigma_{\brl},\Sigma_{\brl'}$ in \cref{sec:deodhar-geometry}. 
The sequence $S$ of braid moves is chosen so that $\mu_{\braid}$ and $\perm_{\braid}$ are as in \tabref{tab:fold}(c--d).%
We have the relabeling maps $\relab:[\mtil]\to[m]$ and $\relab':[\mtil]\to[m]$ as in \cref{sec:folding_braid}, where $\br,\br'$ (resp., $\brl,\brl'$) are on $m$ (resp., on $\mtil$) letters. 
By construction, we can extend the action of $\sig$ from $I$ to $\UJ_{\brl}$ and $\UJ_{\brl'}$. Specifically, for each letter $i_c$ in $\br$, $\sig$ permutes the letters $\orb[i_c]$ in the corresponding consecutive subword $\relab^{-1}(c)$ of $\brl$, and similarly for $\brl'$.

\begin{proposition}\label{prop:quasi-pullback-and-geo-seed} 
Let $C\subset \Jo$ be the set of indices $c$ such that none of $\uc\in\orb[c]$ is used in $\mu_{\braid}$, and let $\orb[C]:=\relab^{-1}(C)$. Let $C'\subset J_{\br'}$ and $\orbCp\subset\UJ_{\brl'}$ be defined similarly. 
 Then
\begin{equation*}%
  \iota^\ast{(\ifreeze{\USigma_{\brl}}[\orbC])} \quasi \ifreeze{\Sigma_{\br}}[C] \qquad \text{ and }\qquad \iota^\ast{(\ifreeze{\USigma_{\brl'}}[\orbCp])}  \quasi \ifreeze{\Sigma_{\br'}}[C'].
\end{equation*}
\end{proposition}
\begin{proof}
We focus on the first quasi-equivalence. By \cref{lemma:IS0_mult_laced}, it suffices to show that for each $\cp\in \Jo\setminus C$ and $\ucp\in\relab^{-1}(\cp)$, we have $\iota^\ast(\ux_\ucp)=x_\cp$. Let us fix such $\cp,\ucp$. 
 Choose also $c\in[0,p]$, $k\in I$, and $\uk\in\orb[k]$. It is enough to show the statement
\begin{equation}\label{eq:ord_lift}
  \ord_{V_\cp}\grid_{c,k}=\ord_{V'_{\ucp}} \ugrid_{c,\uk},
\end{equation}
where $V'_{\ucp}\subset \BigBR'_\br$ is the Deodhar hypersurface corresponding to $\ux_\ucp\in\C[\BR_{\brl}]\cong\C[\BR'_\br]$ and $\ugrid_{c,\uk}$ was defined in \cref{sec:folding_grid}.

We observe that the hollow crossings in $\midbr,\midbr'$ (and thus in $\midbrl,\midbrl'$) have a very special form: one of $\midbr,\midbr'$ has hollow crossings in positions $[r+1,p]$, while the other one has hollow crossings in positions $[r,p-1]$, for some $r$; cf. \tabref{tab:fold}(a--b). In this case, computing $\ord_{V_\cp}\grid_{c,k}$ is straightforward. First, suppose that $\cp\leq r$. Then all crossings in $[p]$ to the left of $\cp$ are solid. It follows from \cref{prop:triang,prop:APS-gives-var-nonappearance,cor:crossing_changes_one_minor} that for $c\in[0,p]$ and $k\in I$, we have $\ord_{V_\cp}\grid_{c,k}=1$ if $(c,k)\in\{(\cp-1,i_\cp),(\cp-2,i_\cp)\}$ and $\ord_{V_\cp}\grid_{c,k}=0$ otherwise. Applying the same argument to compute $\ord_{V'_{\ucp}} \ugrid_{c,\uk}$, we obtain~\eqref{eq:ord_lift}. It remains to consider the case $\cp=p$ when the hollow crossings are in positions $[r,p-1]$. The crossings $r-1$ and $r-2$ are solid, so \cref{prop:APS-gives-var-nonappearance} implies that $\ord_{V_\cp}\grid_{c,k}=0$ for $k=i_{r-1}$, $c<r-1$ or $k=i_{r-2}$, $c<r-2$. Here $\{i_{r-1},i_{r-2}\}=\{i,j\}$. For $c=p-1$, we have $\ord_{V_\cp}\grid_{c,k}=\<\om_k,\ach_{i_p}\>$ by \cref{prop:triang,prop:APS-gives-var-nonappearance}. Thus, for $c\in[r,p-1]$, \cref{lemma:tpar} implies that $\ord_{V_\cp}\grid_{c,k}=\<\om_k,s_{i_{c+1}}\cdots s_{i_{p-1}}\ach_{i_p}\>$. We have thus determined the values $\ord_{V_\cp}\grid_{c,k}$ for all $(c,k)\in[0,p]\times\{i,j\}$ except for $(c,k)=(r-1,i_{r-2})$. By \cref{cor:crossing_changes_one_minor}, we have $\ord_{V_\cp}\grid_{r-1,i_{r-2}}=\ord_{V_\cp}\grid_{r,i_{r-2}}$. It is clear that we have $\ord_{V_\cp}\grid_{c,k}=0$ for $k\in I\setminus\{i,j\}$. Computing $\ord_{V'_{\ucp}} \ugrid_{c,\uk}$ via a similar argument, we obtain~\eqref{eq:ord_lift}.
\end{proof}

 For the remainder of the section, let $C,\orbC,C',\orbCp$ 
 be as in \cref{prop:quasi-pullback-and-geo-seed}.

The sequence $\mu_{\braid}$ is ill-adapted to folding, so we find another mutation sequence $\mu_{\lift}$ relating $\USigma_{\brl}$ and a relabeling $\perm_{\lift}$ of $\ml^* \USigma_{\brl'}$. Explicitly, $\mu_{\lift}$ is a sequence of orbit-mutations lifting the sequence $\mu_{\fold}$ from \tabref{tab:fold}(a--b), and $\perm_{\lift}$ is given in \tabref{tab:fold}(e--f). 
Part~\eqref{mlift-same-as-mu_lift1} of the next result generalizes~\cite[Theorem~3.5]{FoGo_X}, which concerns the ``all solid" case. 

\begin{proposition} \label{prop:mlift-same-as-mu_lift}
	Let $\mu_{\lift}$ %
 and $\perm_{\lift}$ be as listed in \tabref{tab:fold}(e--f). Then 
	\begin{enumerate}
		\item\label{mlift-same-as-mu_lift1} $\perm_{\braid} \circ \mu_{\braid}(\USigma_{\brl})= \perm_{\lift} \circ \mu_{\lift}(\USigma_{\brl})$.
		\item\label{mlift-same-as-mu_lift2} $\mu_{\lift}$ is a sequence of quasi-admissible mutations of $\ifreeze{\USigma_{\brl}}[\orbC]$.
	\end{enumerate}
\end{proposition}

We delay the proof of \cref{prop:mlift-same-as-mu_lift} to the end of the section. \Cref{prop:quasi-admiss-mut} and part~\eqref{mlift-same-as-mu_lift2} of \cref{prop:mlift-same-as-mu_lift} together imply the following result.
\begin{corollary}\label{cor:quasi-pullback-mutations}
We have 
$\fseedx(\mu_{\lift}\ifreeze{\USigma_{\brl}}[\orbC]) \quasi \mu_{\fold}(\fseedx \ifreeze{\USigma_{\brl}}[\orbC]).$
\end{corollary}

\begin{proof}[Proof of \ref{IS2} for long braid moves]
	We have a string of quasi-equivalences:
	\[\fseedx(\mu_{\lift}\ifreeze{\USigma_{\brl}}[\orbC]) \quasi \mu_{\fold}(\fseedx \ifreeze{\USigma_{\brl}}[\orbC]) \quasi \mu_{\fold}\ifreeze{\Sigma_{\br}}[C], \]
	where the first quasi-equivalence is \cref{cor:quasi-pullback-mutations} and the second follows from \cref{prop:quasi-pullback-and-geo-seed} and \cref{lem:quasimutate}.  On the other hand,
	\[ \fseedx(\mu_{\lift}\ifreeze{\USigma_{\brl}}[\orbC]) 
= \fseedx( \perm_{\lift}^{-1} \circ \ml^* \ifreeze{\USigma_{\brl'}}[\orbCp] ) 
= \pi^{-1}_{\fold} \circ \fseedx(\ml^* \ifreeze{\USigma_{\brl'}}[\orbCp]) 
= \perm_{\fold}^{-1} \circ \mis^*(\fseedx \ifreeze{\USigma_{\brl'}}[\orbCp]) \quasi \perm_{\fold}^{-1} \circ \mis^*\ifreeze{\Sigma_{\br'}}[C'], \]
	where the first equality holds by \cref{prop:mlift-same-as-mu_lift} and \eqref{eq:mu_braid-def}, the second holds by direct computation (cf. \tabref{tab:fold}(c--f)), the third holds by \cref{lem:lifting_braid_iso}, and the final quasi-equivalence follows from \cref{prop:quasi-pullback-and-geo-seed} and the fact that $\mis^*$ preserves quasi-equivalence.
	
	Summarizing, we have that $\mu_{\fold} \ifreeze{\Sigma_{\br}}[C]$ is quasi-equivalent to (a relabeling of) $\mis^*\ifreeze{\Sigma_{\br'}}[C']$. %
	Notice that the cluster variables of $\mu_{\fold} \ifreeze{\Sigma_{\br}}[C]$, resp., $\mis^* \ifreeze{\Sigma_{\br'}}[C']$ are equal to the cluster variables of $\mu_{\fold}\Sigma_{\br}$, resp., $\mis^* \Sigma_{\br'}$. By assumption $(\BR_{\br}, \Sigma_{\br})$ is a cluster variety, so \cref{prop:irreducible} implies the cluster variables of $\mu_{\fold}\Sigma_{\br}$ are irreducible elements of $\C[\BR_{\br}]$. On the other hand, by \cref{cor:cluster_vars}, the cluster variables in $\mis^*\Sigma_{\br'}$ are irreducible elements of $\C[\BR_{\br}]$. Thus, the cluster variables in $\mu_{\fold}\Sigma_{\br}$ and $\mis^* \Sigma_{\br'}$ can differ only by units and $\mu_{\fold} \Sigma_{\br}$ is quasi-equivalent to (a relabeling of) $\mis^* \Sigma_{\br'}$.
\end{proof}

\def\holmidbrl{$[p]\setminus \UJ_{\brl}$}
\def\horskp{\\[0pt]}
\def\hskp{\\[-11pt]}
\setlength{\arraycolsep}{2pt}
\begin{table}
\scalebox{0.82}{
\vspace{-10in}
\begin{tabular}{cc}
$B_2/C_2$: & $G_2$:
\\
\begin{tabular}{c}
\begin{tabular}{| c | c | c| c | c| c|}

	\hline
	$\midbr \to \midbr'$ & $\mu_{\fold}$ & $\perm_{\fold}$  \\ \hline
	$i\,j\,i\,j \to j\,i\,j\,i\,$ & $\mu_{(4, 3, 4)}$ & \tiny{$ \begin{pmatrix}
	1 & 2 & 3& 4\\
	2 & 1 & 4 & 3
	\end{pmatrix}$}  \\
	\hline
		$i\,j\,\uni\,j \to j\,i\,j\,\uni$ & $\mu_{4}$ & \tiny{$ \begin{pmatrix}
			1 & 2 & 4\\
			2 & 1 & 3
		\end{pmatrix}$}\\
	\hline
		$i\,j\,i\,\unj \to j\,i\,\unj\, i$ & $\mu_{3}$ & \tiny{$ \begin{pmatrix}
			1 & 2 & 3\\
			2 & 1 & 4
		\end{pmatrix}$}  \\
	\hline
			$i\,j\,\uni\,\unj \to j\,\uni\,\unj\, i$ & $\id$ & \tiny{$ \begin{pmatrix}
			1 & 2\\
			4 & 1
		\end{pmatrix}$} \\
	\hline
			$i\,\unj\,\uni\,j \to j\,i\,\unj\,\uni$ & $\id$ & \tiny{$ \begin{pmatrix}
			1 & 4\\
			2 & 1
		\end{pmatrix}$} \\
	\hline
			$\uni\,\unj\,\uni\,j \to j\,\uni\,\unj\,\uni$ & $\id$ & \tiny{$ \begin{pmatrix}
			4\\
			1
		\end{pmatrix}$} \\
	\hline
			$i\,\unj\,\uni\,\unj \to \unj\,\uni\,\unj\, i$ & $\id$ & \tiny{$ \begin{pmatrix}
			1\\
			4
		\end{pmatrix}$} \\
	\hline
	\end{tabular}
\\
\hskp
(a) %
\end{tabular}
&
\begin{tabular}{c}
\begin{tabular}{| c | c| c|}
	\hline
 $\midbr \to \midbr'$ & $\mu_{\fold}$& $\perm_{\fold}$ \\ \hline
 $\one \two \one \two \one \twox \to \two \one \two \one \two \one$& $\mu_{(6, 3, 4, 6, 5, 6, 3, 4, 5, 6)}$& \tiny{$ \begin{pmatrix}
	1 & 2 & 3& 4 & 5& 6\\
	2 & 1 & 4 & 3& 6& 5
\end{pmatrix}$} \\
	\hline
$\one \two \one \two \one \utwox \to \two \one \two \one \utwo \one$& $\mu_{(3, 4, 5, 3, 4, 5)}$& \tiny{$ \begin{pmatrix}
			1 & 2 & 3& 4 & 5\\
			2 & 1 & 4 & 3& 6
		\end{pmatrix}$} \\
	\hline
 $\one \two \one \two \uone \twox \to \two \one \two \one \two \uone$& $\mu_{(6, 3, 4, 6, 3, 4)}$& \tiny{$ \begin{pmatrix}
			1 & 2 & 3& 4 & 6\\
			2 & 1 & 4 & 3& 5
		\end{pmatrix}$} \\
	\hline
 $\one \two \one \two \uone \utwox \to \two \one \two \uone \utwo \one$& $\mu_{(4 ,3, 4)}$& \tiny{$ \begin{pmatrix}
			1 & 2 & 3& 4 \\
			2 & 1 & 6 & 3
		\end{pmatrix}$} \\
	\hline
 $\one \two \one \utwo \uone \twox \to \two \one \two \one \utwo \uone$& $\mu_{(6 ,3, 6)}$& \tiny{$ \begin{pmatrix}
			1 & 2 & 3& 6 \\
			2 & 1 & 4 & 3
		\end{pmatrix}$} \\
	\hline
$\one \two \one \utwo \uone \utwox \to \two \one \utwo \uone \utwo \one$& $\mu_{3}$& \tiny{$ \begin{pmatrix}
			1 & 2 & 3 \\
			2 & 1 & 6
		\end{pmatrix}$} \\
	\hline
$\one \two \uone \utwo \uone \twox \to \two \one \two \uone \utwo \uone$& $\mu_{6}$& \tiny{$ \begin{pmatrix}
			1 & 2 & 6 \\
			2 & 1 & 3
		\end{pmatrix}$} \\
		\hline
$\one \two \uone \utwo \uone \utwox \to \two \uone \utwo \uone \utwo \one$& $\id$ & \tiny{$ \begin{pmatrix}
			1 & 2 \\
			6 & 1
		\end{pmatrix}$} \\
	\hline
$\one \utwo \uone \utwo \uone \twox \to \two \one \utwo \uone \utwo \uone$& $\id$  & \tiny{$ \begin{pmatrix}
			1 &  6 \\
			2 & 1
		\end{pmatrix}$} \\
	\hline
$\uone \utwo \uone \utwo \uone \twox \to \two \uone \utwo \uone \utwo \uone$& $\id$  & \tiny{$ \begin{pmatrix}
			6 \\
			1
		\end{pmatrix}$} \\
	\hline
$\one \utwo \uone \utwo \uone \utwox \to \utwo \uone \utwo \uone \utwo \one$& $\id$ & \tiny{$ \begin{pmatrix}
			1 \\
			6
		\end{pmatrix}$} \\
	\hline
	\end{tabular}
\\
\hskp
(b) %
\end{tabular}
\\
\horskp
\begin{tabular}{c}
	\begin{tabular}{|c|c|c|}
\hline
&&\negspc
 \holmidbrl & $\mu_{\braid}$ & $\perm_{\braid}$\\ \hline
 $\emptyset$ & $\mu_{(4, 5, 6, 4)}$ & \tiny{$\begin{pmatrix}
 		1&2& 3& 4& 5& 6\\
 		3&1& 2& 5& 4&6
 	\end{pmatrix}$}\\ \hline
   $\{4,5\}$ & $\mu_{6}$ & \tiny{$\begin{pmatrix}
 		1&2& 3& 6\\
 		3&1& 2& 4
 	\end{pmatrix}$}\\ \hline
  $\{6\}$ & $\mu_{(4, 5)}$ & \tiny{$\begin{pmatrix}
 		1&2& 3& 4& 5\\
 		3&1& 2& 6& 5
 	\end{pmatrix}$}\\ \hline
    $\{4,5,6\}$ & $\id$ & \tiny{$\begin{pmatrix}
 		1&2& 3\\
 		6&5& 1
 	\end{pmatrix}$}\\ \hline
     $\{3,4,5\}$ & $\id$ & \tiny{$\begin{pmatrix}
 		1&2& 6\\
 		2&3& 1
 	\end{pmatrix}$}\\ \hline
     $\{1,2,3,4,5\}$ & $\id$ & \tiny{$\begin{pmatrix}
 		6\\
 		1
 	\end{pmatrix}$}\\ \hline
       $\{3,4,5,6\}$ & $\id$ & \tiny{$\begin{pmatrix}
 		1 & 2\\
 		6 & 5
 	\end{pmatrix}$}\\ \hline
        \end{tabular}
\\
\hskp
(c) %
\end{tabular}
&
\begin{tabular}{c}
\begin{tabular}{|c|c|c|}  
\hline
&&\negspc%
 \holmidbrl & $\mu_{\braid}$ & $\perm_{\braid}$\\ \hline
$\emptyset$ & $\mu_{(8, 9, 5, 6,7, 8, 11, 10, 9, 5, 12, 6, 10, 8, 5, 11)}$ & \tiny{$\begin{pmatrix}
	1&2& 3& 4& 5& 6& 7& 8& 9& 10& 11 & 12\\
	2&3 & 4& 1& 9& 6&5& 7& 11& 12& 8& 10
\end{pmatrix}$}\\ \hline
$\{12\}$ & $\mu_{(8, 5, 6,7, 8, 11, 9, 5, 6, 10, 8, 5, 11)}$ & \tiny{$\begin{pmatrix}
		1&2& 3& 4& 5& 6& 7& 8& 9& 10& 11 \\
		2&3 & 4& 1& 12& 6&5& 7& 11& 10& 8
	\end{pmatrix}$}\\ \hline
$[9, 11]$ & $\mu_{(8, 5, 6,7, 8, 12, 5, 6, 8, 5)}$ & \tiny{$\begin{pmatrix}
		1&2& 3& 4& 5& 6& 7& 8& 12 \\
		2&3 & 4& 1& 9& 6&5& 7& 8
	\end{pmatrix}$}\\ \hline
$[9,12]$ & $\mu_{(6, 7, 5, 6, 8, 5)}$ & \tiny{$\begin{pmatrix}
		1&2& 3& 4& 5& 6& 7& 8\\
		3&4 & 2& 1& 11& 12&5& 10
	\end{pmatrix}$}\\ \hline
$[8,11]$ & $\mu_{(6, 5, 7, 12, 6, 5)}$ & \tiny{$\begin{pmatrix}
		1&2& 3& 4& 5& 6& 7&12\\
		2&3 & 4& 1& 6& 7&5& 8
	\end{pmatrix}$}\\ \hline
$[8, 12]$ & $\mu_{(5,6,7)}$ & \tiny{$\begin{pmatrix}
		1&2& 3& 4& 5& 6& 7\\
		2&3& 4& 1& 10& 11&12
	\end{pmatrix}$}\\ \hline
$[5,11]$ & $\mu_{12}$ & \tiny{$\begin{pmatrix}
		1&2& 3& 4& 12\\
		2&3 & 4& 1& 5
	\end{pmatrix}$}\\ \hline
$[5,12]$ & $\id$ & \tiny{$\begin{pmatrix}
		1&2& 3& 4\\
		10&11 &12 &1
	\end{pmatrix}$}\\ \hline
$[4,11]$ & $\id$ & \tiny{$\begin{pmatrix}
		1&2&3&12\\
		2&3&4& 1 
	\end{pmatrix}$}\\ \hline
$[2,12]$ & $\id$ & \tiny{$\begin{pmatrix}
		1\\
		12
	\end{pmatrix}$}\\ \hline
$[1,11]$ & $\id$ & \tiny{$\begin{pmatrix}
		12\\
		1
	\end{pmatrix}$}\\ \hline
	\end{tabular}
\\
\hskp
(d)
\end{tabular}
\\
\horskp
\begin{tabular}{c}
\begin{tabular}{|c|c|c|}\hline
&&\negspc
 \holmidbrl & $\mu_{\lift}$ & $\perm_{\lift}$\\ \hline
	$\emptyset$ & $\mu_{(6, 4, 5, 6)}$ & \tiny{$\begin{pmatrix}
			1&2& 3& 4& 5& 6\\
			2&3& 1& 6& 5&4
		\end{pmatrix}$}\\ \hline
        \end{tabular}
\\
\hskp
(e)
\end{tabular}
&
\begin{tabular}{c}
\scalebox{0.93}{
\begin{tabular}{|c|c|c|}  
\hline
          &&\negspc
 \holmidbrl & $\mu_{\lift}$ & $\perm_{\lift}$\\ \hline
	$\emptyset$ & $\mu_{(12, 5, 6, 7, 8, 12, 9, 10, 11, 12, 5, 6, 7, 8, 9, 10, 11, 12)}$ & \tiny{$\begin{pmatrix}
			1&2& 3& 4& 5& 6& 7& 8& 9& 10& 11 & 12\\
			2&3 & 4& 1& 6& 7& 8& 5& 10& 11& 12& 9
		\end{pmatrix}$}\\ \hline
	$\{12\}$ & $\mu_{(5, 6, 7, 8, 9, 10, 11, 5, 6, 7, 8, 9, 10 , 11)}$ & \tiny{$\begin{pmatrix}
			1&2& 3& 4& 5& 6& 7& 8& 9& 10& 11 \\
			2&3 & 4& 1& 6& 7&8& 5& 10& 11& 12
		\end{pmatrix}$}\\ \hline
	$[9, 11]$ & $\mu_{(12, 5, 6, 7, 8, 12, 5, 6, 7, 8)}$ & \tiny{$\begin{pmatrix}
			1&2& 3& 4& 5& 6& 7& 8& 12 \\
			2&3 & 4& 1& 6& 7&8& 5& 9
		\end{pmatrix}$}\\ \hline
	$[9,12]$ & $\mu_{(8, 5, 6, 7, 8)}$ & \tiny{$\begin{pmatrix}
			1&2& 3& 4& 5& 6& 7& 8\\
			2&3 & 4& 1& 10& 11&12& 5
		\end{pmatrix}$}\\ \hline
	$[8,11]$ & $\mu_{(5, 6, 7, 12, 5, 6, 7)}$ & \tiny{$\begin{pmatrix}
			1&2& 3& 4& 5& 6& 7&12\\
			2&3 & 4& 1& 6&7&8& 5
		\end{pmatrix}$}\\ \hline
\end{tabular}
}
\\
\hskp
(f)
\end{tabular}
\end{tabular}
}
\caption{\label{tab:fold}  The mutation sequences $\mu_{\fold}$, $\mu_{\braid}$, $\mu_{\lift}$, and the relabelings $\perm_{\fold}$, $\perm_{\braid}$, $\perm_{\lift}$ used in \cref{sec:mult_laced}. Hollow crossings are underlined. We denote $\mu_{(a_1, \dots, a_r)}:= \mu_{a_1} \circ \dots \circ \mu_{a_r}$. %
 For the case $B_2/C_2$, we denote $\orb[i]=\{\ui,\uii\}$, $\orb[j]=\{\uj\}$, $\midbrl= \ui\uii\uj\ui\uii\uj$,  $\midbrl'= \uj\ui\uii\uj\ui\uii$; for $G_2$, we denote $\orb[1]=\{\dot1,\dot3,\dot4\}$, $\orb[2]=\{\dot2\}$, $\midbrl= \dot1\dot3\dot4\dot2\dot1\dot3\dot4\dot2\dot1\dot3\dot4\dot2$, $\midbrl'= \dot2\dot1\dot3\dot4\dot2\dot1\dot3\dot4\dot2\dot1\dot3\dot4$. 
In (e) and (f), the cases where $\mu_{\lift}$ and $\mu_{\braid}$ coincide are omitted; we define $\perm_{\lift}:=\perm_{\braid}$ in those cases.}
\end{table}

\begin{table}
	\includegraphics[width=0.85\textwidth]{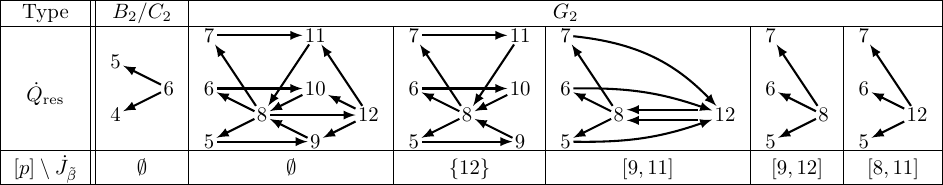}
	\caption{\label{fig:Qres} The quivers $\Qres$ from \cref{prop:mlift-same-as-mu_lift} listed in the same order as in  \tabref{tab:fold}(e--f). In the cases where $\mu_{\braid}= \mu_{\lift}$, $\Qres$ has no arrows.}
\end{table}

\begin{proof}[Proof of \cref{prop:mlift-same-as-mu_lift}] 
Recall that $\brl=\brl_1 \midbrl \brl_2$ and $\brl'=\brl_1 \midbrl'\brl_2$, and that we index the crossings of $\midbrl$ by $1, \dots, \ptil$. Let $J:=\Jo \setminus C$ be the set of indices which are mutated in $\mu_{\fold}$.

We show part~\eqref{mlift-same-as-mu_lift1}. 	
By \cite[Theorem 4]{GSV08}, a seed in $\A(\USigma_{\brl})$ is uniquely determined by its cluster, so we need only check~\eqref{mlift-same-as-mu_lift1}
at the level of cluster variables. This is easy to check for the cluster variables $\{x_c: c  \in C\}$ which are not touched by either mutation sequence.

Let $\USigma_{\brl}=(\ubx, \UQ)$, and let $\Qres$ be the induced subquiver of $\UQ$ on $\UJ:=\relab^{-1}(J)$. Let $\frQres$ be the \emph{framing} of $\Qres$; the extended exchange matrix of $\frQres$ is thus of size $2|\UJ| \times |\UJ|$ and the bottom $|\UJ| \times |\UJ|$ submatrix is the identity. We denote by $\tSigma_{\res}$ the seed $(\uby, \frQres)$ for some cluster $\uby$.  %
 By \cite[Theorem~3.7]{FZ4}, to show~\eqref{mlift-same-as-mu_lift1}, %
 it suffices to check that
	\begin{equation}\label{eq:braid-seed-equal-unfold-res}
	\perm_{\braid} \circ \mu_{\braid}(\tSigma_{\res})= \perm_{\lift} \circ \mu_{\lift}(\tSigma_{\res}).
\end{equation}
The relevant cluster variables in $\mu_{\braid}(\USigma_{\brl})$ and $\mu_{\fold}(\tSigma)$ can then be obtained from those in~\eqref{eq:braid-seed-equal-unfold-res} by specialization determined by $\UQ$.

To check~\eqref{eq:braid-seed-equal-unfold-res}, recall the description of the orders of vanishing of the cluster variables in $\UJ$ from the proof of \cref{prop:quasi-pullback-and-geo-seed}. This description only depends on which crossings in $\midbrl,\midbrl'$ are hollow, which in turn is determined by which crossings in $\brl_2$ are hollow. This implies that to compute $\Qres$, we may assume $\brl_2$ is a type $A_3$ braid word (in the $B_2/C_2$ case) or a type $D_4$ braid word (in the $G_2$ case) consisting entirely of hollow crossings. Applying the algorithm from \cref{sec:comb-algor} (to the simply-laced braids $\brl,\brl'$; cf. \cref{rmk:logic}), we get that $\Qres$ is as displayed in \cref{fig:Qres}. Equation~\eqref{eq:braid-seed-equal-unfold-res} may then be verified by computer.

Part~\eqref{mlift-same-as-mu_lift2} is also established by direct computation in $\Qres$.
\end{proof}
This completes the proof of \cref{thm:moves} for long braid moves.

\subsection{Finishing the proof}\label{subsec:pf_thm_main_mult_laced}
We now have shown \cref{thm:moves} for all braid moves. \cref{thm:main} for multiply-laced $G$ follows by the argument in \cref{subsec:pf_thm_main_simply_laced}. Repeating the proof of \cref{prop:moves_are_mutations_simply_laced}, we have the following.
\begin{proposition}
	Suppose $\br, \br'$ are related by a braid move \ref{bm1}--\ref{bm4}. The seeds $\Sbr, \Sbrp\pb$ are mutation equivalent (up to relabeling cluster variables).
\end{proposition}

 Continuing \cref{rmk:full_rank_simply_laced}, we obtain the following.
\begin{corollary}\label{cor:really_full_rank}
The seeds $\Sbr$ are really full rank for all $\br$.
\end{corollary}

\subsection{Curious Lefschetz property}\label{sec:lefschetz}
We now briefly discuss the consequences of our results for the cohomology $H^\ast(\BR_\br,\C)$; we refer to~\cite{LS} and~\cite[Section~10.1]{GLSBS1} for further details. Recall from \cref{prop:smooth_affine_etc} that $\BR_\br$ is a smooth, affine, complex algebraic variety of dimension $d:=d(\br)$. 
The mixed Hodge structure~\cite{Del71} of $H^k(\BR_\br,\C)$ endows the the cohomology group with a \emph{Deligne splitting} $H^k(\BR_\br,\C) = \bigoplus_{p,q}H^{k,(p,q)}(\BR_\br,\C)$. Since $\BR_\br$ is a sink-recurrent cluster variety of really full rank, it follows from~\cite[Theorem~8.3]{LS} that $H^\ast(\BR_\br,\C)$ is of \emph{mixed Tate type}, i.e., $H^{k,(p,q)}(\BR_\br,\C) = 0$ for $p\neq q$. The form $\ombr$, which coincides with the \emph{GSV form}~\cite{GSV} in view of~\eqref{eq:om_dfn_cluster}, defines an element $[\ombr]\in H^{2,(2,2)}(\BR_\br,\C)$. When the dimension $d$ of $\BR_\br$ is even, we say that $\BR_\br$ satisfies the \emph{curious Lefschetz property} with respect to $[\ombr]$ if the cup product induces isomorphisms
\begin{equation*}%
  [\ombr]^{d-p}: H^{p+s,(p,p)}(\BR_\br,\C) \xrasim H^{2d-p+s,(2d-p,2d-p)}(\BR_\br,\C)
\end{equation*}
for all $p$ and $s$. We also say that $\BR_\br$ satisfies (the weaker) \emph{curious \Poincare symmetry} if 
\begin{equation*}%
  \dim H^{p+s,(p,p)}(\BR_\br,\C) = \dim H^{2d-p+s,(2d-p,2d-p)}(\BR_\br,\C).
\end{equation*}

The following result is a consequence of the results of~\cite{LS}; see~\cite[Theorem~10.1]{GLSBS1}.\footnote{While~\cite{LS} work in the skew-symmetric setting, the curious Lefschetz theorem therein generalizes to the skew-symmetrizable case.}
\begin{theorem}\label{thm:cur_lef}
  Even-dimensional double braid varieties $\BR_\br$ satisfy the curious Lefschetz property with respect to $[\ombr]$ and thus also curious \Poincare symmetry. Odd-dimensional $\BR_{\br}$ satisfy curious \Poincare symmetry.
\end{theorem}
\noindent Note that if $d=d(\br)$ is odd, one can always add an isolated vertex to the seed $\Sigma_\br$, and the corresponding variety $\BR_\br\times\Cx$ will satisfy the curious Lefschetz property.

\section{Combinatorial algorithm}\label{sec:comb-algor}
\def\b{{\mathbf{b}}}
\def\c{{\mathbf{c}}}
\def\Xprime{W}
The exchange matrices for our seeds $\Sigma_\br$ are defined via Deodhar geometry. In particular, writing $\om_{\br}$ in terms of $\dlog x_e \wedge \dlog x_c$ requires knowing orders of vanishing along Deodhar hypersurfaces.  In this section, we give an algorithm to compute the order of vanishing of $\grid_c$ on $V_\cp$, which determines our cluster algebras. %

Let $c\in[0,m]$. The function $h^{\pm}_c$ of \cref{sec:torus-valued} is an $H$-valued character on $T_\br$.  We may thus write $h^{\pm}_c = \prod_{\cp \in \Jo} \gamma^{\pm}_{\br,c,\cp}(x_\cp)$, where $\gamma^{\pm}_{\br,c,\cp}$ are cocharacters of $H$ satisfying $\gamma^-_{\br,c,\cp} = u_c \cdot \gamma^+_{\br,c,\cp}$. By~\eqref{eq:hc} and~\eqref{eq:grid_vs_x_c}, for all $c\in[0,m]$ and $k\in I$, we have
\begin{equation}\label{eq:grid_vs_gamma}
  \grid_{c,\pm k} = (h^{\pm}_c)^{\om_k} = \prod_{\cp\in\Jo} x_\cp^{\<\om_k,\gamma^\pm_{\br,c,\cp}\>}, 
\quad\text{and so}\quad
\ordbr_{V_\cp}\grid_{c,\pm k} = \<\om_k,\gamma^\pm_{\br,c,\cp}\> 
\quad\text{for all $\cp\in\Jo$, $c\in[0,m]$, $k\in I$.}
\end{equation}
\begin{lemma}\label{lem:alg} Let $c \in [0, m]$ and let $e\in\Jo$ be such that $e \geq c$.
\begin{enumerate}
\item \label{lem:alg1}
Suppose $\beta'$ is obtained from $\beta$ by removing the first $c-1$ letters.  Then $\gamma^{\pm}_{\br,c,\cp} = \gamma^{\pm}_{\br',1,\cp-c+1}$.
\item\label{lem:alg2}
Suppose $\beta'$ is obtained from $\beta$ by doing 
non-mutation moves \bmref{bm1} and \bmref{bm4}, only involving indices greater than $c$,
 and let $\cp'$ be the image of $\cp$ under the resulting identification of cluster seeds.  Then we have $\gamma^{\pm}_{\br,c,\cp} = \gamma^{\pm}_{\br',c,\cp'}$.
\item\label{lem:alg3}
Suppose $\beta'$ is obtained from $\beta$ by removing solid crossings greater than $\cp$.  Then $\gamma^{\pm}_{\br,c,\cp} = \gamma^{\pm}_{\br',c,\cp}$.
\end{enumerate}
\end{lemma}
\begin{proof}
Part~\eqref{lem:alg1} follows from \cref{lem:extend-word}. 

For part~\eqref{lem:alg2}, as explained in \cref{sec:bm1_non_spec,sec:bm4}, applying non-mutation moves \bmref{bm1} and \bmref{bm4} results in a relabeling of the cluster variables. In particular, the pullbacks $V_{\cp'}\pb$ and $\grid_{c,\pm k}\pb$ coincide with $V_\cp$ and $\grid_{c,\pm k}$, respectively. Thus, by~\eqref{eq:grid_vs_gamma}, we get $\gamma^{\pm}_{\br,c,\cp} = \gamma^{\pm}_{\br',c,\cp'}$

  We prove~\eqref{lem:alg3}.  Suppose $\beta$ has a solid crossing $\cpp > \cp$ and assume that $\cpp$ is the largest solid crossing in $\br$. Similarly to \cref{prop:moves_are_mutations_simply_laced}, we may append a double braid word $\br_0$ to the left of $\br$ and $\brp$ so that $\cp\in\Jmut$; by part~\eqref{lem:alg1}, this does not affect $\gamma^{\pm}_{\br,c,\cp}$ or $\gamma^{\pm}_{\br',c,\cp'}$. Applying non-mutation moves \bmref{bm1} and \bmref{bm4} if necessary, we may assume that $i_\cpp\in I$; by part~\eqref{lem:alg2}, this preserves $\gamma^{\pm}_{\br,c,\cp}$ and $\gamma^{\pm}_{\br',c,\cp'}$. Since $\cpp$ is the largest solid crossing, we may apply moves~\ref{bm1}--\ref{bm3} involving hollow indices $\cpp+1,\dots,m$ until we have $ i_\cpp = i_{\cpp+1}$; these moves do not affect $\gamma^{\pm}_{\br,c,\cp}$ or $\gamma^{\pm}_{\br',c,\cp'}$ by \cref{rmk:all_hollow}. From now on, we assume $i_{\cpp} = i_{\cpp+1}\in I$.

  Let $\beta''$ be obtained from $\beta$ by removing the letter $i_\cpp$ from $\beta$.  Let $\Xprime \subset \BR_{\beta}$ be the open subset obtained by removing the Deodhar hypersurface $V_\cpp$, if $\cpp$ is mutable; otherwise, let $\Xprime := \BR_{\beta}$.  The projection $\pi:\Xprime \to \BR_{\beta''}$ given by forgetting the flags $(X_\cpp,Y_\cpp)$ is a fiber bundle with fiber $\Cast$.  We have $\pi^*(\grid_c^{\beta''}) = \grid_c^\beta$ and $\pi$ maps $V^\beta_\cp \cap \Xprime$ surjectively onto $V^{\beta''}_\cp$. (Here we use the superscript to refer to the braid variety on which $\grid_c$ and $V_\cp$ are defined.) Since both $V^\beta_\cp\subset \BR_\br$ and $V^{\beta''}_\cp\subset\BR_{\br''}$ are hypersurfaces,  it follows that the order of vanishing of $\grid_c^\beta$ on $V^\beta_\cp$ is equal to that of $\grid_c^{\beta''}$ on $V^{\beta''}_\cp$.  Repeating this argument, we obtain~\eqref{lem:alg3}.
\end{proof}
\def\rev{{\rm rev}}
Using \cref{lem:alg}, we may assume that $c = 1$, $\cp = m-1$ and $\beta =  (-\b^{*\rev}) \a k k$, where $\a,\b$ are words in $I$, and $(-\b^{*\rev})$ is obtained by reversing $\b$ and applying the map $i \mapsto -i^*$ to each letter, and $k = i_\cp = i_{m} \in I$.  Define 
\begin{equation*}%
  \gamma(\a,k,\b):=\gamma^+_{ (-\b^{*\rev})  \a k k,1,m-1},
\end{equation*}
and let $a$ and $b$ denote the Demazure product of $\a$ and $\b$, respectively.  Recall also from \cref{sec:dist_subexpr} that $*$ denotes Demazure product.

\begin{proposition}\label{prop:alg}
Suppose that \cref{thm:main,thm:moves} hold for $G$. Then the cocharacter $\gamma(\a,k,\b)$ satisfies, and is recursively defined by the following properties.

\begin{enumerate}[label=\Roman*),leftmargin=25pt]
\item\label{algI} We have $\gamma(\a,k,\b) = 0$ if $a = \wo$ or $b= \wo$.
\item\label{algII} The cocharacter $\gamma(\a,k,\b) =\gamma(a,k,b)$ only depends on the Demazure products $a,b$.
\item\label{algIII} We have $\gamma(\a,k,\emptyset) = a \alpha^\vee_k$ if $as_k>a$ and $\gamma(\a,k,\emptyset) =0$ if $as_k<a$.
\item\label{algIV} Suppose that $\a,\b$ are reduced.  We let $\a' = i\a$ and $\b = \b'j$, where the Demazure products satisfy $a' = s_i a > a$ and $b= b's_j > b'$.
\begin{enumerate}[label=(\arabic*)]
\item\label{alg1} If $a'*s_k*b > a *s_k *b$, then $\gamma(\a,k,\b) = s_i \cdot \gamma(\a', k, \b)$.
\item\label{alg2} If $a*s_k*b > a *s_k *b'$, then $\gamma(\a,k,\b) = \gamma( \a, k, \b')$.
\item\label{alg3} If $w:=a'*s_k*b = a *s_k *b = a *s_k *b'$, write $\alpha^\vee= \alpha^\vee_i$ and $\beta^\vee = -w^{-1} \cdot \alpha^\vee_j$.
\begin{enumerate}[label=(3\alph*)]
\item\label{alg3a} Suppose that $\alpha^\vee \neq \beta^\vee$.  Then $\gamma(\a',k,\b) = \gamma(\a,k,\b') + x\alpha^\vee+y\beta^\vee$ for $x,y \in \Z$, and we have $\gamma(\a,k,\b) =  \gamma(\a,k,\b') + y\beta^\vee$.
\item\label{alg3b} Suppose that $\alpha^\vee = \alpha^\vee_i = \beta^\vee$.  Then $\gamma(\a,k,\b) - \gamma(\a',k,\b') \in \Z\alpha_i^\vee$,  and
$$
\langle \omega_i,\gamma(\a,k,\b)\rangle = -\langle \omega_i,\gamma(\a',k,\b') \rangle+ \min\left(\langle \omega_i,\gamma(\a',k,\b)\rangle +\langle \omega_i,\gamma(\a,k,\b')\rangle, - \sum_{l \neq i} a_{il} \langle \omega_l,\gamma(\a',k,\b')\rangle\right).
$$
\end{enumerate}
\end{enumerate}
\end{enumerate}
\end{proposition}

\begin{proof}
We first argue that the stated properties determine $\gamma(\a,k,\b)$.  By~\ref{algI} and~\ref{algIII}, we know $\gamma(\a,k,\b)$ when $a = \wo$ or $b = \id$.  If $a \neq \wo$ and $b \neq\id$, property~\ref{algIV} allows us to express $\gamma(\a,k,\b)$ in terms of $\gamma(\a',k,\b),\gamma(\a,k,\b'),\gamma(\a',k,\b')$ where $a' > a$ and $b'<b$.  Thus, all values of $\gamma(\a,k,\b)$ are determined.  We now prove~\ref{algI}--\ref{algIV}.

Suppose that  $a = \wo$ or $b= \wo$.  Then a generic point $(X_\bullet,Y_\bullet)$ in $V_\cp$ satisfies $Y_0 \Rwrel{\wo} X_0$.  It follows that $\grid_1$ does not vanish on $V_\cp$, establishing~\ref{algI}.  

We show~\ref{algII}. It is clear that if $as_k<a$ then $\gamma(\a,k,\b)=0$. Suppose that $as_k>a$. We apply the moves $\br\xrightarrow{\ref{bm4}} (-\b^{*\rev}) \a k(-k^\ast)\xrightarrow{\ref{bm1}}(-\b^{*\rev}) \a (-k^\ast)k \xrightarrow{\ref{bm1}}\cdots\xrightarrow{\ref{bm1}} (-\b^{\ast\rev}) (-k^\ast)\a k$. Since $as_k>a$, these are non-mutation moves, and thus part~\eqref{lem:alg2} of \cref{lem:alg} applies. We may now remove the solid crossings from $\a$ using part~\eqref{lem:alg3} of \cref{lem:alg}, and then reverse the procedure to put $\br$ back into its original form with solid crossings removed from $\a$. 

We prove~\ref{algIII}. If $as_k<a$, we have already shown that $\gamma(\a,k,\emptyset)=0$. Assume $as_k>a$. When $a = \id$, the result follows from \cref{prop:triang}. 
  For $a \neq \id$, we apply induction,~\ref{algII}, and the hollow case of \cref{lemma:tpar}.

We prove~\ref{algIV}.  For~\ref{alg1}, adding the letter $i$ in front of  $(-\b^{*\rev}) \a  k k$ produces a new hollow crossing, and the claim follows from Lemmas~\ref{lemma:tpar} and~\ref{lem:alg}.  Similarly, for~\ref{alg2}, the letter $-j^\ast$ is a hollow crossing in $(-\b^{*\rev}) \a  k k$.  Case~\ref{alg3} holds if both $i$ and $-j^\ast$ are solid crossings in the word $i(-j^\ast)(-(\b')^{*\rev}) \a  k k$.  If swapping the order of $i$ and $-j^\ast$ is a non-mutation move then we are in Case~\ref{alg3a}, and the claim follows from \cref{lemma:tpar} and the linear independence of $\alpha^\vee$ and $\beta^\vee$.  If  swapping the order of $i$ and $-j^\ast$ is a mutation, then we are in Case~\ref{alg3b}, and the claim follows 
from~\eqref{eq:ord=ord}
 and the assumption that \cref{thm:main,thm:moves} have been shown for $G$.
\end{proof}

The algorithm has been implemented at \cite{GalTut}, where some examples can be found.

\begin{remark}\label{rmk:logic}
The logical dependencies in our proof are summarized as follows. In \cref{sec:2braid}, we give a complete proof of \cref{thm:main,thm:moves} for the case when $G$ is simply-laced. Thus, \cref{prop:alg} applies in this case. The proof for the case when $G$ is multiply-laced is given in \cref{sec:mult_laced}; it depends on \cref{prop:alg}, but only invokes it for the simply-laced group~$\UG$.
\end{remark}

The following result follows from our algorithm, but we have been unable to show it directly from Deodhar geometry.
\begin{corollary}
Let $\iota: \BR_{\br} \hookrightarrow \BR_{\brl}$ and $\relab:[\tilde m] \to [m]$ be as in \cref{sec:folding_braid,sec:mult_laced}.  Then for each $\cp\in \Jo$ and $\ucp\in\relab^{-1}(\cp)$, we have $\iota^\ast(\ux_\ucp)=x_\cp$. 
\end{corollary}
\begin{proof}
With notation as in the proof of \cref{prop:quasi-pullback-and-geo-seed}, it suffices to show that $\ord_{V_\cp}\grid_{c,k}=\ord_{V'_{\ucp}} \ugrid_{c,\uk}$ for $c \leq e$, $k\in I$, and $\uk\in\orb[k]$.  This follows from applying \cref{prop:alg} to $\BR_\br$ and $\BR_{\brl}$ separately.
\end{proof}

\bibliographystyle{alpha}
\bibliography{clustergen}

\end{document}